\documentclass[10pt]{amsart}
\usepackage[mathscr]{euscript}
\usepackage[utf8]{inputenc}

\usepackage{enumitem} 
\usepackage{tikz} 
\usepackage{tikz-cd} 
\usetikzlibrary{arrows,calc,decorations.markings}
\usepackage{thmtools,thm-restate}
\usepackage{hyperref} 
\usepackage[capitalise,nosort]{cleveref} 
\usepackage{amsmath}
\usepackage{amssymb}
\usepackage{bbm}
\usepackage{mathtools}
\usepackage[margin=1.3in]{geometry}
\usepackage{stmaryrd} 
\usepackage{xparse}
\usepackage{blindtext}
\usepackage{cite}

\makeatletter
\newcommand{\@bbify}[1]{
  \ifcsname b#1\endcsname
  \message{WARNING: Overwriting b#1 with blackboard letter!}
  \fi
  \expandafter\edef\csname b#1\endcsname
  {\noexpand\ensuremath{\noexpand\mathbb #1}\noexpand\xspace}}
\newcommand{\@calify}[1]{
  \ifcsname c#1\endcsname
  \message{WARNING: Overwriting c#1 with calligraphic letter!}
  \fi 
  \expandafter\edef\csname c#1\endcsname
  {\noexpand\ensuremath{\noexpand\mathcal #1}\noexpand\xspace}}
\newcommand{\@bfify}[1]{
  \ifcsname bf#1\endcsname
  \message{WARNING: Overwriting c#1 with bold letter!}
  \fi
  \expandafter\edef\csname bf#1\endcsname
  {\noexpand\ensuremath{\noexpand\mathbf #1}\noexpand\xspace}}
\newcounter{@letter}\stepcounter{@letter}
\loop\@bbify{\Alph{@letter}}\@calify{\Alph{@letter}}\@bfify{\Alph{@letter}}
\ifnum\the@letter<26\stepcounter{@letter}\repeat
\makeatother

\newenvironment{tz}{\begin{center}\begin{tikzpicture}[scale=1]}{\end{tikzpicture}\end{center}}

\tikzstyle{d}=[double distance=.3ex]
\tikzstyle{w}=[preaction={draw=white, -,line width=4pt}]

\tikzset{over/.style={auto=false,fill=white,inner sep=1.5pt, minimum size=0, outer sep=0}, 
pro/.style={postaction={decorate,decoration={
        markings,
        mark=at position .5 with {\node at (0,0) {$\bullet$};}
      }},
      inner sep=1ex,
      },n/.style={double equal sign distance, -implies},t/.style={double distance=2.5pt, -implies, postaction={draw,-}},
  }
\tikzset{%
node distance=1.5cm, la/.style={scale=0.8}, rr/.style={xshift=1.5cm},
space/.style={xshift=.5cm},
    symbol/.style={%
        draw=none,
        every to/.append style={%
            edge node={node [sloped, allow upside down, auto=false]{$#1$}}},
            
    }
}
  
\def\cellslide{0.5}
\def\celllength{.2cm}

\NewDocumentCommand{\cell}{ O{} O{n} O{\cellslide} O{\celllength} m m m }{
  \coordinate (mid) at ($({#5})!{#3}!({#6})$);
  \coordinate (start) at ($(mid)!{#4}!({#5})$);
  \coordinate (end) at ($(mid)!{#4}!({#6})$);
  \draw[#2] (start) to node
  [inner sep=6pt,outer sep=0,minimum size=0,#1]{{#7}} (end);
}

\newtheorem{thm}{Theorem}[subsection] 
\newtheorem*{thm*}{Theorem}
\newtheorem{cor}[thm]{Corollary}
\newtheorem{lemma}[thm]{Lemma}
\newtheorem{prop}[thm]{Proposition}

\theoremstyle{definition}
\newtheorem{defn}[thm]{Definition}
\newtheorem{notation}[thm]{Notation}
\newtheorem{constr}[thm]{Construction}

\theoremstyle{remark}
\newtheorem{rmk}[thm]{Remark}
\newtheorem{recall}[thm]{Recall}
\newtheorem{ex}[thm]{Example}

\crefname{lem}{Lemma}{Lemmas}
\crefname{thm}{Theorem}{Theorems}
\crefname{defn}{Definition}{Definitions}
\crefname{prop}{Proposition}{Propositions}
\crefname{rmk}{Remark}{Remarks}
\crefname{cor}{Corollary}{Corollaries}
\crefname{ex}{Example}{Examples}
\crefname{notation}{Notation}{Notations}
\crefname{constr}{Construction}{Constructions}
\crefname{recall}{Recall}{Recalls}
\crefname{descr}{Description}{Descriptions}
\crefname{para}{\textsection}{\textsection\textsection}

\newlist{rome}{enumerate}{7}
\setlist[rome]{label=(\roman*)}

\newlist{alphabet}{enumerate}{7}
\setlist[alphabet]{label=(\alph*)}

\newcommand{\pushout}[1]{\node at ($({#1})-(10pt,-10pt)$) {$\ulcorner$};}
\newcommand{\pullback}[1]{\node at ($({#1})+(10pt,-10pt)$) {$\lrcorner$};}
\NewDocumentCommand{\punctuation}{ m m O{5pt} }{\node at ($(#1.east)-(0,#3)$) {#2};}

\newcommand{\defThn}{\ensuremath{\theta}}
\newcommand{\defS}{k}

\NewDocumentCommand\Thn{O{n-1}}{\ensuremath{\Theta}_{{#1}}}
\NewDocumentCommand\Thnsset{O{n-1}}{\mathit{s}\set^{\ensuremath{\Theta_{\scriptscriptstyle{#1}}^{\op}}}}
\NewDocumentCommand\Thnssset{O{n-1}}{\mathit{ss}\set^{\ensuremath{\Theta_{\scriptscriptstyle{#1}}^{\op}}}}
\NewDocumentCommand\CSThn{O{n-1}}{{\ensuremath{\scriptscriptstyle(\infty,{#1})}}}
\NewDocumentCommand\AraThn{O{n-1}}{{\ensuremath{\scriptscriptstyle(\infty,{#1})}}}

\newcommand{\Dnc}{\Delta^{\mathrm{nc}}}

\newcommand{\cat}{\cC\!\mathit{at}}
\newcommand{\set}{\cS\!\mathit{et}}
\newcommand{\sset}{\mathit{s}\set}
\newcommand{\Pcat}{\mathcal{P}\cat}
\newcommand{\DThn}{\Delta\times \Thn}
\newcommand{\DThnS}{\DThn\times \Delta}
\newcommand{\ThnS}{\Thn\times \Delta}
\newcommand{\Dop}{\Delta^{\op}}

\newcommand{\Nec}{\mathcal{N}\mathit{ec}}

\newcommand{\MSspace}{\sset_{\Kan}}

\newcommand{\injThnspace}{(\sset_{\Kan})^{\Thnop}_\inj}
\newcommand{\injsThnspace}{(\sset_{\Kan})^{\ThnSop}_\inj}
\newcommand{\MSThnsset}{\Thnsset_{\CSThn}}
\newcommand{\MSThncat}{\Thnsset_{\CSThn}\text{-}\cat}
\newcommand{\injsThnsset}{(\Thnsset_{\CSThn})^{\Delta^{\op}}_\inj}
\newcommand{\SegsThnsset}{(\Thnsset_{\CSThn})^{\Delta^{\op}}_\Seg}

\newcommand{\Dset}{\set^{\Dop}}

\newcommand{\diag}{\mathrm{diag}}
\newcommand{\Cube}{\mathcal C\textit{ube}}

\newcommand{\SSset}{\mathit{s}\sset}
\newcommand{\MSdiag}{\SSset_{\diag}}
\newcommand{\MSThndiag}{\Thnssset_{\diag,\CSThn}}

\newcommand{\Thnset}{\set^{\Thnop}}
\newcommand{\sThnsset}{\mathit{s}\set^{\Thnop\times \Dop}}
\newcommand{\sThnset}{\set^{\Thnop\times \Dop}}
\newcommand{\Thncat}{\Thnsset\text{-}\cat}
\newcommand{\sThncat}{\Thnssset\text{-}\cat}
\newcommand{\Thnop}{{\Theta_{\scriptscriptstyle n-1}^{\op}}}
\newcommand{\DThnop}{\Dop\times \Thnop}

\newcommand{\ThnSop}{\Thnop\times \Dop}

\newcommand{\pcatThn}{\Pcat(\Thnsset)}
\newcommand{\pcatproj}{\Pcat(\Thnsset_{\CSThn})_\proj}
\newcommand{\pcatinj}{\Pcat(\Thnsset_{\CSThn})_\inj}
\newcommand{\fib}{\mathrm{fib}}

\NewDocumentCommand\Sp{O{m}}{Sp[{#1}]}
\NewDocumentCommand\repD{O{m}}{F[{#1}]}
\NewDocumentCommand\repS{O{\defS}}{\ensuremath{\Delta}[{#1}]}
  
\NewDocumentCommand\repThn{O{\defThn}}%
  {\Thn{[{#1}]}}
\NewDocumentCommand\repDThn{O{m} O{\defThn}}%
  {F[{#1}]\times \Thn{[{#2}]}}
\NewDocumentCommand\repDThnS{O{m} O{\defThn} O{\defS}}%
  {F[{#1}]\times \Thn{[{#2}]}\times \ensuremath{\Delta}[{#3}]}
\NewDocumentCommand\repThnS{O{\defThn} O{\defS}}%
  {\Thn[{#1}]\times \ensuremath{\Delta}[{#2}]}
  
\NewDocumentCommand\HomSh{O{m} O{X}}{{\Hom_{\Sh_{{#1}}{#2}}}}
\NewDocumentCommand\HomSigma{O{m} O{X}}{{\Hom_{\ensuremath{\Sigma}_{{#1}}{#2}}}}

\newcommand{\Ch}{\mathfrak{C}}
\newcommand{\Nh}{\mathfrak{N}}

\DeclareMathOperator{\Ob}{Ob}
\DeclareMathOperator{\colim}{colim}

\DeclareMathOperator{\Mor}{Mor}
\DeclareMathOperator{\Ho}{Ho}

\DeclareMathOperator{\map}{map}
\DeclareMathOperator{\Hom}{Hom}
\DeclareMathOperator{\Map}{Map}
\DeclareMathOperator{\op}{op}
\DeclareMathOperator{\ev}{ev}
\DeclareMathOperator{\id}{id}

\newcommand{\Seg}{\mathrm{Seg}}

\newcommand{\Kan}{\CSThn[0]}
\newcommand{\inj}{\mathrm{inj}}
\newcommand{\proj}{\mathrm{proj}}
\newcommand{\sCat}{{\sset\text{-}\cat}}

\newcommand{\tndnec}[3]{\Nec({#1})_{{#2},{#3}}^{\mathrm{tnd}}}
\newcommand{\catnec}[3]{\Nec({#1})_{{#2},{#3}}}

\newcommand{\Sh}{\mathfrak{S}}

\newcommand{\NL}{N^h}
\newcommand{\CL}{c^h}

\NewDocumentCommand\PP{O{m} O{X} O{Y}}{{P_{#1}({#2}\hookrightarrow {#3})}}
\NewDocumentCommand\inc{O{m} O{X} O{Y}}{{I}}
\NewDocumentCommand\QmXY{O{m} O{X} O{Y}}{{Q}}

\title{A homotopy coherent nerve for $(\infty,n)$-categories}

\author{Lyne Moser}
\address{Fakultät für Mathematik, Universität Regensburg, Regensburg, Germany}
\email{lyne.moser@ur.de}

\author{Nima Rasekh}
\address{Max Planck Institute for Mathematics, Bonn, Germany}
\email{rasekh@mpim-bonn.mpg.de}

\author{Martina Rovelli}
\address{Department of Mathematics and Statistics, University of Massachusetts Amherst, Amherst, USA}
\email{mrovelli@umass.edu}

\keywords{$(\infty,n)$-categories, homotopy coherent nerve, enriched categories, (complete) Segal objects.}
\subjclass[2020]{18N65; 55U35; 18N50.}

\begin{document}

\maketitle

\begin{abstract}
In the case of $(\infty,1)$-categories, the homotopy coherent nerve gives a right Quillen equivalence between the models of simplicially enriched categories and of quasi-categories. This shows that homotopy coherent diagrams of $(\infty,1)$-categories can equivalently be defined as functors of quasi-categories or as simplicially enriched functors out of the homotopy coherent categorifications. 

In this paper, we construct a homotopy coherent nerve for $(\infty,n)$-categories. We show that it realizes a right Quillen equivalence between the models of categories strictly enriched in $(\infty,n-1)$-categories and of Segal category objects in $(\infty,n-1)$-categories. This similarly enables us to define homotopy coherent diagrams of $(\infty,n)$-categories equivalently as functors of Segal category objects or as strictly enriched functors out of the homotopy coherent categorifications. 
\end{abstract}

\setcounter{tocdepth}{1}
\tableofcontents

\section*{Introduction}

\subsection*{The challenge of coherent mathematics}

The concept of equality has firmly established itself as an important part of mathematical foundation and enables us to define a variety of mathematical objects, particularly algebraic ones,  such as groups or rings. However, in recent decades we are more and more confronted with objects whose structure cannot be captured via equalities. A simple example is given by the loop space; we can compose loops, however the two possible compositions of three loops are only homotopic rather than equal.

These encounters have motivated the rise of \emph{coherent mathematical structures}. Intuitively, the notion of a coherent structure is easy to convey: one simply replaces equalities with an appropriately chosen data, which could be a path in a topological space, a quasi-isomorphism of two chain complexes or a term in an identity type in a given type theory. However, making this idea precise turns into a challenge. Indeed, each layer of data that witnesses an equality necessitates one higher layer of data that guarantees all previous choices are appropriately compatible. This can already be witnessed in the definition of a monoidal category whose associator, the isomorphism witnessing associativity, needs to satisfy the pentagon identity. As a result, any effort to explicate coherent structures results in an infinite and interlocked tower of intractable data. 

In certain situations the infinite tower of data that arises in such situations can be tackled effectively via modern machinery, such as operads. For example, we can give a precise definition of a homotopy group via the $A_\infty$-operad and then show loop spaces are an example of such coherent groups. These methods using operadic techniques have, among others, been effectively used by Haugseng and various collaborators, to study a wide range of homotopy coherent settings \cite{gepnerhaugseng2015enriched,haugseng2017morita,chuhaugseng2021segal,elmantohaugseng2023bispans}.

\subsection*{Homotopy coherent nerve}

Despite those advances, we cannot always tackle the issue of defining coherent structures by hand and we need to find a more conceptual approach that can generalize a given (algebraic) structure to its appropriately defined coherent analogue. Here we can benefit from the well-known observation that algebraic structures can be characterized via appropriately chosen functors. For example, the category of monoid objects in a finitely complete category $\cC$ is precisely given via a full subcategory of simplicial objects in $\cC$. This suggests that an important first step towards defining homotopy coherent structures consists of developing an appropriate notion of homotopy coherent functors out of small categories, such as $\Delta$. Similar to above, intuitively a homotopy coherent functor should satisfy functoriality only up to appropriately chosen data. However, again it is challenging to translate this intuition into a precise mathematical definition and there are two broad ways we can approach this problem:

\begin{enumerate}[leftmargin=.6cm]
    \item[(I)] We can adjust a given indexing category in a specific way so that functors out of this category now incorporate the desired coherence.
    \item[(II)] Instead of solving the problem one category at a time, we identify an appropriate homotopy coherent generalization  of the notion of a category itself. Then a coherent diagram would simply be a functor in this generalized setting.
\end{enumerate}

A first comprehensive solution following the line of thinking outlined in (I) was employed by Cordier and Porter \cite{cordier1982coherent,CordierPorterHomotopyCoherent}. They constructed an adjunction
\begin{tz}
\node[](1) {$\Dset$}; 
\node[right of=1,xshift=1.6cm](2) {$\sCat$}; 
\punctuation{2}{,};

\draw[->] ($(2.west)-(0,5pt)$) to node[below,la]{$\NL$} ($(1.east)-(0,5pt)$);
\draw[->] ($(1.east)+(0,5pt)$) to node[above,la]{$\CL$} ($(2.west)+(0,5pt)$);

\node[la] at ($(1.east)!0.5!(2.west)$) {$\bot$};
\end{tz}
known as the \emph{homotopy coherent categorification} and \emph{homotopy coherent nerve} adjunction between simplicial sets and simplicially enriched categories. In particular, for a given ($1$-)category $\cC$, in the simplicially enriched category $c^h N^h \cC$, every pair of composable morphisms is now only related by a path to its original composite in $\cC$, and any original instance of associativity in $\cC$ is now witnessed by higher simplices. Hence a homotopy coherent functor out of $\cC$ can be defined very precisely as a simplicially enriched functor out of $c^h N^h \cC$, resting assured that the simplicial enrichment takes care of the desired coherence. 

A proper development of a fully coherent category theory realizing approach (II) did follow not long after. Starting in the $90$s we saw the rise of various weak models of $(\infty,1)$-categories, prominent among them quasi-categories \cite{joyal2008notes} and complete Segal spaces \cite{rezk2001css}. There the $(\infty,1)$-categories are defined as certain simplicial objects and functors are defined as simplicial morphisms, meaning the coherence is built into the definition of a functor via simplicial identities. In fact early versions of quasi-categories were precisely introduced with the goal of characterizing homotopy coherent data \cite{boardmanvogt1973qcats}.

A priori this suggests two different definitions of a homotopy coherent diagram, however, closing this long developmental arc, it was proved first by Joyal, then Lurie \cite{htt}, and also Dugger--Spivak \cite{DuggerSpivakRigidification,DuggerSpivakMapping} that the adjunction $c^h \dashv N^h$ in fact gives us an equivalence, by establishing a Quillen equivalence of model categories, which in particular means the two notions of homotopy coherent data are appropriately equivalent. As a consequence, every quasi-category is (up to equivalence) of the form $N^h\cC$ for some Kan-enriched category $\cC$ and so for a given simplicial set $K$ a homotopy coherent diagram in sense above $c^h K \to \cC$ is the same as a homotopy coherent diagram in the sense of quasi-categories $K \to N^h\cC$. 

To summarize, as a result of this extensive work, we can now very precisely define a homotopy coherent diagram as a functor of quasi-categories or, equivalently, as a simplicially enriched strict functor out of the categorification of the homotopy coherent nerve, each approach having shown their advantages in a variety of settings.

\begin{enumerate}[leftmargin=0.6cm]
    \item \textbf{Classifying diagrams in $(\infty,1)$-categories:} The homotopy coherent nerve enables us to give explicit descriptions of homotopy coherence via classifying objects. For example, in \cite[\textsection4.4.5]{htt}, Lurie uses the homotopy coherent nerve to construct the \emph{homotopy coherent idempotent classifier} and uses that to prove that homotopy coherent idempotent completion is an infinite operation, meaning (unlike the $1$-categorical case) there are finitely complete categories that are not idempotent complete.   
    \item \textbf{Coherent diagrams valued in spaces:} The same way that the category of sets plays a central role in classical category theory, the $(\infty,1)$-category of spaces plays an analogous role in $(\infty,1)$-category theory, being the natural codomain of representable functors. As a result, defining and studying homotopy coherent diagrams of spaces plays a central role. However, there is no direct non-technical way to construct the quasi-category of spaces given all the higher coherences it entails and the most standard construction is given by the Kan-enriched category of Kan complexes. That means we cannot use method (II) to study homotopy coherent diagrams of spaces, and need the homotopy coherent nerve, an important example being the first construction of the Yoneda embedding for quasi-categories; see \cite[Proposition 5.1.3.1]{htt}.
    \item \textbf{Straightening construction for $(\infty,1)$-categories:} In \cite{htt} Lurie uses the homotopy coherent nerve in an essential manner to define the {\it straightening construction}, which for a given simplicial set $K$ identifies homotopy coherent diagrams out of $c^h K^{\op}$ valued in spaces with right fibrations over~$K$; see \cite[Theorem 2.2.1.2]{htt}. The straightening construction provides us with the most effective method to analyze coherent diagrams and particularly identify representable functors. It is hence the key step in the development of $(\infty,1)$-category theory, such as the study of limits or presentability; see \cite{joyal2008notes} and \cite[\textsection4-5]{htt}.
    \item \textbf{$(\infty,1)$-limits:} When working with $(\infty,1)$-categories modeled by strictly Kan-enriched categories, we can rely on the extensive literature regarding simplicially enriched colimits; see e.g.~\cite{riehl2014categoricalhomotopy}. However, this approach is computationally unfeasible as it necessitates constructing free contractible homotopy coherent diagrams (concretely modeled by the cofibrant replacement of the terminal diagram). Therefore, instead of studying limits for diagrams valued in a category strictly enriched over spaces, one prefers to use the notion of a limit for diagrams valued in the corresponding quasi-category, as defined by Joyal \cite{joyal2008notes}, using quasi-categories of cones.

    While studying limits via cones is much more effective, it creates the possibility of a mismatch between the two possible notions of limits. However, using the interplay between homotopy coherent nerves and homotopy coherent diagrams permits Riehl and Verity \cite{RiehlVerityNCoh} and the third author \cite{rovelli} to describe limits in a quasi-category as a weighted simplicially enriched limits of its corresponding homotopy coherent diagram, which has both demonstrates that the notions agree appropriately  as well as aids with computations. 
\end{enumerate}

\subsection*{Developing a theory of \texorpdfstring{$(\infty,n)$}{(oo,n)}-categories}
While many structures (such as groups) by default assemble into categories, some naturally exhibit more data, the prime example being categories, which assemble into a $2$-category given by categories, functors, and natural transformations. Proceeding inductively we can more generally define a strict $n$-category as consisting of objects, $1$-morphisms, $2$-morphisms between $1$-morphisms, up to $n$-morphisms between $(n-1)$-morphisms, or, more succinctly a category enriched over strict $(n-1)$-categories. Similar to before we are confronted with objects that satisfy equalities only in a coherent manner, an example being monoidal $n$-categories, and hence would like to define and study coherent structures in this setting. 

As before there are two main ways to tackle this problem:
\begin{enumerate}[leftmargin=0.6cm]
    \item[(I)] We can work with a notion of weak $n$-category that is strictly enriched and adjust the chosen diagram so that strictly enriched functors already encode the desired homotopy coherence.
    \item[(II)] We can develop a notion of weak $n$-categories, such that functors are by definition coherent.  
\end{enumerate}

For historical reasons, we will start with approach (II) as it has been developed much more extensively. There is now a wide range of weak models of $(\infty,n)$-categories, explicitly given as presheaves on appropriately chosen diagram categories, such as (saturated) $n$-complicial sets \cite{VerityComplicialAMS,or}, $n$-fold complete Segal spaces \cite{BarwickThesis}, complete Segal $\Thn[n]$-spaces \cite{rezkTheta}, $n$-quasi-categories \cite{ara}, and saturated $n$-comical sets \cite{CKM,DKM}. Hence, relying on the existing literature we can define already homotopy coherent diagrams as functors in these weak models.

The situation regarding approach (I) has as of yet remained unclear. We can generalize simplicially enriched categories (that we used in the $(\infty,1)$-categorical setting) in a way that incorporates $n$-categories, by strictly enriching categories over any of the weak models of $(\infty,n-1)$-categories introduced above. While we know that this strictly enriched model is abstractly equivalent to a weak model via an intricate zig-zag of equivalences \cite{br1,br2}, we currently do not have a homotopy coherent categorification and homotopy coherent nerve adjunction that can help us adjust a given $n$-category in a manner that incorporates homotopy coherence. This is despite the fact that such a construction would be key in obtaining several further results, analogous to the work done for $(\infty,1)$-categories.

\begin{enumerate}[leftmargin=0.6cm]
    \item \textbf{Classifying diagrams in $(\infty,n)$-categories:} Similar to the $(\infty,1)$-categorical situation we would like to have the ability to construct classifying objects for important diagrams with the goal of understanding the data of a diagram by analyzing its classifying object. There are successful examples in $(\infty,2)$-category that managed to avoid the nerve, such as the construction of the free homotopy coherent adjunction due to Riehl and Verity, which benefited from the fact that the free homotopy coherent adjunction happened to be a simplicial computad, which guarantees the required coherence \cite{riehlverity2016adjunctions}. This does not hold for general diagrams of interest (for example the classifying diagram of a bimonad \cite{bv2007hopfmonads}) and hence any further advances in this direction requires a deep understanding of more general coherent diagrams.
    \item \textbf{Coherent diagrams valued in $(\infty,n-1)$-categories:} Arguably the most important $(\infty,n)$-category is the $(\infty,n)$-category of $(\infty,n-1)$-categories and any advance in the theory of $(\infty,n)$-categories, particularly the study of representable functors and the Yoneda embedding, necessitates a conceptual and computational understanding of homotopy coherent diagrams valued in $(\infty,n-1)$-categories. Similar to the case of $(\infty,1)$-categories existing constructions of this $(\infty,n)$-category are given via strict models and so we need a homotopy coherent nerve to be able to define homotopy coherent diagrams valued in $(\infty,n-1)$-categories.
    \item \textbf{Straightening construction for $(\infty,n)$-categories:} Any advances in the theory of $(\infty,n)$-categories necessitates an ability to analyze functors valued in the $(\infty,n)$-category of $(\infty,n-1)$-categories, and particularly computationally feasible criteria when such a functor is representable. As discussed above, in the $(\infty,1)$-categorical context this has mainly been achieved via the straightening construction, which studies presheaves via fibrations. We hence anticipate the existence of a similar straightening construction for $(\infty,n)$-categories, the construction of which should similarly fundamentally hinge on an appropriately defined categorification functor.
    \item \textbf{$(\infty,n)$-limits:} Similar to the $(\infty,1)$-case, the correct notion of a limit for diagrams valued in an $(\infty,n)$-category presented by an enriched category over a model of $(\infty,n-1)$-categories is already established as part of a more general pattern for enriched categories; see \cite{shulman}. 
    However, similar to the $(\infty,1)$-categorical case discussed above, this approach is often computationally unfeasible, suggesting the need for an alternative, more computationally feasible, approach to limits via cones. However, any such approach would need to be compatible with limits in the strictly enriched setting, which similar to the case for $(\infty,1)$-categories necessitates an appropriately defined homotopy coherent nerve.
\end{enumerate}

\subsection*{A homotopy coherent nerve of \texorpdfstring{$(\infty,n)$}{(oo,n)}-categories}
To summarize the previous paragraph, we already have a weak notion of $(\infty,n)$-categories and their corresponding notion of functor. However, we lack the ability to strictify coherent data in a way that gives us an equivalence between weak and strict functors, although having such an ability is a key component towards further advancing $(\infty,n)$-category theory. The goal of this paper is to precisely address these two shortcomings.

 Concretely we construct in \cref{defadjunction} an adjunction $\Ch \dashv \Nh$ consisting of the homotopy coherent categorification and homotopy coherent nerve between a strictly enriched model of $(\infty,n)$-categories (categories strictly enriched over complete Segal $\Thn$-spaces) and a weak model of $(\infty,n)$-categories (Segal category objects in complete Segal $\Thn$-spaces), and show that it is a Quillen equivalence in \cref{ThmMain}.

\begin{thm*}
There is a Quillen equivalence
\begin{tz}
\node[](1) {$\MSThncat$}; 
\node[right of=1,xshift=3.1cm](2) {$\pcatinj$}; 

\draw[->] ($(1.east)-(0,5pt)$) to node[below,la]{$\Nh$} ($(2.west)-(0,5pt)$);
\draw[->] ($(2.west)+(0,5pt)$) to node[above,la]{$\Ch$} ($(1.east)+(0,5pt)$);

\node[la] at ($(1.east)!0.5!(2.west)$) {$\bot$};
\end{tz}
between the model structure $\MSThncat$ of which the fibrant objects are the categories enriched over complete Segal $\Thn$-spaces, and the model structure $\pcatinj$  of which the fibrant objects are the injectively Segal category objects in complete Segal $\Thn$-spaces.
\end{thm*}

The Quillen equivalence enables us to realize all of the goals outlined above.
\begin{enumerate}[leftmargin=0.6cm]
    \item First of all we can now define a homotopy coherent diagram out of a category $\cC$ as a strictly enriched functor out of $\Ch\Nh\cC$, where the enrichment guarantees the desired homotopy coherence. Moreover, given a (fibrant) $\MSThnsset$-enriched category $\cC$, we can use this explicit Quillen equivalence to represent the same $(\infty,n)$-category as the Segal category object $\Nh\cC$ in complete Segal $\Thn$-spaces. Furthermore, every diagram $W\to \Nh\cC$ can be represented as a diagram $\Ch W\to \cC$. This precisely establishes that the two possible notions of homotopy coherent diagrams coincide with each other.
    \item As a particular application of the previous item, the category of complete Segal $\Thn$-spaces is enriched over itself, meaning we can define a homotopy coherent diagram valued in complete Segal $\Thn$-spaces as a functor of precategory objects valued in its homotopy coherent nerve.
    \item In follow-up work \cite{MRR2} we use the homotopy coherent categorification to construct a straightening construction, which for every $W \in \pcatinj$ constructs an equivalence between strictly enriched functors $\Ch W^{\op} \to \Thnsset$ and double $(\infty,n-1)$-right fibrations over $W$. This is a direct generalization of the $(\infty,1)$-categorical straightening construction in \cite{htt}, and is expected to play a similar fundamental role in all of $(\infty,n)$-category theory. 
    \item In \cite{MRR3}, we develop a notion of limit for $(\infty,n)$-categories via double $(\infty,n-1)$-categorical cones that does correctly coincide with the strict definitions, generalizing work done in the $2$-categorical setting by clingman--Moser \cite{clingmanmoser2022limits}, Grandis \cite{grandis2020limits}, Grandis--Paré \cite{GraPar1999,GraPar2019}, and Verity \cite{Verity}. Combining our results here with work done in \cite{MRR2} we will show in upcoming work that this notion of limit for $(\infty,n)$-categories is independent of the model.
\end{enumerate}

\subsection*{Necklace calculus}

In two seminal papers Dugger and Spivak developed a theory of necklaces, as an effective tool to study hom spaces of homotopy coherent categorifications of quasi-categories \cite{DuggerSpivakMapping,DuggerSpivakRigidification}. The power of the necklace machinery can be witnessed in the widespread applications it has found in several other (related) contexts, such as \cite{hhr2021straightening,cls2022cubical,bs2023segalification,lowenmertens2023templicial}.

As part of our effort to study and construct the homotopy coherent nerve, we describe effective tools to make computations via necklaces in a context suitable for $(\infty,n)$-categories; this \emph{necklace calculus} could be of independent interest. In particular, we characterize a broad class of simplicial sets that play an important role in the study of $(\infty,n)$-categories, the \emph{$1$-ordered simplicial sets}, for which the computation of the hom space via necklaces can be reduced to the colimit over a poset. See 
\cref{DS:computations1ordered} for a more explicit statement.

The theory of $1$-ordered simplicial sets and their associated necklace calculus gives us a concrete method to compute hom objects of homotopy coherent categorifications of relevant objects. For example, given $m\geq 0$ and a $\Theta_{n-1}$-space $X$, one can consider the Segal category object $L(\repD\times X)$, which models an $(\infty,n)$-category with $m+1$ objects $0,1,\ldots,m$ and hom $\Thn$-spaces $X$ between consecutive objects (see \cref{RmkPushout}). Here, the simplicial set $\repD$ models the category $[m]$, and its homotopy coherent categorification $\Ch \repD=\CL\repD$ is classically understood (see \cref{defn:CL}). The canonical projection $L(\repD\times X) \to F[m]$ induces a family of discrete fibrations which relate the categories of necklaces obtained from $L(\repD\times X)$ and the category of necklaces of $\repD$ (see \cref{prop:discretefib}). The necklace calculus developed in this paper allows us to compute the hom objects of $\Ch L(\repD\times X)$ from those of $\Ch\repD$
(see \cref{prop:computationpushprod}), which is a key ingredient for the proof of the main theorem.

\subsection*{Acknowledgments}

The paper benefited from helpful conversations with Julie Bergner and Viktoriya Ozornova. We would also like to thank the anonymous referee for their insightful comments. This material is based upon work supported by the National Science Foundation under Grant No.~DMS-1928930 while the first and third authors participated in a program supported by the Mathematical Sciences Research Institute. The program was held in Summer 2022 in partnership with the Universidad Nacional Autónoma de M\'exico. The third author is grateful for support from the National Science Foundation under Grant No.~DMS-2203915, and the first author for the support of the Max Planck Institute for Mathematics.

\section{Preliminaries and background}

In this section we recall the relevant model structure for $(\infty,n-1)$-categories in \cref{subsec:MSn-1}, the model structure for categories enriched over $(\infty,n-1)$-categories in \cref{subsec:MSenriched}, the model structure for Segal categories in $(\infty,n-1)$-categories in \cref{subsec:MSPcat}, and the diagonal model structure in \cref{subsec:MSdiag}. We also recall in \cref{subsec:strictnerve} the Quillen equivalence between models of $(\infty,n)$-categories given by the strict nerve of categories enriched over $(\infty,n-1)$-categories.

\subsection{Model structures for \texorpdfstring{$(\infty,n-1)$}{(infinity,n-1)}-categories} \label{subsec:MSn-1}

We recall the model structure $\MSThnsset$ on $\Thnsset$ for $(\infty,n-1)$-categories given by Rezk's complete Segal $\Thn$-spaces \cite{rezkTheta}.

For $n\geq1$, recall from \cite{JoyalDisks} Joyal's cell category $\Thn[n]$. For $n=1$, then $\Thn=\Thn[0]$ is the terminal category, and for $n>1$, the category $\Thn$ is the \emph{wreath product} $\Delta\wr \Thn[n-2]$ (see e.g.~\cite[Definition 3.1]{BergerIterated}).

Throughout the paper we will use the following notational conventions.
\begin{notation}
We write:
\begin{itemize}[leftmargin=0.6cm]
\item $\repD\in \set^{\Dop}$ for the representable at $m\geq 0$, and $\Sp\coloneqq \repD[1]\amalg_{\repD[0]}\ldots \amalg_{\repD[0]}\repD[1]$ for the spine of $\repD$,
    \item $\repThn\in \Thnset$ for the representable at $\defThn\in \Thn$,
    \item $\repS\in \sset$ for the representable at $\defS\geq 0$,
    \item $\repThn\times\repS\in \Thnsset$ for the representable at $(\defThn,[\defS])\in \ThnS$, 
    \item $\repDThn\in \sThnset$ for the representable at $([m],\defThn)\in \DThn$, 
    \item $\repDThnS\in \sThnsset$ for the representable at $([m],\defThn,[\defS])\in \DThnS$.
\end{itemize}
The categories $\Dset$, $\Thnset$, $\sset$, $\Thnsset$, and $\sThnset$ are all naturally included into $\sThnsset$, and we regard all the above as objects of it without further specification. We refer to an object of $\Thnsset$ as a \emph{$\Thn$-space}.
\end{notation}

Roughly speaking, we think of $\repD$ as the standard $m$-simplex living in the \emph{categorical direction} and of $\repS$ as the standard $\defS$-simplex living in the \emph{spacial direction}. More generally, we follow the convention that, given any small category $\cA$, the simplicial direction in $\cA^{\Delta^{\op}}$ is considered to be categorical, whereas the simplicial direction in $s\cA$ is considered to be spacial.

The model structure $\MSThnsset$ is defined recursively as a localization of the injective model structure $\injThnspace$ on the category of $\Thn$-presheaves valued in $\MSspace$ with respect to a set $S_{\CSThn}$ of maps in $\Thnset$. 

The set $S_{\CSThn[0]}$ is the empty set, and for $n>1$ the set $S_{\CSThn}$ consists of the following monomorphisms: 
\begin{itemize}[leftmargin=0.6cm]
    \item the \emph{Segal maps} \[ \repThn[1;\theta_1]\amalg_{[0]} \ldots\amalg_{[0]} \repThn[1;\theta_\ell]\hookrightarrow \repThn[\ell;\theta_1,\ldots,\theta_\ell], \] for all $\ell\geq 1$ and $\theta_1,\ldots,\theta_\ell\in \Thn[n-2]$,
    \item the \emph{completeness map}
    \[ \repD[0]\hookrightarrow N \bI \] seen as a map in $\Thnset$ through the inclusion $\Dset\hookrightarrow \Thnset$ induced by pre-composition along the projection $\Thn\to \Delta$ given by $[\ell;\theta_1,\ldots,\theta_l]\mapsto [\ell]$, where $\bI$ denotes the free-living isomorphism, 
    \item the \emph{recursive maps} \[ \repThn[1;A]\hookrightarrow\repThn[1;B], \] where $A\hookrightarrow B\in \Thnsset[n-2]$ ranges over all monomorphisms in $S_{\CSThn[n-2]}$.
\end{itemize}
Note that by \cite[Theorem 8.1]{rezkTheta} the model structure $\MSThnsset$ obtained by localizing the injective model structure $\injThnspace$ with respect to the set $S_{\CSThn}$ is cartesian closed. This is enough to guarantee that the model structure $\MSThnsset$ is excellent in the sense of \cite[Definition A.3.2.16]{htt}.

\subsection{Enriched model structures for \texorpdfstring{$(\infty,n)$}{(infinity,n)}-categories} \label{subsec:MSenriched}

Since the model structure $\MSThnsset$ is excellent, the category $\Thncat$ supports the left proper model structure $\MSThncat$ from \cite[\textsection3.10]{br1}, obtained as a special instance of \cite[Proposition A.3.2.4, Theorem A.3.2.24]{htt}. The main features of this model structure rely on the notion of \emph{homotopy category} from \cite[\textsection~A.3.2]{htt}, which we now recall.

\begin{defn}
Let $\cC$ be a $\Thnsset$-enriched category. The \emph{homotopy category} of $\cC$ is the category $\Ho\cC $ such that
\begin{itemize}[leftmargin=0.6cm]
    \item its set of objects $\Ob(\Ho \cC)$ is $\Ob \cC$,
    \item for $a,b\in \Ob\cC$, its hom set is given by
    \[ (\Ho\cC)(a,b)\coloneqq \Ho(\MSThnsset)(\repS[0],\Hom_\cC(a,b)),\]
    where $\Ho(\MSThnsset)$ is the homotopy category of the model category $\MSThnsset$,
    \item composition is induced from that of $\cC$. 
\end{itemize}
\end{defn}

Finally, we recall some of the data defining the model structure $\MSThncat$.

\begin{recall}
In the model structure $\MSThncat$, a $\Thnsset$-enriched category $\cC$ is \emph{fibrant} if, for all $a,b\in \Ob\cC$, the hom $\Thn$-space $\Hom_\cC(a,b)$ is fibrant in $\MSThnsset$, and a $\Thnsset$-enriched functor $F\colon \cC\to \cD$ is:
\begin{itemize}[leftmargin=0.6cm]
    \item a \emph{weak equivalence}
    if the induced functor $\Ho F\colon \Ho\cC\to \Ho\cD$ between homotopy categories is essentially surjective on objects, and for all $a,b\in\Ob\cC$ the induced map
    \[F_{a,b}\colon\Hom_{\cC}(a,b)\to\Hom_\cD(Fa,Fb)\]
    is a weak equivalence in $\MSThnsset$,
    \item a \emph{fibration between fibrant objects} if it the induced functor $\Ho F\colon \Ho\cC\to \Ho\cD$ between homotopy categories is an isofibration of categories, and for all $a,b\in\Ob\cC$ the induced map
    \[F_{a,b}\colon\Hom_{\cC}(a,b)\to\Hom_\cD(Fa,Fb)\]
    is a fibration in $\MSThnsset$, 
    \item a \emph{trivial fibration} if it is surjective on objects, and for all $a,b\in\Ob\cC$ the induced map
    \[F_{a,b}\colon\Hom_{\cC}(a,b)\to\Hom_\cD(Fa,Fb)\]
   a trivial fibration in $\MSThnsset$.
\end{itemize} 
\end{recall}

The homs of the homotopy category of a fibrant $\MSThnsset$-enriched category admit a more explicit description in terms of $\pi_0\colon \sset\to \set$, the left adjoint to the inclusion $\set\hookrightarrow\sset$. 

\begin{prop} \label{prop:homofHoC}
Let $\cC$ be a fibrant $\MSThnsset$-enriched category. Then, for all $a,b\in \Ob\cC$, there is a natural isomorphism of sets
    \[ (\Ho\cC)(a,b)\cong\pi_0(\Hom_\cC(a,b)_{[0]}).\]
\end{prop}

\begin{proof}
Since the model structure $\MSThnsset$ is simplicial, as a consequence of \cite[Proposition 9.5.24]{Hirschhorn} we have that, for every object $A\in \MSThnsset$ and every fibrant object $X\in \MSThnsset$, an isomorphism of sets
\[ \Ho(\MSThnsset)(A,X)\cong \pi_0\Map_{\Thnsset}(A,X), \]
where $\Map_{\Thnsset}(-,-)$ denotes the hom space functor. Hence, if $\cC$ is fibrant $\MSThncat$, then, for every $a,b\in \Ob\cC$, the hom $\Thn$-space $\Hom_{\cC}(a,b)$ is fibrant in $\MSThnsset$ and so we get an isomorphism of sets
\[ \Ho(\cC)(a,b)\cong \pi_0\Map_{\Thnsset}(\repS[0],\Hom_\cC(a,b))\cong \pi_0(\Hom_\cC(a,b)_{[0]}). \qedhere\]
\end{proof}

Many of the $\Thnsset$-enriched categories that feature in this paper have the following property, so we introduce a terminology that streamlines the exposition.

\begin{defn} 
A $\Thnsset$-enriched category $\cC$ is \emph{directed} if
\begin{itemize}[leftmargin=0.6cm]
    \item its set of objects $\Ob\cC$ is $\{0,1,\ldots,m\}$, for some $m\geq 0$, 
    \item for $0\leq j\leq i\leq m$, the hom $\Thn$-space $\Hom_\cC(i,j)$ is given by
    \[\Hom_\cC(i,j)=\begin{cases}
\emptyset & \text{if} \; j<i \\
\repS[0] & \text{if} \; j=i.
\end{cases}\] 
    \end{itemize}
In particular, composition maps in a directed $\Thnsset$-enriched category $\cC$ involving the above hom $\Thn$-spaces are uniquely determined. Moreover, the value of a $\Thnsset$-enriched functor from a directed $\Thnsset$-enriched category is also uniquely determined on these hom $\Thn$-spaces.
\end{defn}

The assignment $(\cC,a,b)\mapsto\Hom_{\cC}(a,b)$ of the hom $\Thn$-space to every two objects $a$ and $b$ of a $\Thnsset$-enriched category $\cC$ defines a functor $\Hom\colon {}^{\{0,1\}/}\Thncat\to \Thnsset$, where ${}^{\{0,1\}/}\Thncat$ denotes the category of bi-pointed $\Thnsset$-enriched categories. This functor admits a left adjoint, the \emph{suspension} functor $\Sigma\colon\Thnsset\to {}^{\{0,1\}/}\Thncat$. Given an object $X\in \Thnsset$, the $\Thnsset$-enriched category $\Sigma X$ is the directed $\Thnsset$-enriched category with object set $\{0,1\}$ and hom $\Thn$-space given by $\Hom_{\Sigma X}(0,1)=X$.

The model structure $\MSThncat$ is designed so that the adjunction $\Sigma\dashv \Hom$ has good homotopical properties. Here ${}^{\{0,1\}/}\MSThncat$ denotes the slice model structure, in which cofibrations, fibrations, and weak equivalences are created by the forgetful functor to $\MSThncat$. 

\begin{prop} \label{susp:prestrivcof}
The adjunction
\begin{tz}
\node[](1) {$\MSThnsset$}; 
\node[right of=1,xshift=2.8cm](2) {${}^{\{0,1\}/}\MSThncat$};
\punctuation{2}{,};

\draw[->] ($(2.west)-(0,5pt)$) to node[below,la]{$\Hom$} ($(1.east)-(0,5pt)$);
\draw[->] ($(1.east)+(0,5pt)$) to node[above,la]{$\Sigma$} ($(2.west)+(0,5pt)$);

\node[la] at ($(1.east)!0.5!(2.west)$) {$\bot$};
\end{tz}
is a Quillen pair.
\end{prop}

\begin{proof}
This follows directly from \cite[Lemma E.2.13] {JoyalVolumeII} and the local properties of trivial fibrations and fibrations between fibrant objects.
\end{proof}

The following lemma gives a useful criterion to recognize when a $\Thnsset$-enriched functor is a (trivial) cofibration in $\MSThnsset$.

\begin{lemma}  \label{pushoutlemma}
Let $\cP$ and $\cQ$ be directed $\Thnsset$-enriched categories such that
\begin{itemize}[leftmargin=0.6cm]
    \item they have the same set of objects $\Ob\cP=\{0,1,\ldots,m\}=\Ob\cQ$,
    \item for $0<j-i<m$, they have the same hom $\Thn$-spaces $\Hom_\cP(i,j)=\Hom_\cQ(i,j)$.
\end{itemize} 
Let $F\colon \cP\to \cQ$ be a $\Thnsset$-enriched functor such that
\begin{itemize}[leftmargin=0.6cm]
    \item on objects, it is the identity at $\{0,1,\ldots,m\}$,
    \item for all $0<j-i<m$, the map $F_{i,j}$ on hom $\Thn$-spaces is the identity.
\end{itemize}
Then the following is a pushout in $\Thncat$.
\begin{tz}
\node[](1) {$\Sigma \Hom_{\cP}(0,m)$}; 
\node[below of=1](2) {$\Sigma \Hom_\cQ(0,m)$}; 
\node[right of=1,xshift=1.7cm](3) {$\cP$}; 
\node[below of=3](4) {$\cQ$}; 
\pushout{4};

\draw[->] (1) to node[above,la]{$\iota_{0,m}$} (3);
\draw[->] (1) to node[left,la]{$\Sigma F_{0,m}$} (2);
\draw[->] (3) to node[right,la]{$F$} (4);
\draw[->] (2) to node[below,la]{$\iota_{0,m}$} (4);
\end{tz}
Moreover, if $F_{0,m}$ is a (trivial) cofibration in $\MSThnsset$, then $F\colon \cP\to \cQ$ is a (trivial) cofibration in $\MSThncat$.
\end{lemma}

\begin{proof}
In order to show that $\cQ$ satisfies the universal property of the desired pushout, we show that there is a unique $\Thnsset$-enriched functor $H\colon \cQ\to\cC$ making the following diagram commute.
\begin{tz}
\node[](1) {$\Sigma \Hom_{\cP}(0,m)$}; 
\node[below of=1](2) {$\Sigma \Hom_\cQ(0,m)$}; 
\node[right of=1,xshift=1.7cm](3) {$\cP$}; 
\node[below of=3](4) {$\cQ$}; 
\node[below right of=4,xshift=.5cm](5) {$\cC$}; 

\draw[->] (1) to node[above,la]{$\iota_{0,m}$} (3);
\draw[->] (1) to node[left,la]{$\Sigma F_{0,m}$} (2);
\draw[->] (3) to node[right,la]{$F$} (4);
\draw[->] (2) to node[below,la]{$\iota_{0,m}$} (4);
\draw[->,bend left] (3) to node[right,la]{$G$} (5);
\draw[->,bend right=20] (2) to node[below,la]{$K$} (5);
\draw[->,dashed] (4) to node[above,la,xshift=2pt]{$H$} (5);
\end{tz}

First, we construct $H$. For $0\leq i\leq m$, we set $H(i)\coloneqq G(i)$, for $0< j-i<m$, we set
\[ H_{i,j}\coloneqq G_{i,j}\colon \Hom_\cQ(i,j)=\Hom_\cP(i,j)\to \Hom_\cC(G(i),G(j)), \]
and we set
\[ H_{0,m}\coloneqq K_{0,1}\colon \Hom_\cQ(0,m)\to \Hom_\cC(G(0),G(m)). \]
The maps $H_{i,j},H_{j,k},H_{i,k}$ are compatible with composition for all $0\leq i<j<k\leq m$ with $k-i<m$ since the corresponding maps of $G$ do. It remains to show that $H_{0,i},H_{i,m},H_{0,m}$ are compatible with composition for all $0\leq i\leq m$. For $0\leq i\leq m$ we have that the following diagram commutes,
\begin{tz}
\node[](1) {$\Hom_\cQ(0,i)\times \Hom_\cQ(i,m)$}; 
\node[below of=1,xshift=-1.5cm](2) {$\Hom_\cP(0,i)\times \Hom_\cP(i,m)$}; 
\draw[d] (1) to (2); 
\node[right of=1,xshift=5cm](3) {$\Hom_\cQ(0,m)$}; 
\node[below of=3,xshift=-1.5cm](4) {$\Hom_\cP(0,m)$}; 

\draw[->] (1) to node[above,la]{$\circ_{0,i,m}$} (3);
\draw[->] (2) to node[above,la]{$\circ_{0,i,m}$} (4);
\draw[->] (4) to node[left,la,xshift=-3pt]{$F_{0,m}$} (3);

\node[below of=2,xshift=1.5cm](5) {$\Hom_\cC(G(0),G(i))\times \Hom_\cC(G(i),G(m))$}; 
\node[right of=5,xshift=5cm](6) {$\Hom_\cC(G(0),G(m))$}; 
\draw[->] (4) to node[left,la]{$G_{0,m}$} (6);
\draw[->] (2) to node[left,la]{$H_{0,i}\times H_{i,m}=G_{0,i}\times G_{i,m}$} (5); 
\draw[->] (5) to node[below,la]{$\circ_{G(0),G(i),G(m)}$} (6);
\draw[->] (3) to node[right,la]{$K_{0,1}=H_{0,m}$} (6);
\end{tz}
where the top rectangle commutes by compatibility of $F$ with composition, the bottom one by compatibility of $G$ with composition, and the right-hand triangle since $G\circ \iota_{0,m}=K_{0,1}\circ \Sigma F_{0,m}$. This shows that $H_{0,i},H_{i,m},H_{0,m}$ are compatible with composition for all $0\leq i\leq m$.
Moreover, observe that $H$ is the unique $\Thnsset$-enriched functor with the desired properties. This shows that $\cQ$ is the pushout
\[\cQ\cong\cP\amalg_{\Sigma\Hom_{\cP}(0,m)}\Sigma\Hom_{\cQ}(0,m).\]

Finally, the ``moreover'' part follows directly from the facts that, if $F_{0,m}$ is a (trivial) cofibration in $\MSThnsset$, then $\Sigma F_{0,m}$ is a (trivial) cofibration in $\MSThncat$ by \cref{susp:prestrivcof}, and that (trivial) cofibrations are closed under pushout. 
\end{proof}

\begin{notation}
For $m\geq 0$ and $X\in\Thnsset$, we denote by $\Sigma_mX$ the pushout of $m$ copies of $\Sigma X$ along consecutive sources and targets:
\[\Sigma_mX\coloneqq\Sigma X\amalg_{[0]} \ldots \amalg_{[0]} \Sigma X.\]
By convention $\Sigma_0 X$ is the terminal enriched category $[0]$. 
This construction extends to a functor $\Sigma_m\colon \Thnsset\to \Thncat$.
\end{notation}

The $\Thnsset$-enriched category $\Sigma_m X$ admits the following description.

\begin{prop} \label{homofSigmam}
Let $m\geq 0$ and $X\in \Thnsset$. Then the $\Thnsset$-enriched category $\Sigma_m X$ is the directed $\Thnsset$-enriched category such that:
\begin{itemize}[leftmargin=0.6cm]
    \item its set of objects $\Ob(\Sigma_m X)$ is $\{0,1,\ldots,m\}$,
    \item for $0\leq i<j\leq m$, the hom $\Thn$-space is $\Hom_{\Sigma_m X}(i,j)= X^{\times(j-i)}$,
    \item for $0\leq i< j< k\leq m$, the composition map is given by
  \begin{tz}
    \node[](1) {$\Hom_{\Sigma_m X}(i,j)\times \Hom_{\Sigma_m X}(j,k)=X^{\times (j-i)}\times X^{\times (k-j)}$}; 
    \node[below of=1, xshift=.42cm](2) {$\Hom_{\Sigma_m X}(i,k)=X^{\times (k-i)}$};
    \punctuation{2}{.};
    
    \draw[->] ($(1.south)-(.5cm,0)$) to node[left,la]{$\circ_{i,j,k}$} ($(2.north)-(.92cm,0)$);
    \draw[->] ($(1.south)+(1.6cm,0)$) to node[right,la]{$\cong$} ($(2.north)+(1.18cm,0)$);
    \end{tz}
\end{itemize}
\end{prop}

\subsection{Weakly enriched model structures for \texorpdfstring{$(\infty,n)$}{(infinity,n)}-categories} \label{subsec:MSPcat}

Let $\pcatThn$ denote the full subcategory of $\sThnsset$ spanned by those $(\DThn)$-spaces $W$ such that $W_0$ is discrete, i.e., such that $W_0$ in the image of $\set\hookrightarrow \Thnsset$. As also mentioned in \cite[\textsection7]{br1}, one sees that the inclusion $I\colon \pcatThn\to \sThnsset$ admits a left adjoint $L$, so there is an adjunction
\begin{tz}
\node[](2) {$\sThnsset$}; 
\node[right of=2,xshift=2.7cm](3) {$\pcatThn$}; 
\punctuation{3}{.};

\draw[->] ($(3.west)-(0,5pt)$) to node[below,la]{$I$} ($(2.east)-(0,5pt)$);
\draw[->] ($(2.east)+(0,5pt)$) to node[above,la]{$L$} ($(3.west)+(0,5pt)$);

\node[la] at ($(2.east)!0.5!(3.west)$) {$\bot$};
\end{tz}

In \cite{br1}, Bergner--Rezk construct two model structures on the category $\pcatThn$: the ``projective-like'' and the ``injective-like'' model structures. Here, we denote these two model structures by $\pcatproj$ and $\pcatinj$. As shown in \cite[Proposition 7.1]{br1}, these model structures are Quillen equivalent via the identity functor.

\begin{prop}\label{prop:idQE}
The adjunction
\begin{tz}
\node[](1) {$\pcatproj$}; 
\node[right of=1,xshift=3.6cm](2) {$\pcatinj$}; 

\draw[->] ($(2.west)-(0,5pt)$) to node[below,la]{$\id$} ($(1.east)-(0,5pt)$);
\draw[->] ($(1.east)+(0,5pt)$) to node[above,la]{$\id$} ($(2.west)+(0,5pt)$);

\node[la] at ($(1.east)!0.5!(2.west)$) {$\bot$};
\end{tz}
is a Quillen equivalence. 
\end{prop}

We now describe the main features of the injective-like model structure $\pcatinj$: the fibrant objects, a set of generating cofibrations, a fibrant replacement, and weak equivalences between fibrant objects. Let $\injsThnsset$ denote the injective model structure on the category $(\Thnsset)^{\Dop}\cong \sThnsset$ of simplicial objects in $\MSThnsset$.

\begin{recall} \label{recall:fibrant}
An object $W$ is fibrant in $\pcatinj$ if $W$ is fibrant in $\injsThnsset$ and the Segal map 
\[ W_m\to W_1\times^{(h)}_{W_0}\ldots \times^{(h)}_{W_0} W_1 \]
is a weak equivalence in $\MSThnsset$, for all $m\geq 1$. Here, the ordinary pullbacks are homotopy pullbacks because they are taken over the discrete object $W_0$ (see \cite[\textsection 4.1]{br1}).
\end{recall}

Recall that $L\colon \sThnsset\to \pcatThn$ denotes the left adjoint functor to the inclusion.

\begin{notation}
Let $A\to B$ and $X\to Y$ be two maps in a presheaf category. We denote by 
$(A\to B)\widehat{\times} (X\to Y)$ the pushout-product map
\[ (A\to B)\widehat{\times} (X\to Y)\coloneqq (A\times Y\amalg_{A\times X} B\times X\to B\times Y). \]
\end{notation}

\begin{recall} \label{rem:gencof}
By \cite[\textsection 6.1]{br1}, a set of generating cofibrations for the injective-like model structure $\pcatinj$ is given by the set containing the map
\[ \emptyset\to \repD[0] \]
and all maps of the form
\[ L((\partial\repD\hookrightarrow \repD)\widehat{\times} (\partial\repThn\hookrightarrow \repThn)\widehat{\times}(\partial \repS\hookrightarrow \repS)) \]
for $m\geq 1, \defThn\in \Thn, k\geq 0$.
\end{recall}

\begin{recall} \label{rem:genanodyne}
Using standard model categorical techniques, we see that a fibrant replacement functor 
\[(-)^{\fib}\colon \pcatThn\to \pcatThn\]
for the injective-like model structure $\pcatinj$ can be realized by running the small object argument to the set containing all maps of the form
\[L((\Sp  \hookrightarrow \repD)\widehat{\times} (\partial\repThn\hookrightarrow \repThn)\widehat{\times}(\partial \repS\hookrightarrow \repS))\]
for $m\geq 1, \defThn\in \Thn, k\geq 0$ and all maps of the form
\[ L((\partial\repD\hookrightarrow \repD)\widehat{\times} (X\hookrightarrow Y)) \]
for $m\geq 1$, $X\hookrightarrow Y\in \cJ$, where $\cJ$ is a set of generating trivial cofibrations for $\MSThnsset$. This is briefly mentioned in \cite[\textsection 6.7]{br1} and is discussed explicitly in \cite[\textsection 5]{br1} for the case $n=1$. In particular, for $W\in \pcatThn$, the fibrant replacement map $W\to W^{\fib}$ is a transfinite composition of pushouts of the above maps.
\end{recall}

The notion of weak equivalences in $\pcatinj$ relies on the notion of \emph{Dwyer-Kan equivalences} from \cite[\textsection 3.12]{br1}, which are in turn phrased in terms of the homotopy category and mapping objects for objects of $\pcatinj$. We briefly recall these.

\begin{defn}
Let $W$ be a fibrant object in $\pcatinj$. For $a,b\in W_0$, the \emph{mapping $\Thn$-space} $\Map_W(a,b)$ is the following pullback in $\Thnsset$. 
\begin{tz}
\node[](1) {$\Map_W(a,b)$}; 
\node[below of=1](2) {$\repS[0]$}; 
\node[right of=1,xshift=1.5cm](3) {$W_1$}; 
\node[below of=3](4) {$W_0\times W_0$}; 
\pullback{1};

\draw[->] (1) to (3);
\draw[->] (1) to (2);
\draw[->] (3) to (4);
\draw[->] (2) to node[below,la]{$(a,b)$} (4);
\end{tz}
\end{defn}

The following description of the homotopy category for an object in $\pcatinj$ can be extracted from \cite[Lemma 7.5]{br1} and a similar argument to \cref{prop:homofHoC}.

\begin{defn}
Let $W$ be a fibrant object in $\pcatinj$. The \emph{homotopy category} of~$W$ is the category $\Ho W$ such that
\begin{itemize}[leftmargin=0.6cm]
    \item its set of objects $\Ob(\Ho W)$ is $W_{0}$,
    \item for $a,b\in W_0$, its hom set is given by $(\Ho W)(a,b)\coloneqq \pi_0 (\Map_{W}(a,b)_{[0]})$,
    \item composition comes from the Segal maps.
\end{itemize}
\end{defn}

The weak equivalences between fibrant objects have a similar flavor to the weak equivalences in the enriched setting and are given by the \emph{Dwyer-Kan equivalences} from \cite[Definition 3.15]{br1}.

\begin{defn}
A map $f\colon W\to Z$ between fibrant objects in $\pcatinj$ is a \emph{Dwyer-Kan equivalence} if the induced functor $\Ho W\to \Ho Z$ is an equivalence of categories and, for all $a,b\in W_0$, the induced map
\[\Map_W(a,b)\to \Map_Z(fa,fb)\]
is a weak equivalence in $\MSThnsset$.
\end{defn}

Having discussed a construction for a fibrant replacement, and having fixed the weak equivalences between fibrant objects, the weak equivalences between ordinary objects are then enforced.  

\begin{recall} \label{weinpcat}
    A map $f\colon W\to Z$ in $\pcatThn$ (with $W,Z$ not necessarily fibrant) is a \emph{weak equivalence} in $\pcatinj$ if and only if the induced map $f^{\fib}\colon W^\fib\to Z^\fib$ is a Dwyer-Kan equivalence.
\end{recall}

\subsection{The strict nerve} \label{subsec:strictnerve}

There is a canonical inclusion $N\colon \Thncat\to \pcatThn$ that admits a left adjoint $c\colon \pcatThn\to \Thncat$.

At a $\Thnsset$-enriched category $\cC$, the strict nerve $N\cC$ is the $(\DThn)$-space given at $m= 0$ by $(N\cC)_0=\Ob\cC$ -- the set of objects of $\cC$ seen as an object in $\Thnsset$ -- and at $m\geq 1$ by the object in $\Thnsset$
\begin{align*}
    (N\cC)_m &\textstyle\coloneqq\coprod_{c_0,\ldots,c_m\in \Ob\cC}\Hom_\cC(c_0,c_1)\times\Hom_\cC(c_1,c_2)\times\ldots\times \Hom_\cC(c_{m-1},c_m) \\
    &\cong \Mor\cC\times_{\Ob\cC}\Mor\cC\times_{\Ob\cC}\ldots\times_{\Ob\cC}\Mor\cC,
\end{align*}
where $\Mor\cC$ is the object of $\Thnsset$ given by $\Mor\cC\coloneqq\coprod_{c_0,c_1\in \Ob\cC}\Hom_{\cC}(c_0,c_1)$.

The following appears as \cite[Theorem 7.6]{br1}. 

\begin{prop} \label{projQuillen}
The adjunction
\begin{tz}
\node[](1) {$\pcatproj$}; 
\node[right of=1,xshift=3.2cm](2) {$\MSThncat$}; 

\draw[->] ($(2.west)-(0,5pt)$) to node[below,la]{$N$} ($(1.east)-(0,5pt)$);
\draw[->] ($(1.east)+(0,5pt)$) to node[above,la]{$c$} ($(2.west)+(0,5pt)$);

\node[la] at ($(1.east)!0.5!(2.west)$) {$\bot$};
\end{tz}
is a Quillen equivalence. 
\end{prop}

However, the following example shows that the analog statement fails when replacing the projective with the injective model structure.

\begin{rmk}
The adjunction
\begin{tz}
\node[](1) {$\pcatinj$}; 
\node[right of=1,xshift=3.1cm](2) {$\MSThncat$}; 

\draw[->] ($(2.west)-(0,5pt)$) to node[below,la]{$N$} ($(1.east)-(0,5pt)$);
\draw[->] ($(1.east)+(0,5pt)$) to node[above,la]{$c$} ($(2.west)+(0,5pt)$);

\node[la] at ($(1.east)!0.5!(2.west)$) {$\bot$};
\end{tz}
is \emph{not} a Quillen pair. Indeed, given a fibrant $\MSThnsset$-enriched category $\cC$, \cref{ex:SigmaBI} shows that the nerve $N\cC$ is generally not fibrant in $\injsThnsset$. 
\end{rmk}

\begin{ex} \label{ex:SigmaBI}
Let $X$ be a fibrant object in $\MSThnsset$ that is not in the image of the inclusion $\Thnset\hookrightarrow\Thnsset$. The $\Thnsset$-enriched category $\Sigma X$ is by construction fibrant in $\MSThncat$, however its strict nerve $N\Sigma X$ is not fibrant in $\injsThnsset$. To see this, we first observe that the model structure $\injsThnsset$ is enriched over $\MSThnsset$ (see e.g.~\cite[Theorem 5.4]{Moserinj}), and we denote by $\Hom_{\sThnsset}(-,-)$ its hom $\Thn$-space functor. Now, the map $\partial\repD[2]\hookrightarrow \repD[2]$ is a cofibration in $\injsThnsset$, but the map
\[ \Hom_{\sThnsset}(\repD[2],N\Sigma X)\to\Hom_{\sThnsset}(\partial\repD[2],N\Sigma X),\]
is isomorphic to the map
\[ \repS[0]\amalg X\amalg X\amalg\repS[0]\to \repS[0]\amalg (X\times X)\amalg (X\times X)\amalg\repS[0],\]
induced by the diagonal map of the non-discrete $\Thn$-space $X$. As the diagonal map of a $\Thn$-space $X$ is a fibration in $\injThnspace$ if and only if the $\Thn$-space $X$ is discrete, the above map is not a fibration in $\injThnspace$ and hence also not a fibration in $\MSThnsset$. This contradicts the fact that the model structure  $\injsThnsset$ is enriched in $\MSThnsset$.
\end{ex}

\subsection{Diagonal model structures} \label{subsec:MSdiag}

Now consider the diagonal functor $\delta\colon \Delta\to \Delta\times \Delta$ given by sending $[\defS]\mapsto ([\defS],[\defS])$ and either projection $\pi\colon \Delta \times \Delta \to \Delta$. These induce adjunctions 
\begin{tz}
\node[](1) {$\SSset$}; 
\node[right of=1,xshift=1.2cm](2) {$\sset$}; 

\draw[->] ($(2.west)-(0,5pt)$) to node[below,la]{$\delta_*$} ($(1.east)-(0,5pt)$);
\draw[->] ($(1.east)+(0,5pt)$) to node[above,la]{$\delta^*$} ($(2.west)+(0,5pt)$);

\node[la] at ($(1.east)!0.5!(2.west)$) {$\bot$};

 \node[right of=2,xshift=1cm](1) {$\sset$}; 
 \node[right of=1,xshift=1.2cm](2) {$\SSset$}; 

 \draw[->] ($(2.west)-(0,5pt)$) to node[below,la]{$\pi_*$} ($(1.east)-(0,5pt)$);
 \draw[->] ($(1.east)+(0,5pt)$) to node[above,la]{$\pi^*$} ($(2.west)+(0,5pt)$);

 \node[la] at ($(1.east)!0.5!(2.west)$) {$\bot$};
 \end{tz}
 where $\SSset$ is the category of bisimplicial sets. We think of both simplicial directions in $\SSset$ as spacial directions.

We now lift these adjunctions to Quillen equivalences. Let $\MSdiag$ be the diagonal model structure on $\SSset$ from \cite[Theorem~2.11]{rasekh2023left}, in which the cofibrations are the monomorphisms and the weak equivalences are created by the functor $\delta^*\colon\SSset\to \MSspace$. By construction, it is a localization of the injective model structure $(\MSspace)^{\Dop}_\inj$. By \cite[Theorem~2.13]{rasekh2023left} we have the following result.

 \begin{prop} \label{deltapiQE}
 The adjunctions 
 \begin{tz}
\node[](1) {$\MSdiag$}; 
\node[right of=1,xshift=1.8cm](2) {$\MSspace$}; 

\draw[->] ($(2.west)-(0,5pt)$) to node[below,la]{$\delta_*$} ($(1.east)-(0,5pt)$);
\draw[->] ($(1.east)+(0,5pt)$) to node[above,la]{$\delta^*$} ($(2.west)+(0,5pt)$);

\node[la] at ($(1.east)!0.5!(2.west)$) {$\bot$};

 \node[right of=2,xshift=1.5cm](1) {$\MSspace$}; 
 \node[right of=1,xshift=1.8cm](2) {$\MSdiag$}; 

 \draw[->] ($(2.west)-(0,5pt)$) to node[below,la]{$\pi_*$} ($(1.east)-(0,5pt)$);
 \draw[->] ($(1.east)+(0,5pt)$) to node[above,la]{$\pi^*$} ($(2.west)+(0,5pt)$);

 \node[la] at ($(1.east)!0.5!(2.west)$) {$\bot$};
 \end{tz}
are Quillen equivalences.
 \end{prop}

They induce by post-composition adjunctions
\begin{tz}
\node[](1) {$\SSset^{\Thnop}$}; 
\node[right of=1,xshift=1.9cm](2) {$\sset^{\Thnop}$}; 

\draw[->] ($(2.west)-(0,5pt)$) to node[below,la]{$(\delta_*)_*$} ($(1.east)-(0,5pt)$);
\draw[->] ($(1.east)+(0,5pt)$) to node[above,la]{$\diag\coloneqq (\delta^*)_*$} ($(2.west)+(0,5pt)$);

\node[la] at ($(1.east)!0.5!(2.west)$) {$\bot$};

 \node[right of=2,xshift=1.5cm](1) {$\sset^{\Thnop}$}; 
 \node[right of=1,xshift=1.9cm](2) {$\SSset^{\Thnop}$}; 

 \draw[->] ($(2.west)-(0,5pt)$) to node[below,la]{$(\pi_*)_*$} ($(1.east)-(0,5pt)$);
 \draw[->] ($(1.east)+(0,5pt)$) to node[above,la]{$\iota\coloneqq (\pi^*)_*$} ($(2.west)+(0,5pt)$);

 \node[la] at ($(1.east)!0.5!(2.west)$) {$\bot$};
 \end{tz}
 
 We denote by $(\MSdiag)^{\Thnop}_{\inj}$ the injective model structure on the category of $\Thn$ -presheaves valued in $\MSdiag$. As a consequence of \cite[Remark~A.2.8.6]{htt} and \cref{deltapiQE}, we obtain: 

 \begin{prop}\label{diagiotaQEinj}
 The adjunctions
 \begin{tz}
\node[](1) {$(\MSdiag)^{\Thnop}_{\inj}$}; 
\node[right of=1,xshift=2.8cm](2) {$\injThnspace$}; 

\draw[->] ($(2.west)-(0,5pt)$) to node[below,la]{$(\delta_*)_*$} ($(1.east)-(0,5pt)$);
\draw[->] ($(1.east)+(0,5pt)$) to node[above,la]{$\diag$} ($(2.west)+(0,5pt)$);

\node[la] at ($(1.east)!0.5!(2.west)$) {$\bot$};

\node[right of=2,xshift=2cm](1) {$\injThnspace$}; 
\node[right of=1,xshift=2.8cm](2) {$(\MSdiag)^{\Thnop}_{\inj}$}; 

\draw[->] ($(2.west)-(0,5pt)$) to node[below,la]{$(\pi_*)_*$} ($(1.east)-(0,5pt)$);
\draw[->] ($(1.east)+(0,5pt)$) to node[above,la]{$\iota$} ($(2.west)+(0,5pt)$);

\node[la] at ($(1.east)!0.5!(2.west)$) {$\bot$};
\end{tz}
are Quillen equivalences.
 \end{prop}
 
 We denote by $\MSThndiag$ the localization of the model structure $(\MSdiag)^{\Thnop}_{\inj}$ with respect to the maps in $S_{\CSThn}$ from \cref{subsec:MSn-1}. As a consequence of \cite[Theorem 3.3.20(1)(b)]{Hirschhorn} and \cref{diagiotaQEinj}, we have:

\begin{prop}\label{diagiotaQE}
 The adjunctions
 \begin{tz}
\node[](1) {$\MSThndiag$}; 
\node[right of=1,xshift=2.5cm](2) {$\MSThnsset$}; 

\draw[->] ($(2.west)-(0,5pt)$) to node[below,la]{$(\delta_*)_*$} ($(1.east)-(0,5pt)$);
\draw[->] ($(1.east)+(0,5pt)$) to node[above,la]{$\diag$} ($(2.west)+(0,5pt)$);

\node[la] at ($(1.east)!0.5!(2.west)$) {$\bot$};

\node[right of=2,xshift=2cm](1) {$\MSThnsset$}; 
\node[right of=1,xshift=2.5cm](2) {$\MSThndiag$}; 

\draw[->] ($(2.west)-(0,5pt)$) to node[below,la]{$(\pi_*)_*$} ($(1.east)-(0,5pt)$);
\draw[->] ($(1.east)+(0,5pt)$) to node[above,la]{$\iota$} ($(2.west)+(0,5pt)$);

\node[la] at ($(1.east)!0.5!(2.west)$) {$\bot$};
\end{tz}
are Quillen equivalences.
 \end{prop}

\section{The homotopy coherent categorification and its description}

This section is devoted to constructing the homotopy coherent categorification-nerve adjunction
\begin{tz}
\node[](1) {$\pcatThn$}; 
\node[right of=1,xshift=2.6cm](2) {$\Thncat$}; 

\draw[->] ($(2.west)-(0,5pt)$) to node[below,la]{$\Nh$} ($(1.east)-(0,5pt)$);
\draw[->] ($(1.east)+(0,5pt)$) to node[above,la]{$\Ch$} ($(2.west)+(0,5pt)$);

\node[la] at ($(1.east)!0.5!(2.west)$) {$\bot$};
\end{tz}
and describing the left adjoint $\Ch$. To this end, building on work by Dugger--Spivak, we introduce the notion of a \emph{$1$-ordered} simplicial set in \cref{subsec:necklace} and study its category of necklaces. In \cref{subsec:classicalhtpynerve} we recall the classical homotopy coherent categorification $\CL$ by Cordier--Porter, and the description of its hom spaces in terms of necklaces. In \cref{subsec:htpycohcat} we define the desired functor $\Ch$ using $\CL$, and in \cref{subsec:homsofC} (resp.~\cref{subsec:HoofC}) we give explicit formulas for the hom $\Thn$-spaces (resp.~the homotopy category) of the homotopy coherent categorification $\Ch$.

\subsection{Necklaces and 1-ordered simplicial sets}\label{subsec:necklace}

We recall the main terminology about necklaces, introduced in \cite[\textsection 3]{Dugger}. 

A \emph{necklace} is a simplicial set, i.e., an object in $\Dset$, given by a wedge of representables 
\[ T=\repD[m_1]\vee \ldots \vee \repD[m_t] \] 
obtained by gluing the final vertex $m_i\in \repD[m_i]$ to the initial vertex $0\in\repD[m_{i+1}]$ for all $1\leq i\leq t-1$. By convention, if $t>1$, then $m_i>0$ for all $1\leq i\leq t$. We say that $\repD[m_i]$ is a \emph{bead} of $T$, and an initial or a final vertex in some bead is a \emph{joint} of $T$. We write $B(T)$ for the set of beads of $T$; in particular, we have that $|B(T)|=t$.

We consider the necklace $T$ to be a bi-pointed simplicial set $(T,\alpha,\omega)$ where $\alpha$ is the initial vertex $\alpha=0\in \repD[m_0]\hookrightarrow T$ and $\omega$ is the final vertex $\omega=m_t\in \repD[m_t]\hookrightarrow  T$.
We write $\Nec$ for the full subcategory of the category $\Dset_{*,*}$ of bi-pointed simplicial sets spanned by the necklaces.

 Given a simplicial set $K$ and $a,b\in K_0$, we denote by $K_{a,b}$ the simplicial set bi-pointed at $(a,b)\colon \repD[0]\amalg\repD[0]\to K$. A \emph{necklace} in $K_{a,b}$ is a bi-pointed map $T\to K_{a,b}$. We denote by $\catnec{K}{a}{b}\coloneqq \Nec_{/K_{a,b}}$ the category of necklaces $T\to K_{a,b}$ in $K$ from $a$ to $b$, obtained as a full subcategory of the slice category ${\Dset_{*,*}}_{/K_{a,b}}$.

\begin{defn}
Let $K$ be a simplicial set and $a,b\in K_0$. A necklace
\[ f\colon T=\repD[m_1]\vee\ldots\vee \repD[m_t]\to K_{a,b}\]
is \emph{totally non-degenerate} if, for all $0\leq i \leq t$, the restriction of $f$ to the $i$-th bead
\[ \repD[m_i]\hookrightarrow \repD[m_1]\vee\ldots\vee \repD[m_k]=T \xrightarrow{f} K\]
is a non-degenerate $m_i$-simplex of $K$.
\end{defn}

We write $\tndnec{K}{a}{b}$ for the full subcategory of $\catnec{K}{a}{b}$ spanned by the totally non-degenerate necklaces. 

We now recall the notion of ordered simplicial sets presented in \cite[\textsection 3.1]{DuggerSpivakRigidification} and introduce the weaker notion of \emph{$1$-ordered} simplicial sets.

\begin{notation} 
Let $K$ be a simplicial set. Denote by $\preceq_K$ the relation on the set of $0$-simplices~$K_0$ given by $x\preceq_Ky$ if and only if there is a necklace of the form $f\colon\Sp=\repD[1]\vee\ldots\vee\repD[1]\to K$ such that $f(\alpha)=x$ and $f(\omega)=y$ for some $m\geq 0$.
\end{notation}

\begin{defn}
    A simplicial set $K$ is 
    \begin{itemize}[leftmargin=0.6cm]
        \item \emph{ordered} if the relation $\preceq_K$ is antisymmetric and the canonical map $K_m\to K_0^{\times (m+1)}$ is injective, for all $m\geq 1$,
        \item \emph{$1$-ordered} if the relation $\preceq_K$ is antisymmetric and, for all $m\geq 1$, the restriction of the Segal map to the set $K_m^{\mathrm{nd}}$ of non-degenerate $m$-simplices of $K$
        \[ K_m^{\mathrm{nd}}\subseteq K_m\to K_1\times_{K_0}\ldots \times_{K_0} K_1 \]
        is injective and, for every non-degenerate $m$-simplex $\repD\to K$, its restriction along the inclusion $\Sp\hookrightarrow \repD$ is a monomorphism $\Sp\hookrightarrow K$.
    \end{itemize}
\end{defn}

\begin{rmk} 
Note that the definition of ordered simplicial sets coincides with that of Dugger--Spivak from \cite[Definition 3.2]{DuggerSpivakRigidification}. 
\end{rmk}

\begin{lemma} \label{lem:orderedis1ordered}
    Every ordered simplicial set is $1$-ordered.
\end{lemma}

\begin{proof}
    Suppose that $K$ is an ordered simplicial set. For $m\geq 1$, consider the following commutative triangle
    \begin{tz}
        \node[](1) {$K_m^\mathrm{nd}\subseteq K_m$}; 
        \node[right of=1,xshift=2cm](3) {$K_1\times_{K_0}\ldots \times_{K_0} K_1$}; 
        \node[below of=3](4) {$K_0^{\times(m+1)}$}; 
        \draw[->] (1) to (4);
        \draw[->](1) to (3); 
        \draw[->](3) to (4);
    \end{tz}
    where the composite $K_m^\mathrm{nd}\subseteq K_m\to K_0^{\times (m+1)}$ is injective by assumption. Then the top map $K_m^{\mathrm{nd}}\subseteq K_m\to K_1\times_{K_0}\ldots \times_{K_0} K_1$ is also injective by cancellation of injective maps.

    Next, we show that, for every non-degenerate $m$-simplex $\repD\to K$, its restriction along $\Sp\hookrightarrow \repD$ is a monomorphism $\Sp\hookrightarrow K$. By the injectivity of the map $K_1\to K_0\times K_0$, it suffices to prove that its restriction along $\coprod_{m+1} \repD[0]\hookrightarrow\repD$ is a monomorphism $\coprod_{m+1} \repD[0]\hookrightarrow K$. We prove this by contraposition. 
    
    Let $\sigma\colon \repD\to K$ be an $m$-simplex whose restriction $(\sigma(0),\ldots,\sigma(m))\colon \coprod_{m+1}\repD[0]\to K$ is not a monomorphism. Then we have an ordered tuple $\sigma(0)\preceq_K \ldots \preceq_K \sigma(m)$ and, as $(\sigma(0),\ldots,\sigma(m))$ is not a monomorphism, there is $0\leq i\leq m-1$ such that $\sigma(i)= \sigma(i+1)$. Consider the $m$-simplex given by $\sigma\circ d^i\circ s^i\colon \repD\to K$. Then the image of $\sigma\circ d^i\circ s^i$ under $K_m\hookrightarrow K_0^{\times (m+1)}$ is also $(\sigma(0),\ldots,\sigma(m))$. Hence, by injectivity of $K_m\hookrightarrow K_0^{\times (m+1)}$, we get that $\sigma=\sigma\circ d^i\circ s^i$ is degenerate. 
\end{proof}

\begin{rmk} \label{examplesof1ordered}
By \cite[Lemma 3.3]{DuggerSpivakRigidification}, we have that every necklace is ordered and that every simplicial subset of an ordered simplicial set is ordered. Hence, it follows from \cref{lem:orderedis1ordered} that the simplicial sets $\repD$, $\partial \repD$, and $\Sp$, for $m\geq 0$, and all necklaces are $1$-ordered.
\end{rmk}

We now aim to characterize the totally non-degenerate necklaces of a $1$-ordered simplicial set as the monomorphisms. For this, we first need the following. 

\begin{lemma} \label{nondegsimplexareinjective}
    Let $K$ be a $1$-ordered simplicial set. Then an $m$-simplex $\sigma\colon \repD\to K$ is non-degenerate if and only if it is a monomorphism.
\end{lemma}

\begin{proof}
    We show that an $m$-simplex $\sigma\colon \repD\to K$ is degenerate if and only if it is not a monomorphism. First note that, if an $m$-simplex $\sigma\colon \repD\to K$ is degenerate, then $\sigma$ is not a monomorphism as it factors through a map $\repD\to \repD[m']$ with $m'<m$ that is not a monomorphism.

    Now, suppose that $\sigma\colon \repD\to K$ is not a monomorphism. We show that its restriction to $0$-simplices $(\sigma(0),\ldots,\sigma(m))\colon \coprod_{m+1}\repD[0]\to K$ is not a monomorphism, showing that the induced map $\Sp\hookrightarrow \repD\xrightarrow{\sigma} K$ is also not a monomorphism. As $K$ is $1$-ordered, this implies that $\sigma$ is degenerate. 

    Since $\sigma\colon \repD\to K$ is not a monomorphism, we can choose $0\leq m'\leq m$ the smallest integer such that there are monomorphisms $\alpha,\beta\colon \repD[m']\hookrightarrow \repD$ with $\alpha\neq \beta$ and $\sigma\circ \alpha=\sigma\circ \beta$. Suppose by contradiction that $m'\geq 1$. As $\sigma\circ \alpha=\sigma\circ \beta$, we have that $\sigma(\alpha(i))=\sigma(\beta(i))$ for all $0\leq i\leq m'$. As $\sigma$ is injective on $0$-simplices by minimality of $m'$, we get that $\alpha(i)=\beta(i)$ for all $0\leq i\leq m'$. Hence $\alpha,\beta$ are two $m'$-simplices of $\repD$ such that their restrictions to $0$-simplices are equal, and so we must have $\alpha=\beta$ as $\repD$ is an ordered simplicial set. This gives a contradiction and shows that $m'=0$, as desired. 
\end{proof}

 \begin{lemma} \label{losndsubss}
     Let $K$ be a $1$-ordered simplicial set and $x\in K_0$. Let $K_{\preceq x}$ and $K_{\succeq x}$ be the simplicial subsets of $K$ given at $m\geq 0$ by 
     \begin{align*}
     (K_{\preceq x})_m &= \{ \sigma\in K_m\mid \sigma(i)\preceq_K x \text{ for all } 0\leq i\leq m\}, \\(K_{\succeq x})_m &= \{ \sigma\in K_m\mid x\preceq_K\sigma(i) \text{ for all } 0\leq i\leq m\}. 
     \end{align*}
     Then the map $K_{\preceq x} \vee K_{\succeq x} \to K$ induced by the canonical inclusions is a monomorphism.
 \end{lemma}

\begin{proof}
    Since $K_{\preceq x}$ and $K_{\succeq x}$ are simplicial subsets of $K$, to establish the desired monomorphism, we only need to prove that for $m\geq 0$, except for the degenerate $m$-simplex at the $0$-simplex $x$, no $m$-simplex of $K$ lies in the image of both simplicial subsets. If such an $m$-simplex $\sigma$ of $K$ existed, then we must have $x\preceq_K \sigma(i) \preceq_K x$, for all $0\leq i\leq m$, and so $\sigma(i) = x$, for all $0\leq i\leq m$, by antisymmetry of $\preceq_K$. 

    Hence, in order to finish the proof it suffices to show that, for every $\sigma\colon\repD \to K$ such that $\sigma(i)=x$, for all $0\leq i\leq m$, then $\sigma$ is the degenerate $m$-simplex $\repD\to \repD[0]\xrightarrow{x} K$. If $m=0$, there is nothing to prove. Now, let $m\geq 1$. As $\sigma\colon \repD\to K$ is not a monomorphism and $K$ is $1$-ordered, by \cref{nondegsimplexareinjective}, we have that $\sigma$ is degenerate. Hence it factors as 
    \[ \sigma\colon \repD\to \repD[m']\xrightarrow{\tau} K \]
    for some $0\leq m'<m$. As $\tau(i)=x$ for all $0\leq i\leq m'$, then by induction $\tau$ is the degenerate $m'$-simplex $\repD[m']\to \repD[0]\xrightarrow{x} K$. Hence $\sigma$ is the degenerate $m$-simplex constant at $x$, as desired.
\end{proof}

\begin{lemma} \label{tndvsmono}
Let $K$ be a $1$-ordered simplicial set. Then a necklace $T\to K_{a,b}$ is totally non-degenerate if and only if it is a monomorphism. 
\end{lemma}

\begin{proof}
First, if a necklace $f\colon T\to K_{a,b}$ is not totally non-degenerate, then there is a bead $\repD[m_i]$ of $T$ such that the induced map 
\[ \repD[m_i]\hookrightarrow T\xrightarrow{f} K\]
is a non-degenerate $m_i$-simplex of $K$. Then the above map is not a monomorphism by \cref{nondegsimplexareinjective}, and so $f$ is also not a monomorphism. 

We now show that if a necklace $f\colon T\to K_{a,b}$ is totally non-degenerate, then it is a monomorphism. We do this by induction on the number of beads $t$ of $T$. If $t=1$, this follows directly from the definition of totally non-degenerate necklaces and \cref{nondegsimplexareinjective}.

Now, let $t>1$. We can write $T=T'\vee \repD[m_t]$, where $T'$ is a necklace with $t-1$ beads. As $f\colon T\to K$ is totally non-degenerate, so are the induced necklaces 
\[ T'\hookrightarrow T\xrightarrow{f} K \quad \text{and} \quad \repD[m_t]\hookrightarrow T\xrightarrow{f} K. \]
By induction, the above maps are monomorphisms. Then $f$ factors as the composite of two monomorphisms 
\[ T=T' \vee \repD[m_t] \hookrightarrow K_{\preceq f(i)} \vee K_{\succeq f(i)} \hookrightarrow K, \]
where $i$ is the last vertex of $T'$, and the second map is a monomorphism by \cref{losndsubss}. Hence $f$ is a monomorphism. 
\end{proof}

Using this characterization of totally non-degenerate necklaces in $1$-ordered simplicial sets and results by Dugger--Spivak, we show that the inclusion $\tndnec{K}{a}{b}\hookrightarrow \catnec{K}{a}{b}$ is final, in the sense of \cite[\textsection IX.3]{MacLane}.

\begin{rmk} \label{rem:DSepi}
    Let $K$ be a simplicial set and $a,b\in K_0$. As explained in the paragraph before Proposition 4.7 in \cite[\textsection 4]{DuggerSpivakRigidification}, for every necklace $T\to K_{a,b}$, there is a totally non-degenerate $\overline{T}\to K_{a,b}$ and an epimorphism of simplicial sets $T\to \overline{T}$ over $K_{a,b}$.
\end{rmk}

\begin{prop} \label{prop:final}
Let $K$ be a $1$-ordered simplicial set and $a,b\in K_0$. Then the inclusion functor
\[ J\colon \tndnec{K}{a}{b}\to \catnec{K}{a}{b} \]
is final. 
\end{prop}

\begin{proof}
We show that, for every necklace $U\to K_{a,b}$ in $\catnec{K}{a}{b}$, the comma category $U\downarrow J$ is non-empty and connected.

We first show that the category $U\downarrow J$ is non-empty. Using \cref{rem:DSepi} for the necklace $U\to K_{a,b}$, there is a totally non-degenerate necklace $\overline{U}\to K_{a,b}$ and an epimorphism $U\to \overline{U}$ over~$K_{a,b}$. This defines a map $U\to \overline{U}$ in $\catnec{K}{a}{b}$ from the given necklace $U\to K_{a,b}$ to a totally non-degenerate necklace $\overline{U}\to K_{a,b}$. Hence the comma category $U\downarrow J$ is non-empty.

We now show that the category $U\downarrow J$ is connected. Let $U\to T$ be a map in $\catnec{K}{a}{b}$ from the necklace $U\to K_{a,b}$ to a totally non-degenerate necklace $T\to K_{a,b}$. Using \cref{rem:DSepi} for the necklace $U\to T_{\alpha,\omega}$, there is 
a totally non-degenerate necklace $\overline{T}\to T_{\alpha,\omega}$ and an epimorphism $U\to \overline{T}$ over $T_{\alpha,\omega}$. By \cref{examplesof1ordered}, the necklace $T$ is $1$-ordered, so by \cref{tndvsmono} the map $\overline{T}\to T$ is a monomorphism of simplicial sets. Moreover, the simplicial set $K$ is $1$-ordered by assumption, so by \cref{tndvsmono} the map $T\to K$ is a monomorphism, too. Hence the composite $\overline{T}\hookrightarrow T\hookrightarrow K_{a,b}$ is also a monomorphism, and by \cref{tndvsmono} it defines a totally non-degenerate necklace $\overline{T}\to K_{a,b}$. By \cite[Proposition~4.7(b)]{DuggerSpivakRigidification}, there is a map $\overline{U}\to \overline{T}$ making the following diagram commute.
\begin{tz}
\node[](1) {$U$}; 
\node[right of=1](2) {$\overline{T}$}; 
\node[right of=2](5) {$T$}; 
\node[below of=1](3) {$\overline{U}$}; 
\node[below of=2](4) {$K_{a,b}$}; 
\draw[->>] (1) to (2); 
\draw[right hook->] (2) to (4); 
\draw[->>] (1) to (3);
\draw[right hook->] (3) to (4);
\draw[->,dashed] (3) to (2);
\draw[right hook->] (2) to (5); 
\draw[right hook->] (5) to (4);
\draw[->,bend left] (1) to (5);
\end{tz}
Then the composite $\overline{U}\to \overline{T}\to T$ defines a map in $U\downarrow J$ from the totally non-degenerate necklace $\overline U\to K_{a,b}$ to the totally non-degenerate necklace $T\to K_{a,b}$, which shows that the comma category $U\downarrow J$ is connected.
\end{proof}

\subsection{The classical homotopy coherent categorification-nerve} \label{subsec:classicalhtpynerve}

We first recall the homotopy coherent nerve construction by Cordier--Porter \cite{CordierPorterHomotopyCoherent}.

\begin{defn} \label{defn:CL}
Let $m\geq 0$. Define $\CL[m]$ to be the directed $\sset$-enriched category such that 
\begin{itemize}[leftmargin=0.6cm]
    \item its set of objects $\Ob(\CL[m])$ is $\{0,1,\ldots,m\}$, 
    \item for $0\leq i<j\leq m$, the hom space is 
    \[ \textstyle\Hom_{\CL[m]}(i,j)\coloneqq \prod_{[i+1,j-1]}\repS[1], \] 
    where $[i+1,j-1]\subseteq \{0,1,\ldots,m\}$ denotes the interval between $i+1$ and $j-1$,
    \item for $0\leq i< j< k\leq m$, the composition map is given by
    \end{itemize}
    \begin{tz}
    \node[](1) {$\Hom_{\CL[m]}(i,j)\times \Hom_{\CL[m]}(j,k)=\prod_{[i+1,j-1]}\repS[1]\times \prod_{[j+1,k-1]}\repS[1]$}; 
    \node[below of=1,xshift=-.05cm](2) {$\Hom_{\CL[m]}(i,k)=\prod_{[i+1,k-1]}\repS[1]$};
    \punctuation{2}{.};
    
    \draw[->] ($(1.south)-(1.5cm,0)$) to node[left,la]{$\circ_{i,j,k}$} ($(2.north)-(1.45cm,0)$);
    \draw[->] ($(1.south)+(1.5cm,0)$) to node[right,la]{$\prod_{[i+1,j-1]}\id_{\repS[1]}\times \langle 1\rangle \times \prod_{[j+1,k-1]} \id_{\repS[1]}$} ($(2.north)+(1.55cm,0)$);
    \end{tz}
\end{defn}

\begin{rmk} \label{CLoncofaces}
    By \cite[Definition 1.1.5.3]{htt}, the assignment $[m]\mapsto \CL[m]$ extends to a cosimplicial object $\Delta\to \sCat$. In particular, by unpacking definitions, the coface map $d^\ell\colon [m-1]\to [m]$ for $0\leq \ell\leq m$ is sent to the $\sset$-enriched functor $\CL d^\ell\colon \CL[m-1]\to \CL[m]$ given on objects by 
    \[ d^\ell\colon \{0,1,\ldots,m-1\} \to \{0,1,\ldots, m\} \]
    and on hom spaces, for $0\leq i<j\leq m-1$, by the identity if $j< \ell$ or $i\geq \ell$, and by 
   \begin{tz}
    \node[](1) {$\Hom_{\CL[m-1]}(i,j)\cong \prod_{[i+1,j]\setminus \{\ell\}}\repS[1]$}; 
    \node[below of=1,xshift=-.4cm](2) {$\Hom_{\CL[m]}(i,j+1)\cong\prod_{[i+1,j]}\repS[1]$};
    \punctuation{2}{.};
    
    \draw[->] ($(1.south)-(1.4cm,0)$) to node[left,la]{$(\CL d^\ell)_{i,j}$} ($(2.north)-(1cm,0)$);
    \draw[->] ($(1.south)+(1.3cm,0)$) to node[right,la]{$(\prod_{[i+1,\ell-1]}\id_{\repS[1]})\times \langle 0\rangle\times (\prod_{[\ell+1,j]}\id_{\repS[1]})$} ($(2.north)+(1.7cm,0)$);
    \end{tz}
    if $i< \ell \leq j$.
\end{rmk}

By taking the left Kan extension of the assignment $\Delta\to \sCat$ given by $[m]\mapsto \CL[m]$, we obtain an adjunction
\begin{tz}
\node[](1) {$\Dset$}; 
\node[right of=1,xshift=1.6cm](2) {$\sCat$}; 
\punctuation{2}{.};

\draw[->] ($(2.west)-(0,5pt)$) to node[below,la]{$\NL$} ($(1.east)-(0,5pt)$);
\draw[->] ($(1.east)+(0,5pt)$) to node[above,la]{$\CL$} ($(2.west)+(0,5pt)$);

\node[la] at ($(1.east)!0.5!(2.west)$) {$\bot$};
\end{tz}

Dugger--Spivak provide in \cite[Proposition~4.3]{DuggerSpivakMapping} the following explicit description of the hom spaces of the categorification $\CL$ in terms of necklaces.

 \begin{thm} \label{DS:computationhoms}
 Let $K$ be a simplicial set and $a,b\in K_0$. Then there is a natural isomorphism in $\sset$
 \[ \Hom_{\CL K}(a,b)\cong \colim_{T\in \catnec{K}{a}{b}} \Hom_{\CL T}(\alpha,\omega). \]
 \end{thm}

 In the case of $1$-ordered simplicial sets, the above result refines to a description in terms of totally non-degenerate necklaces.
 
 \begin{cor} \label{DS:computations1ordered}
Let $K$ be a $1$-ordered simplicial set and $a,b\in K_0$. Then there is a natural isomorphism in $\sset$
 \[ \Hom_{\CL K}(a,b)\cong \colim_{T\in \tndnec{K}{a}{b}} \Hom_{\CL T}(\alpha,\omega). \]
 \end{cor}

 \begin{proof}
     This follows from \cref{prop:final,DS:computationhoms} together with \cite[Theorem IX.3.1]{MacLane}.
 \end{proof}

 We denote by $\Dset_{\CSThn[1]}$ Joyal's model structure on simplicial sets, in which the fibrant objects are the quasi-categories. Then, in the case of a quasi-category, \cite[Corollary~5.3]{DuggerSpivakMapping} shows that the hom spaces of its categorification $\CL$ are further related to its mapping spaces, as follows.

\begin{thm} \label{thm:zigzag2}
Let $K$ be a fibrant object in $\Dset_{\CSThn[1]}$ and $a,b\in K_0$. Then there is a natural zig-zag of weak equivalences in $\MSspace$ connecting the spaces
\[\Hom_{\CL K}(a,b)\sim\map_K(a,b), \]
where $\map_K(a,b)=\{a\}\times_K\times K^{\repD[1]}\times_K \{b\}$.
\end{thm}

\subsection{The homotopy coherent categorification-nerve} \label{subsec:htpycohcat}

The adjunction $\CL\dashv \NL$ from \cref{subsec:classicalhtpynerve} induces by post-composition an adjunction
\begin{tz}
\node[](1) {$(\Dset)^{\DThnop}$}; 
\node[right of=1,xshift=3.3cm](2) {$(\sCat)^{\DThnop}$}; 
    \punctuation{2}{.};

\draw[->] ($(2.west)-(0,5pt)$) to node[below,la]{$\NL_*$} ($(1.east)-(0,5pt)$);
\draw[->] ($(1.east)+(0,5pt)$) to node[above,la]{$\CL_*$} ($(2.west)+(0,5pt)$);

\node[la] at ($(1.east)!0.5!(2.west)$) {$\bot$};
\end{tz}

We consider the category $\Thnssset$ of $\Thn$-bi-spaces and also the category $\sThncat$ of $\Thnssset$-enriched categories. 

Recall that the category of $\Thnssset$-enriched categories can equivalently be seen as the full subcategory of $(\sCat)^{\DThnop}$ spanned by those functors $\cC\colon \ThnSop\to \sCat$ that are constant at the level of objects, i.e., such that $\Ob(\cC_{\defThn,\defS})=\Ob(\cC_{0,0})$ for all $\defThn\in \Thn$ and $\defS\geq 0$. The inclusion can be implemented in a similar way to \cite[\textsection3.6]{riehl2014categoricalhomotopy}.

Moreover, recall that $\Pcat(\Thnsset)$ is the full subcategory of $\sThnsset$ spanned by those $(\DThn)$-spaces $W$ such that $W_{0,\defThn,\defS}=W_{0,0,0}$ for all $\defThn\in \Thn$ and $\defS\geq 0$. Moreover, observe that there is an identification $\sThnsset\cong (\Dset)^{\DThnop}$, which sends $W\in \sThnsset$ to $\widehat W\colon \Thnop\times \Dop\to \set^{\Dop}$ given at $\defThn\in \Thn$ and $m,k\geq 0$ by $(\widehat W_{\defThn,\defS})_m\coloneqq W_{m,\defThn,\defS}$. So we can regard $\Pcat(\Thnsset)$ as a full subcategory of $(\Dset)^{\DThnop}$.

Then the above adjunction restricts to an adjunction
\begin{tz}
\node[](1) {$\pcatThn$}; 
\node[right of=1,xshift=2.7cm](2) {$\sThncat$}; 
    \punctuation{2}{.};

\draw[->] ($(2.west)-(0,5pt)$) to node[below,la]{$\NL_*$} ($(1.east)-(0,5pt)$);
\draw[->] ($(1.east)+(0,5pt)$) to node[above,la]{$\CL_*$} ($(2.west)+(0,5pt)$);

\node[la] at ($(1.east)!0.5!(2.west)$) {$\bot$};
\end{tz}

Next, recall the functor $\diag\colon \Thnssset\to \Thnsset$ from \cref{subsec:MSdiag}. Given that it is also a right adjoint functor, it preserves products, hence the adjunction $\diag\dashv (\delta_*)_*$ induces by base-change an adjunction between enriched categories
\begin{tz}
\node[](1) {$\sThncat$}; 
\node[right of=1,xshift=2.5cm](2) {$\Thncat$}; 
    \punctuation{2}{.};

\draw[->] ($(2.west)-(0,5pt)$) to node[below,la]{$((\delta_*)_*)_*$} ($(1.east)-(0,5pt)$);
\draw[->] ($(1.east)+(0,5pt)$) to node[above,la]{$\diag_*$} ($(2.west)+(0,5pt)$);

\node[la] at ($(1.east)!0.5!(2.west)$) {$\bot$};
\end{tz}

\begin{defn} \label{defadjunction}
We define the homotopy coherent categorification-nerve adjunction to be the following composite of adjunctions.
\begin{tz}
\node[](1) {$\pcatThn$}; 
\node[right of=1,xshift=2.7cm](2) {$\sThncat$}; 
\node[right of=2,xshift=2.5cm](3) {$\Thncat$}; 

\node at ($(1.west)-(6pt,.3pt)$) {$\Ch\colon$};
\node at ($(3.east)+(7pt,-1.5pt)$) {$\colon\! \Nh$};

\draw[->] ($(3.west)-(0,5pt)$) to node[below,la]{$((\delta_*)_*)_*$} ($(2.east)-(0,5pt)$);
\draw[->] ($(2.east)+(0,5pt)$) to node[above,la]{$\diag_*$} ($(3.west)+(0,5pt)$);

\node[la] at ($(1.east)!0.5!(2.west)$) {$\bot$};

\draw[->] ($(2.west)-(0,5pt)$) to node[below,la]{$\NL_*$} ($(1.east)-(0,5pt)$);
\draw[->] ($(1.east)+(0,5pt)$) to node[above,la]{$\CL_*$} ($(2.west)+(0,5pt)$);

\node[la] at ($(2.east)!0.5!(3.west)$) {$\bot$};
\end{tz} 
\end{defn}

In order to develop intuition on the action of $\Ch$, we compute here some of its values. 

\begin{ex}
    As a first example, we can see that, for $m\geq 0$, we have 
    \[ \Ch \repD= \CL [m], \]
    where $\CL [m]$ is the $\sset$-enriched category from \cref{defn:CL} seen as a $\Thnsset$-enriched category through base-change along the canonical inclusion $\sset\hookrightarrow \Thnsset$. 
   \begin{itemize}[leftmargin=0.6cm]
   \item When $m=0$, we get that $\Ch \repD[0]$ is the terminal $\Thnsset$-enriched category $[0]$. 
   \item When $m=1$, we get that $\Ch \repD[1]$ is the directed $\Thnsset$-enriched category with object set $\{0,1\}$ and hom $\Thn$-space $\Hom_{\Ch \repD[1]}(0,1)=\Delta[0]$ and so it is the $\Thnsset$-enriched category generated by the following data
    \begin{tz}
        \node[](0) {$0$}; 
        \node[right of=0](1) {$1$}; 
        \draw[->] (0) to node[above,la]{$f$} (1);
    \end{tz}
    \item When $m=2$, we get that $\Ch \repD[2]$ is the directed $\Thnsset$-enriched category with object set $\{0,1,2\}$ and hom $\Thn$-spaces 
    \[ \Hom_{\Ch \repD[2]}(0,1)=\Hom_{\Ch \repD[2]}(1,2)=\Delta[0] \quad \text{and} \quad \Hom_{\Ch \repD[2]}(0,2)=\Delta[1] \]
    and so it is the $\Thnsset$-enriched category generated by the following data
    \begin{tz}
        \node[](0) {$0$}; 
        \node[above right of=0](1) {$1$};
        \node[below right of=1](2) {$2$}; 
        \draw[->] (0) to node[left,la,yshift=3pt]{$f$} (1); 
        \draw[->] (1) to node[right,la,yshift=3pt]{$g$} (2); 
        \draw[->] (0) to node[below,la]{$h$} (2);
    \end{tz}
    together with a homotopy between $h$ and the composite $gf$.
    \end{itemize}
\end{ex}

\begin{ex}
    We also study the $\Thnsset$-enriched category $\Ch L(\repD[m]\times \repThn[1;0])$ for small values of $m\geq 0$, where we recall that $L\colon \sThnsset\to \pcatThn$ denotes the left adjoint functor to the inclusion. These can be computed using the techniques developed later in \cref{Section 3}.
    \begin{itemize}[leftmargin=0.6cm]
    \item When $m=0$, as $L(\repD[m]\times \repThn[1;0])=\repD[0]$, we get that $\Ch L(\repD[0]\times \repThn[1;0])=\Ch \repD[0]=[0]$. 
    \item When $m=1$, using \cref{lem:Sh1vsSigma1} applied to the case where $X=\repThn[1;0]$, we get that $\Ch L(\repD[1]\times \repThn[1;0])$ is the directed $\Thnsset$-enriched category $\Sigma \repThn[1;0]$ with object set $\{0,1\}$ and hom $\Thn$-space \[ \Hom_{\Ch L(\repD[1]\times \repThn[1;0])}(0,1)=\repThn[1;0] \]
    and so it is the $\Thnsset$-enriched category generated by the following data
    \begin{tz}
        \node[](1) {$0$}; 
        \node[right of=1,xshift=0.5cm](2) {$1$}; 
        \draw[->,bend left] (1) to node(a)[above,la]{$f$} (2);
        \draw[->,bend right] (1) to node(b)[below,la]{$f'$} (2);
        \cell[right,la][n][0.53]{a}{b}{$\alpha$};
    \end{tz}
    \item When $m=2$, using \cref{prop:computationpushprod} applied to the case where $m=2$ and $X\hookrightarrow Y$ is the identity at $\repThn[1;0]$, we get that $\Ch L(\repD[2]\times \repThn[1;0])$ is the directed $\Thnsset$-enriched category with object set $\{0,1,2\}$ and hom $\Thn$-spaces 
    \[ \Hom_{\Ch L(\repD[2]\times \repThn[1;0])}(0,1)=\Hom_{\Ch L(\repD[2]\times \repThn[1;0])}(1,2)=\repThn[1;0]\]
    and 
    \[ \Hom_{\Ch L(\repD[2]\times \repThn[1;0])}(0,2)=(\repThn[1;0]\times \repThn[1;0]) \amalg_{\repThn[1;0]} \repThn[1;0]\times \Delta[1] \]
    and so it is the $\Thnsset$-enriched category generated by the following data
    \begin{tz}
        \node[](0) {$0$}; 
        \node[above right of=0,xshift=1cm,yshift=1cm](1) {$1$};
        \node[below right of=1,xshift=1cm,yshift=-1cm](2) {$2$}; 
        \draw[->,bend left=20] (0) to node(a)[left,la,yshift=3pt]{$f$} (1); 
        \draw[->,bend right=20] (0) to node(b)[right,la,yshift=-3pt]{$f'$} (1); 
        \cell[above,la,xshift=3pt][n][0.53]{a}{b}{$\alpha$};
        \draw[->,bend left=20] (1) to node(a)[right,la,yshift=3pt]{$g$} (2);
        \draw[->,bend right=20] (1) to node(b)[left,la,yshift=-3pt]{$g'$} (2);
        \cell[above,la,xshift=-3pt,yshift=-3pt][n][0.53]{a}{b}{$\beta$};
        \draw[->,bend left=15] (0) to node(a)[above,la]{$h$} (2);
        \draw[->,bend right=15] (0) to node(b)[below,la]{$h'$} (2);
        \cell[right,la][n][0.53]{a}{b}{$\gamma$};
    \end{tz}
    together with a homotopy between $\gamma$ and the horizontal composite $\beta\alpha$, which in particular gives homotopies between $h$ and the composite $gf$ and between $h'$ and the composite $g'f'$. 
    \end{itemize}
\end{ex}

\subsection{The homs of the homotopy coherent categorification} \label{subsec:homsofC}

Using the description of the hom spaces of the homotopy coherent categorification $\CL$ of a simplicial set, we can compute explicitly the hom $\Thn$-spaces of the homotopy coherent categorification $\Ch$ of an object in $\pcatThn$. The following two results are obtained by applying level-wise \cref{DS:computationhoms,DS:computations1ordered}.

 \begin{prop} \label{cor:computationhomsC}
 Let $W$ be an object in $\pcatThn$ and $a,b\in W_0$. Then there is a natural isomorphism in $\Thnsset$
 \[ \Hom_{\Ch W}(a,b)\cong \diag( \colim_{T\in \catnec{W_{-,\star,\star}}{a}{b}} \Hom_{\CL T}(\alpha,\omega)) \]
 where $\colim_{T\in \catnec{W_{-,\star,\star}}{a}{b}} \Hom_{\CL T}(\alpha,\omega)$ is the $\Thn$-bi-space given at $\defThn\in\Thn$ and $\defS\geq 0$ by the colimit in $\sset$
 \[  \colim_{T\in \catnec{W_{-,\defThn,\defS}}{a}{b}} \Hom_{\CL T}(\alpha,\omega). \]
 \end{prop}

 \begin{cor} \label{cor:computationshomC1ordered}
Let $W$ be an object in $\pcatThn$ and $a,b\in W_0$. Suppose that, for all $\defThn\in \Thn$ and $\defS\geq 0$, the simplicial set $W_{-,\defThn,\defS}$ is $1$-ordered. Then there is a natural isomorphism in $\Thnsset$
\[ \Hom_{\Ch W}(a,b)\cong \diag( \colim_{T\in \tndnec{W_{-,\star,\star}}{a}{b}} \Hom_{\CL T}(\alpha,\omega)) \]
where $\colim_{T\in \tndnec{W_{-,\star,\star}}{a}{b}} \Hom_{\CL T}(\alpha,\omega)$ is the $\Thn$-bi-space given at $\defThn\in\Thn$ and $\defS\geq 0$ by the colimit in $\sset$
 \[  \colim_{T\in \tndnec{W_{-,\defThn,\defS}}{a}{b}} \Hom_{\CL T}(\alpha,\omega). \]
 \end{cor}

 We now aim to compare the hom $\Thn$-spaces of the categorification $\Ch W$ with the mapping $\Thn$-spaces of $W$ in the case where $W$ is fibrant. For this, we first need the following.

\begin{lemma} \label{fibrantpcat}
Let $W$ be a fibrant object in $\pcatinj$. For every $\defThn\in \Thn$ and $\defS\geq 0$, the simplicial set $W_{-,\defThn,\defS}$ is fibrant in $\Dset_{\CSThn[1]}$.
\end{lemma}

\begin{proof}
Recall that, if $W$ is fibrant in $\pcatinj$, then it is fibrant in $\injsThnsset$ and it satisfies the Segal condition, i.e., it is fibrant in the localization $\SegsThnsset$ of $\injsThnsset$ with respect to the Segal maps 
\[ (\Sp \hookrightarrow\repD)\times \repThn\]
for all $m\geq 1$ and $\defThn\in \Thn$. By \cite[Lemma 3.5]{JT}, if a saturated class of monomorphisms satisfying the right cancellation property contains the Segal maps $\Sp \hookrightarrow \repD$ -- which is the case of the class of trivial cofibrations of $\SegsThnsset$ -- then it must contain the inner horn inclusions $L^t[m]\hookrightarrow \repD$, for all $m\geq 2$ and $0<t<m$. Furthermore, as the model structure $\SegsThnsset$ is cartesian closed by \cite[Theorem~5.2]{br2}, the maps
\[ (L^t[m]\hookrightarrow \repD)\times \repThn\times \repS \] 
are trivial cofibrations in $\SegsThnsset$, for all $m\geq 2$, $0<t<m$, $\defThn\in \Thn$, and $\defS\geq 0$.

Now, for all $m\geq 2$, $0<t<m$, $\defThn\in \Thn$, and $\defS\geq 0$, a lift in the below left diagram in~$\Dset$ corresponds to a lift in the below right diagram in $\sThnsset$, which exists by the above discussion. 
\begin{tz}
    \node[](1) {$L^t[m]$}; 
    \node[below of=1](2) {$\repD$}; 
    \node[right of=1,xshift=.7cm](3) {$W_{-,\defThn,\defS}$}; 

    \draw[right hook->] (1) to (2); 
    \draw[->] (1) to (3); 
    \draw[->,dashed] (2) to (3);

    \node[right of=3,xshift=2.7cm](1) {$L^t[m]\times \repThn\times \repS$}; 
    \node[below of=1](2) {$\repD\times \repThn\times \repS$}; 
    \node[right of=1,xshift=1.8cm](3) {$W$}; 

    \draw[right hook->] (1) to (2); 
    \draw[->] (1) to (3); 
    \draw[->,dashed] ($(2.north)+(1.4cm,0)$) to (3);
\end{tz}
This shows that $W_{-,\defThn,\defS}$ is fibrant in $\Dset_{\CSThn[1]}$, as desired. 
\end{proof}

\begin{defn} \label{def:homW}
Let $W$ be a fibrant object in $\pcatinj$. For $a,b\in W_0$, we define $\hom_W(a,b)$ to be the following pullback in $\sThnsset$. 
\begin{tz}
\node[](1) {$\hom_W(a,b)$}; 
\node[below of=1](2) {$\repS[0]$}; 
\node[right of=1,xshift=1.8cm](3) {$W^{\repD[1]}$}; 
\node[below of=3](4) {$W\times W$}; 
\pullback{1};

\draw[->] (1) to (3);
\draw[->] (1) to (2);
\draw[->] (3) to (4);
\draw[->] (2) to node[below,la]{$(a,b)$} (4);
\end{tz}
\end{defn}

\begin{rmk}
    Since $\hom_W(a,b)$ is homotopically constant, i.e., for every $m\geq 0$, the map $\hom_W(a,b)_0\to \hom_W(a,b)_m$ is a weak equivalence in $\MSThnsset$, we equivalently regard it as an object of $\Thnssset$ through the canonical isomorphism $\sThnsset\cong \Thnssset$. 
\end{rmk}

Note that $\Map_W(a,b)=\hom_W(a,b)_0$. We then have the following. 

\begin{prop} \label{cor:Homvsdiaghom}
Let $W$ be a fibrant object in $\pcatinj$ and $a,b\in W_0$. Then there is a natural zig-zag of weak equivalences in $(\MSspace)^{\Thnop}_\inj$ connecting the $\Thn$-spaces
\[\Hom_{\Ch W}(a,b)\sim \diag \hom_W(a,b). \]
\end{prop}

\begin{proof}
For $\defThn\in \Thn$ and $\defS\geq 0$, as $W_{-,\defThn,\defS}$ is fibrant in $\Dset_{\CSThn[1]}$ by \cref{fibrantpcat}, by \cref{thm:zigzag2} we have a natural zig-zag of weak equivalences in $\MSspace$
\[ \Hom_{\CL_* W}(a,b)_{\defThn,\defS}\cong \Hom_{\CL(W_{-,\defThn,\defS})}(a,b)\sim \map_{W_{-,\defThn,\defS}}(a,b)\cong\hom_W(a,b)_{-,\defThn,\defS}. \]
Hence, we obtain a natural zig-zag of weak equivalences in $\injsThnspace$
\[ \Hom_{\CL_* W}(a,b)\sim \hom_W(a,b), \]
and so in its localization $(\MSdiag)^{\Thnop}_\inj$. As $\diag\colon (\MSdiag)^{\Thnop}_\inj\to \injThnspace$ preserves weak equivalences by \cref{diagiotaQEinj}, we obtain the desired natural zig-zag of weak equivalences in $\injThnspace$
\[ \Hom_{\Ch W}(a,b)= \diag \Hom_{\CL_* W}(a,b)\sim \diag \hom_W(a,b). \qedhere \]
\end{proof}

Thanks to the previous result, in order to compare $\Hom_{\Ch W}(a,b)$ with $\Map_W(a,b)$, it is enough to compare $\Map_W(a,b)$ with $\diag \hom_W(a,b)$. 

\begin{prop} \label{prop:Mapvsdiaghom}
Let $W$ be a fibrant object in $\pcatinj$ and $a,b\in W_0$. Then there is a natural weak equivalence in $(\MSspace)^{\Thnop}_\inj$ 
\[ \Map_W(a,b)\xrightarrow{\sim} \diag\hom_W(a,b). \]
\end{prop}

\begin{proof}
As $W$ is a fibrant object $\pcatinj$ (see \cref{recall:fibrant}), by \cite[Theorem~2.30]{RasekhD} the map $\pi\colon \{a\}\times_W W^{\repD[1]} \to W$ in $\sThnsset$ is a $\Thn$-left fibration in the sense of \cite[Definition~2.1]{RasekhD}. By \cite[Lemma~2.10]{RasekhD}, $\Thn$-left fibrations are stable under pullbacks. So the pullback of $\pi$ along $b\colon \repD[0]\to W$, which is by \cref{def:homW} precisely
    \[\hom_W(a,b)\to\repD[0],\]
    is a $\Thn$-left fibration. It then follows from \cite[Lemma~2.6]{RasekhD} that the map
    \[ \Map_W(a,b)=\hom_W(a,b)_0\to \hom_W(a,b) \]
    is a weak equivalence in $\injsThnspace$, and hence also in its localization $(\MSdiag)^{\Thnop}_\inj$. As $\diag\colon (\MSdiag)^{\Thnop}_\inj\to \injThnspace$ preserves weak equivalences by \cref{diagiotaQEinj}, we obtain the desired weak equivalence in $\injThnspace$
\[ \Map_W(a,b)=\diag \Map_W(a,b)\xrightarrow{\sim} \diag\hom_W(a,b). \qedhere \]
\end{proof}

Combining \cref{cor:Homvsdiaghom,prop:Mapvsdiaghom}, we get the following. 

\begin{cor} \label{HomChigher}
Let $W$ be a fibrant object in $\pcatinj$ and $a,b\in W_0$. Then there is a natural zig-zag of weak equivalences in $(\MSspace)^{\Thnop}_\inj$ connecting the $\Thn$-spaces
\[\Hom_{\Ch W}(a,b)\sim \Map_W(a,b).\]
\end{cor}

\subsection{The homotopy category of the homotopy coherent categorification} \label{subsec:HoofC}

We now compare the homotopy category of $\Ch W$ with that of $W$ in the case where $W$ is fibrant.

\begin{lemma} \label{zigzagforhtpycat}
Let $W$ be a fibrant object in $\pcatinj$ and $a,b\in W_0$. If $\Ch W\to (\Ch W)^{\fib}$ is a fibrant replacement in $\injThnspace\text{-}\cat$, then there is a natural zig-zag of weak equivalences in $\injThnspace$ between the $\Thn$-spaces
\[ \Hom_{(\Ch W)^{\fib}}(a,b)\sim \Map_W(a,b). \]
\end{lemma}

\begin{proof}
    This follows from \cref{HomChigher} and the fact that by definition of the fibrant replacement $(\Ch W)^\fib$ the map $\Hom_{\Ch W}(a,b)\xrightarrow{\sim}\Hom_{(\Ch W)^\fib}(a,b)$ is a weak equivalence in $\injThnspace$.
\end{proof}

\begin{lemma} \label{lemmaFibRepPcat}
Let $W$ be a fibrant object in $\pcatinj$. If $\Ch W\to (\Ch W)^{\fib}$ is a fibrant replacement in $\injThnspace\text{-}\cat$, then $(\Ch W)^{\fib}$ is in fact fibrant in $\MSThncat$.
\end{lemma}

\begin{proof}
Let $a,b\in W_0$. Since $(\Ch W)^\fib$ is a fibrant object in $\injThnspace\text{-}\cat$ and $W$ is a fibrant object in $\pcatinj$, then $\Hom_{(\Ch W)^{\fib}}(a,b)$ and $\Map_W(a,b)$ are fibrant in $\injThnspace$. Moreover, by \cref{zigzagforhtpycat}, we have a zig-zag of weak equivalence in $\injThnspace$
\[ \Hom_{(\Ch W)^{\fib}}(a,b)\sim \Map_W(a,b). \]
As both $\Thn$-spaces are fibrant in $\injThnspace$, we can assume that the above zig-zag only passes through fibrant objects of $\injThnspace$ (by fibrantly replacing the intermediate objects if necessary). As $\Map_W(a,b)$ is further fibrant in $\MSThnsset$, then by \cite[Lemma 3.2.1]{Hirschhorn} we have that $\Hom_{(\Ch W)^{\fib}}(a,b)$ is also fibrant in $\MSThnsset$. It follows that $(\Ch W)^{\fib}$ is fibrant in $\MSThncat$, as desired.
\end{proof}

\begin{prop} \label{lem:eqhtpycat}
Let $W$ be a fibrant object in $\pcatinj$. Then there is a natural isomorphism of categories
 \[  \Ho(\Ch W)\cong \Ho W.\]
\end{prop}

\begin{proof}
By construction, the homotopy categories $\Ho(\Ch W)$ and $\Ho W$ have the same set of objects~$W_0$, hence it is enough to show that their hom sets are isomorphic. 

Let $\Ch W\to (\Ch W)^{\fib}$ be a fibrant replacement in $\injThnspace\text{-}\cat$. By \cref{lemmaFibRepPcat} we have that $(\Ch W)^{\fib}$ is in fact a fibrant replacement in $\MSThncat$. As $\Ch W\to (\Ch W)^{\fib}$ is a Dwyer-Kan equivalence, we have an equivalence of categories $\Ho(\Ch W)\simeq\Ho((\Ch W)^{\fib})$.

Let $a,b\in W_0$. By \cref{zigzagforhtpycat}, we have a natural zig-zag of weak equivalences in $\injThnspace$
\[ \Hom_{(\Ch W)^\fib}(a,b)\sim \Map_W(a,b). \]
As weak equivalences in $\injThnspace$ are level-wise, we get a natural zig-zag of weak equivalences in $\MSspace$
\[ \Hom_{(\Ch W)^\fib}(a,b)_{[0]}\sim \Map_W(a,b)_{[0]}. \]
As $\pi_0\colon \sset\to \set$ sends weak equivalences in $\MSspace$ to isomorphisms, we obtain
\begin{align*}
    \Ho(\Ch W)(a,b) &\cong \Ho((\Ch W)^{\fib})(a,b)\cong \pi_0 (\Hom_{(\Ch W)^{\fib}}(a,b)_{[0]})
 \\ &\cong \pi_0(\Map_W(a,b)_{[0]})\cong \Ho (W)(a,b),
\end{align*}
where the isomorphism $\Ho((\Ch W)^{\fib})(a,b)\cong \pi_0 (\Hom_{(\Ch W)^{\fib}}(a,b)_{[0]})$ holds by \cref{prop:homofHoC}. This concludes the proof.
\end{proof}

\section{Explicit computations of the homotopy coherent categorification} \label{Section 3}

In order to show in \cref{Section 4} that $\Ch$ is left Quillen, we need to understand the image of the (trivial) cofibrations in $\pcatinj$ from \cref{rem:gencof,rem:genanodyne}, which are of the form 
\[ \PP\hookrightarrow L(\repD\times Y)\coloneqq L((\partial \repD\hookrightarrow \repD)\widehat{\times}(X\hookrightarrow Y)) \] 
\[ \text{and} \quad L((\Sp\hookrightarrow \repD)\widehat{\times}(X\hookrightarrow Y)) \]
for $m\geq 1$ and $X\hookrightarrow Y$ a monomorphism in $\Thnsset$, where $L\colon \sThnsset\to \pcatThn$ denotes the left adjoint to the inclusion. In this section, we collect the technical results regarding these maps, and the reader is encouraged to skip this section on a first read.

In \cref{Study of PP} we introduce $\PP$ and in \cref{Necklaces in PP} we describe the category of necklaces in $\PP$. In \cref{Auxiliary results about weighted colimits} we discuss how the category of necklaces in $\PP$ is a discrete fibration over the category of necklaces in $\repD$, and then describe the hom $\Thn$-spaces of $\Ch\PP$ as a certain weighted colimit. This relies on results that will be postponed until \cref{Projective cofibrancy results}. In \cref{Study of CPP} we use this to describe the $\Thnsset$-enriched category $\Ch\PP$ and study the $\Thnsset$-enriched functor $\Ch(\PP\hookrightarrow L(\repD\times Y))$. Finally, in \cref{Study of Sigmam} we construct and study a $\Thnsset$-enriched functor $\Ch(L(\repD\times X))\to \Sigma_m X$, related to the image under $\Ch$ of the second type of monomorphisms.

\subsection{Study of \texorpdfstring{$\PP$}{P(X,Y)}} \label{Study of PP}

We denote by $\pi_0\colon \Thnsset\to \set$ the left adjoint to the inclusion $\set\hookrightarrow \Thnsset$. Also recall the left adjoint $L\colon \sThnsset\to \pcatThn$ to the inclusion. We get the following description.

\begin{lemma} \label{RmkPushout}
For $m\geq 1$ and $X\in \Thnsset$, we can compute $L(\repD\times X)$ and $L(\partial\repD\times X)$ as the following pushouts in $\sThnsset$.
\begin{tz}
\node[](1) {$\coprod_{m+1} X$};
\node[below of=1](2) {$\coprod_{m+1} \pi_0 X$}; 
\node[right of=1,xshift=1.9cm](3) {$\repD\times X$}; 
\node[below of=3](4) {$L(\repD\times X)$}; 
\pushout{4};

\draw[->] (1) to (3);
\draw[->] (1) to (2);
\draw[->] (3) to (4);
\draw[->] (2) to (4);

\node[right of=3,xshift=2cm](1) {$\coprod_{m+1} X$};
\node[below of=1](2) {$\coprod_{m+1} \pi_0 X$}; 
\node[right of=1,xshift=2cm](3) {$\partial\repD\times X$}; 
\node[below of=3](4) {$L(\partial\repD\times X)$}; 
\pushout{4};

\draw[->] (1) to (3);
\draw[->] (1) to (2);
\draw[->] (3) to (4);
\draw[->] (2) to (4);
\end{tz}
\end{lemma}

In this section we want to understand the object $\PP$ that we now define.

\begin{notation} \label{notationPP}
For $m\geq 1$ and $X\hookrightarrow Y$ a monomorphism in $\Thnsset$, we write $\PP$ for the following pushout in $\pcatThn$ (hence also in $\sThnsset$).
\begin{tz}
\node[](1) {$L(\partial\repD\times X)$};
\node[below of=1](2) {$L(\partial\repD\times Y)$}; 
\node[right of=1,xshift=1.9cm](3) {$L(\repD\times X)$}; 
\node[below of=3](4) {$\PP$}; 
\node[below right of=4,xshift=2cm](5) {$L(\repD\times Y)$}; 
\pushout{4};

\draw[->] (1) to (3);
\draw[->] (1) to (2);
\draw[->] (3) to (4);
\draw[->] (2) to (4);
\draw[->,bend left] (3) to (5);
\draw[->,bend right=15] (2) to (5);
\draw[->,dashed] (4) to node[below,la,xshift=-5pt]{$\inc$} (5);
\end{tz}
By the universal property of pushout, it comes with a map $\inc$ in~$\pcatThn$ as depicted above.

Note that, if we consider the identity $Y\hookrightarrow Y$, then $\PP[m][Y]\cong L(\repD\times Y)$.
\end{notation}

\begin{lemma} \label{lem:PPpushout}
Let $m\geq 1$ and $X\hookrightarrow Y$ be a monomorphism in $\Thnsset$. Then $\PP$ is the following pushout in $\sThnsset$. 
\begin{tz}
\node[](1) {$\coprod_{m+1} Y$};
\node[below of=1](2) {$\coprod_{m+1} \pi_0 Y$}; 
\node[right of=1,xshift=2.9cm](3) {$\partial\repD\times Y\amalg_{\partial\repD\times X} \repD\times X$}; 
\node[below of=3](4) {$\PP$}; 
\pushout{4};

\draw[->] (1) to (3);
\draw[->] (1) to (2);
\draw[->] (3) to (4);
\draw[->] (2) to (4);
\end{tz}
\end{lemma}

\begin{proof}
    This is an instance of pushouts commuting with pushouts, using \cref{RmkPushout}.
\end{proof}

\begin{rmk} \label{rem:PPtoFm}
For $m\geq 1$ and $X\hookrightarrow Y$ a monomorphism in $\Thnsset$, the map
\[ \partial\repD\times Y\amalg_{\partial\repD\times X} \repD\times X\to \partial\repD\amalg_{\partial\repD}\repD =\repD\]
induced by the projection maps gives a commutative square 
\begin{tz}
\node[](1) {$\coprod_{m+1} Y$};
\node[below of=1](2) {$\coprod_{m+1} \pi_0 Y$}; 
\node[right of=2,xshift=1.4cm](2') {$\coprod_{m+1} \repD[0]$};
\node[right of=1,xshift=3.9cm](3) {$\partial\repD\times Y\amalg_{\partial\repD\times X} \repD\times X$}; 
\node[below of=3](4) {$\repD$}; 
\punctuation{4}{.};

\draw[->] (1) to (3);
\draw[->] (1) to (2);
\draw[->] (3) to (4);
\draw[->] (2) to (2');
\draw[->] (2') to (4);
\end{tz}
By \cref{lem:PPpushout}, as $\PP$ is the pushout of the above span, we get an induced map
\[ \QmXY\colon \PP\to \repD.\] 
\end{rmk}

We particularly care to study $\PP$ in the case where $Y$ is \emph{connected}.

\begin{defn}
A $\Thn$-space $Y$ is \emph{connected} if there is an isomorphism of sets $\pi_0 Y\cong \{*\}$.
\end{defn}

\begin{rmk} \label{rmk:monoconnectedtarget}
For $\defThn\in \Thn$ and $\defS\geq 0$, the representable $\repThn\times\repS$ is a connected $\Thn$-space. In particular, this says that all monomorphisms in $\Thnsset$ of the form 
\[ (\partial\repThn\hookrightarrow \repThn)\widehat{\times} (\partial\repS\hookrightarrow \repS) \]
are monomorphisms with connected target.
\end{rmk}

In the case where $Y$ is connected, we can describe $\PP$ as follows.

\begin{lemma} \label{PP0}
Let $m\geq 1$, $Y$ be a connected $\Thn$-space, and $X\hookrightarrow Y$ be a monomorphism in $\Thnsset$. Then there is an isomorphism in $\Thnsset$
\[ \PP_0\cong \{0,1,\ldots,m\}. \]
\end{lemma}

\begin{proof}
First note that, as $X\subseteq Y$, we have an isomorphism in $\Thnsset$
\[ \textstyle (\partial\repD\times Y\amalg_{\partial\repD\times X} \repD\times X)_0 \cong \coprod_{m+1} Y. \]
By applying the (colimit-preserving) functor $(-)_0\colon \sThnsset\to \Thnsset$ to the pushout of \cref{lem:PPpushout}, we obtain that $\PP_0\cong \coprod_{m+1} \repD[0]\cong \{0,1,\ldots,m\}$.
\end{proof}

\begin{rmk} \label{PPthetakpushout}
    Let $m\geq 1$, $Y$ be a connected $\Thn$-space, and $X\hookrightarrow Y$ be a monomorphism in $\Thnsset$. By \cref{lem:PPpushout}, we obtain that, for $\defThn\in \Thn$ and $\defS\geq 0$, the simplicial set
    $\PP_{-,\defThn,\defS}
 $ is the following pushout in $\set^{\Dop}$.
 \begin{tz}
\node[](1) {$\coprod_{m+1} \coprod_{Y_{\defThn,\defS}} \repD[0]$};
\node[below of=1](2) {$\coprod_{m+1} \repD[0]$}; 
\node[right of=1,xshift=3.8cm](3) {$(\coprod_{Y_{\defThn,\defS}\setminus X_{\defThn,\defS}} \partial\repD) \amalg (\coprod_{X_{\defThn,\defS}} \repD)$}; 
\node[below of=3](4) {$\PP_{-,\defThn,\defS}$}; 
\pushout{4};

\draw[->] (1) to (3);
\draw[->] (1) to (2);
\draw[->] (3) to (4);
\draw[->] (2) to (4);
\end{tz}
Moreover, the component $ \QmXY_{-,\defThn,\defS}\colon \PP_{-,\defThn,\defS}\to \repD_{-,\defThn,\defS}=\repD$ of the map from \cref{rem:PPtoFm} is induced by the fold map 
\[  \textstyle (\coprod_{Y_{\defThn,\defS}\setminus X_{\defThn,\defS}} \partial\repD) \amalg (\coprod_{X_{\defThn,\defS}} \repD)\to \partial\repD\amalg\repD\hookrightarrow \repD\amalg\repD\to \repD.  \]
\end{rmk}

The object $\PP$ satisfies the following useful property introduced in \cref{subsec:necklace}.

\begin{prop} \label{prop:PP1ordered}
Let $m\geq 1$, $Y$ be a connected $\Thn$-space, and $X\hookrightarrow Y$ be a monomorphism in $\Thnsset$. For all $\defThn\in \Thn$ and $\defS\geq 0$, the simplicial set $\PP_{-,\defThn,\defS}$ is $1$-ordered.
\end{prop}

\begin{proof}
    By \cref{PP0}, we have that $\PP_{0,\defThn,\defS} = \{0,1,\ldots,m\}$ and by construction every $1$-simplex goes from $i$ to $j$ where $i \leq j$. Hence the relation $\preceq_{\PP_{-,\defThn,\defS}}$ is precisely the linear order $0 \leq 1\leq \ldots \leq m$, and so it is in particular anti-symmetric.

    For $m'\geq 1$, we first show that, for every $m'$-simplex $\repD[m']\to \PP_{-,\defThn,\defS}$, its restriction along the inclusion $\Sp[m']\hookrightarrow \repD[m']$ is a monomorphism $\Sp[m']\to \PP_{-,\defThn,\defS}$. Let $\sigma\colon \repD[m']\to \PP_{-,\defThn,\defS}$ be a non-degenerate $m'$-simplex of $\PP_{-,\defThn,\defS}$. By the description of $\PP_{-,\defThn,\defS}$ given in \cref{PPthetakpushout}, such an $m'$-simplex comes from a non-degenerate $m'$-simplex
    \[  \textstyle \overline{\sigma}\colon \repD[m']\to (\partial)\repD\hookrightarrow(\coprod_{Y_{\defThn,\defS}\setminus X_{\defThn,\defS}} \partial\repD) \amalg (\coprod_{X_{\defThn,\defS}} \repD). \]
    Since the simplicial sets $\partial\repD$ and $\repD$ are $1$-ordered by \cref{examplesof1ordered}, it follows that the induced map $\Sp[m']\hookrightarrow\repD[m']\xrightarrow{\overline{\sigma}} (\partial)\repD$ is a monomorphism. Now, as the composite
    \[ \textstyle (\partial) \repD\hookrightarrow (\coprod_{Y_{\defThn,\defS}\setminus X_{\defThn,\defS}} \partial\repD) \amalg (\coprod_{X_{\defThn,\defS}} \repD)\to \PP_{-,\defThn,\defS} \]
    is also a monomorphism, it follows that the induced map $\Sp[m']\hookrightarrow\repD[m']\xrightarrow{\sigma} \PP_{-,\defThn,\defS}$ is the composite of monomorphisms 
    \[\Sp[m']\hookrightarrow\repD[m']\xrightarrow{\overline{\sigma}} (\partial)\repD\hookrightarrow\PP_{-,\defThn,\defS}\]
    and so is also a monomorphism. 
    
    Next, we show that the restriction of the Segal map 
    \[ \PP_{m',\defThn,\defS}\to \PP_{1,\defThn,\defS}\times_{\PP_{0,\defThn,\defS}} \ldots \times_{\PP_{0,\defThn,\defS}}  \PP_{1,\defThn,\defS}\]
    to the subset $\PP_{m',\defThn,\defS}^{\mathrm{nd}}$ of non-degenerate $m'$-simplices of $\PP_{-,\defThn,\defS}$ is injective. Let $\sigma,\tau\colon \repD[m']\to \PP_{-,\defThn,\defS}$ be non-degenerate $m'$-simplices of $\PP_{-,\defThn,\defS}$ such that their restrictions along $\Sp[m']\hookrightarrow \repD[m']$ coincide. As before, they come from non-degenerate $m'$-simplices 
    \[ \overline{\sigma}\colon \repD[m']\to \{y\}\times(\partial)\repD \quad \text{and} \quad  \overline{\tau}\colon \repD[m']\to \{y'\}\times (\partial)\repD, \]
    where $y,y'\in Y_{\defThn,\defS}$, and  $\{y\}\times (\partial)\repD$, $\{y'\}\times (\partial)\repD$ are the corresponding factors of the coproduct $(\coprod_{Y_{\defThn,\defS}\setminus X_{\defThn,\defS}} \partial\repD) \amalg (\coprod_{X_{\defThn,\defS}} \repD)$. As the restrictions of $\sigma$ and $\tau$ along $\Sp[m']\hookrightarrow \repD[m']$ coincide, and the map 
    \[ \textstyle (\coprod_{Y_{\defThn,\defS}\setminus X_{\defThn,\defS}} \partial\repD) \amalg (\coprod_{X_{\defThn,\defS}} \repD)\to \PP_{-,\defThn,\defS}\]
    is injective on $1$-simplices, it follows that $y=y'$. So $\overline{\sigma}$ and $\overline{\tau}$ are two non-degenerate $m'$-simplices of $\{y\}\times (\partial)\repD$ whose restrictions along $\Sp[m']\hookrightarrow \repD[m']$ coincide. Hence, as $\partial\repD$ and $\repD$ are $1$-ordered by \cref{examplesof1ordered}, it follows that $\overline{\sigma}=\overline{\tau}$ and so $\sigma=\tau$.
\end{proof}

\subsection{Study of necklaces in \texorpdfstring{$\PP$}{P(X,Y)}} \label{Necklaces in PP}

In this subsection, we study the category of necklaces in $\PP$ in the case where $Y$ is connected.

\begin{rmk}
Let $m\geq 1$, $Y$ be a connected $\Thn$-space, and $X\hookrightarrow Y$ be a monomorphism in $\Thnsset$. As $\PP_{-,\defThn,\defS}$ is $1$-ordered for all $\defThn\in\Thn$ and $\defS\geq 0$ by \cref{prop:PP1ordered}, in order to study the homotopy coherent categorification of $\PP$, by \cref{cor:computationshomC1ordered} it is enough to study the totally non-degenerate necklaces in $\PP_{-,\defThn,\defS}$.
\end{rmk}

\begin{lemma} \label{lemmanecklacePPvsLY}
Let $m\geq 1$, $Y$ be a connected $\Thn$-space, and $X\hookrightarrow Y$ be a monomorphism in $\Thnsset$. For all $\defThn\in \Thn$, $\defS\geq 0$, and all $0<j-i<m$, then the canonical map $\inc\colon \PP\to L(\repD\times Y)$ induces a natural isomorphism of categories
\[ \tndnec{\PP_{-,\defThn,\defS}}{i}{j}\cong \tndnec{L(\repD\times Y)_{-,\defThn,\defS}}{i}{j}. \]
\end{lemma}

\begin{proof} 
    Recall that $\PP[m][Y]\cong L(\repD\times Y)$ and that by \cref{PP0}
    \[ \PP_{0,\defThn,\defS}\cong \{0,1,\ldots,m\}\cong \PP[m][Y]_{0,\defThn,\defS}\cong L(\repD\times Y)_{0,\defThn,\defS}. \]
    We denote respectively by $(\PP_{-,\defThn,\defS})_{[i,j]}$ and $(L(\repD\times Y)_{-,\defThn,\defS})_{[i,j]}$ the simplicial subsets of $\PP_{-,\defThn,\defS}$ and $L(\repD\times Y)_{-,\defThn,\defS}$ spanned by the $0$-simplices $i,i+1,\dots,j$. Using the description of $\PP_{-,\defThn,\defS}$ given in \cref{PPthetakpushout}, we get that $(\PP_{-,\defThn,\defS})_{[i,j]}$ is the following pushout in $\set^{\Dop}$. 
     \begin{tz}
\node[](1) {$\coprod_{j-i+1} \coprod_{Y_{\defThn,\defS}} \repD[0]$};
\node[below of=1](2) {$\coprod_{j-i+1} \repD[0]$}; 
\node[right of=1,xshift=4.3cm](3) {$(\coprod_{Y_{\defThn,\defS}\setminus X_{\defThn,\defS}} \partial\repD_{[i,j]}) \amalg (\coprod_{X_{\defThn,\defS}} \repD_{[i,j]})$}; 
\node[below of=3](4) {$(\PP_{-,\defThn,\defS})_{[i,j]}$}; 
\pushout{4};

\draw[->] (1) to (3);
\draw[->] (1) to (2);
\draw[->] (3) to (4);
\draw[->] (2) to (4);
\end{tz}
Similarly, as $L(\repD\times Y)_{-,\defThn,\defS}\cong \PP[m][Y]_{-,\defThn,\defS}$, we get that $(L(\repD\times Y)_{-,\defThn,\defS})_{[i,j]}$ is the following pushout in $\set^{\Dop}$.
\begin{tz}
\node[](1) {$\coprod_{j-i+1} \coprod_{Y_{\defThn,\defS}} \repD[0]$};
\node[below of=1](2) {$\coprod_{j-i+1} \repD[0]$}; 
\node[right of=1,xshift=2.9cm](3) {$\coprod_{Y_{\defThn,\defS}} \repD_{[i,j]}$}; 
\node[below of=3](4) {$(L(\repD\times Y)_{-,\defThn,\defS})_{[i,j]}$}; 
\pushout{4};

\draw[->] (1) to (3);
\draw[->] (1) to (2);
\draw[->] (3) to (4);
\draw[->] (2) to (4);
\end{tz}
As $\partial\repD_{[i,j]}\cong \repD_{[i,j]}$ as $0<j-i<m$,  there is an isomorphism in $\set^{\Dop}$
\[ \textstyle (\coprod_{Y_{\defThn,\defS}\setminus X_{\defThn,\defS}} \partial\repD_{[i,j]}) \amalg (\coprod_{X_{\defThn,\defS}} \repD_{[i,j]})\cong\coprod_{Y_{\defThn,\defS}} \repD_{[i,j]}, \]
and so the two pushouts must be isomorphic. This gives an isomorphism in $\set^{\Dop}$
    \[ (\PP_{-,\defThn,\defS})_{[i,j]}\cong (L(\repD\times Y)_{-,\defThn,\defS})_{[i,j]}.\]
    The desired isomorphism of categories follows from the fact that the order $\preceq_{\PP_{-,\defThn,\defS}}$ (resp. $\preceq_{L(\repD\times Y)_{-,\defThn,\defS}}$) are given by $0\leq 1\leq \ldots \leq m$, and so every necklace from $i$ to $j$ has to be fully contained in $(\PP_{-,\defThn,\defS})_{[i,j]}$ (resp.~$(L(\repD\times Y)_{-,\defThn,\defS})_{[i,j]}$).
    \end{proof}

    We now aim to show that the projection $\QmXY_{-,\defThn,\defS}\colon \PP_{-,\defThn,\defS}\to \repD$ induces a functor between their categories of totally non-degenerate necklaces and that this functor is a discrete fibration. We first need the following.

\begin{lemma} \label{Qnondegsimplex}
Let $m\geq 1$, $Y$ be a connected $\Thn$-space, and $X\hookrightarrow Y$ be a monomorphism in $\Thnsset$. For $\defThn\in \Thn$ and $\defS\geq 0$, the map 
\[ \QmXY_{-,\defThn,\defS}\colon \PP_{-,\defThn,\defS}\to \repD \]
sends a non-degenerate simplex of $\PP_{-,\defThn,\defS}$ to a non-degenerate simplex of $\repD$.
\end{lemma}

\begin{proof}
    Consider an $m'$-simplex $\repD[m']\to \PP_{-,\defThn,\defS}$. By the description of $\PP_{-,\defThn,\defS}$ given in \cref{PPthetakpushout}, this amounts to an $m'$-simplex of the form
\[ \textstyle \repD[m']\to\coprod_{Y_{\defThn,\defS}\setminus X_{\defThn,\defS}}\partial\repD\quad\text{ or }\quad\repD[m']\to \coprod_{X_{\defThn,\defS}}\repD. \]
By \cref{PPthetakpushout}, these are sent by $\QmXY_{-,\defThn,\defS}$ to an $m'$-simplex
\[\textstyle \repD[m']\to\coprod_{Y_{\defThn,\defS}\setminus X_{\defThn,\defS}}\partial\repD\to\partial\repD\hookrightarrow\repD\quad\text{ or }\quad\repD[m']\to \coprod_{X_{\defThn,\defS}}\repD\to\repD. \]
In particular, an $m'$-simplex of $\PP_{-,\defThn,\defS}$ is non-degenerate if and only if its image under $\QmXY_{-,\defThn,\defS}$ is non-degenerate in $\repD$.
\end{proof}

\begin{prop}
Let $m\geq 1$, $Y$ be a connected $\Thn$-space, and $X\hookrightarrow Y$ be a monomorphism in $\Thnsset$. For $\defThn\in \Thn$ and $\defS\geq 0$, the map $\QmXY_{-,\defThn,\defS}\colon \PP_{-,\defThn,\defS}\to \repD$ induces by post-composition a functor
\[ (\QmXY_{-,\defThn,\defS})_!\colon \tndnec{\PP_{-,\defThn,\defS}}{0}{m}\to \tndnec{\repD}{0}{m}. \]
 \end{prop}

\begin{proof}
    By post-composing with the canonical map $\QmXY_{-,\defThn,\defS}\colon \PP_{-,\defThn,\defS}\to \repD$, we get a functor
    \[ (\QmXY_{-,\defThn,\defS})_!\colon \catnec{\PP_{-,\defThn,\defS}}{0}{m}\to \catnec{\repD}{0}{m}. \] 
    Furthermore, by \cref{Qnondegsimplex}, the map $\QmXY_{-,\defThn,\defS}$ sends a non-degenerate simplex of $\PP_{-,\defThn,\defS}$ to a non-degenerate simplex of $\repD$. It then follows that $(\QmXY_{-,\defThn,\defS})_!$ sends a totally non-degenerate necklace of $(\PP_{-,\defThn,\defS})_{0,m}$ to a totally non-degenerate necklace of $\repD_{0,m}$. Hence $(\QmXY_{-,\defThn,\defS})_!$ restricts to a functor
    \[ (\QmXY_{-,\defThn,\defS})_!\colon \tndnec{\PP_{-,\defThn,\defS}}{0}{m}\to \tndnec{\repD}{0}{m}, \]
    as desired.
\end{proof}

Recall from e.g.~\cite[Definition 2.1.1]{LRfib} the notion of a discrete fibration.

\begin{prop} \label{prop:discretefib}
Let $m\geq 1$, $Y$ be a connected $\Thn$-space, and $X\hookrightarrow Y$ be a monomorphism in $\Thnsset$. For $\defThn\in \Thn$ and $\defS\geq 0$, the functor
\[ (\QmXY_{-,\defThn,\defS})_!\colon \tndnec{\PP_{-,\defThn,\defS}}{0}{m}\to \tndnec{\repD}{0}{m} \]
is a discrete fibration.
\end{prop}

\begin{proof}
Let $T\to (\PP_{-,\defThn,\defS})_{0,m}$ be an object in $\tndnec{\PP_{-,\defThn,\defS}}{0}{m}$ and consider its image $T\to \repD_{0,m}$ under $(\QmXY_{-,\defThn,\defS})_!$. Given a map $f\colon U\to T$ in $\tndnec{\repD}{0}{m}$, the composite
\[ U\xrightarrow{f} T\to (\PP_{-,\defThn,\defS})_{0,m},\]is the unique lift of $f$ via $(\QmXY_{-,\defThn,\defS})_!$. Hence $(\QmXY_{-,\defThn,\defS})_!$ is a discrete fibration.
\end{proof}

As a consequence, to study the category of totally non-degenerate necklaces in $\PP_{-,\defThn,\defS}$, it is enough to study the category $\tndnec{\repD}{0}{m}$ and compute the fibers of the discrete fibration $(Q_{-,\defThn,\defS})_!$, which we now do.

\begin{rmk}\label{rem:NecFmposet}
Recall from \cref{examplesof1ordered} that $\repD$ is a $1$-ordered simplicial set, and so every totally non-degenerate necklace $T\to \repD_{0,m}$ in $\tndnec{\repD}{0}{m}$ is a monomorphism by \cref{tndvsmono}. Then every map in $\tndnec{\repD}{0}{m}$ has to be a monomorphism as well by the cancellation property of monomorphisms. Hence $\tndnec{\repD}{0}{m}$ is a poset. For a combinatorial description of this category, see \cref{subsec:Cube}.
\end{rmk}

\begin{rmk} \label{Bonarrows}
Recall from \cref{subsec:necklace} that $B(T)$ denotes the set of beads of a necklace $T\in \Nec$. Now, given a monomorphism $f\colon U\hookrightarrow T$ in $\Nec$, by \cite[Lemma 3.3]{DuggerSpivakRigidification} each bead $\repD[m_i]$ of $U$ is mapped into a unique bead of $T$, which we denote $B(f)(\repD[m_i])$. So we get a well-defined map of sets $B(f)\colon B(U)\to B(T)$, and the assignment $f\mapsto B(f)$ is functorial in all monomorphisms $f$. Note that the assignment $B(f)$ is not well-defined in general, because if $f$ is not a monomorphism, it might map a whole bead of $U$ to a joint of $T$, which does not belong to a unique bead of $T$. 

As a consequence, since every map in $\tndnec{\repD}{0}{m}$ is a monomorphism by \cref{rem:NecFmposet}, we get a functor 
\[ B\colon \tndnec{\repD}{0}{m}\to \set, \quad (T\hookrightarrow\repD_{0,m})\mapsto B(T). \] 
\end{rmk}

\begin{prop} \label{computefibers}
 Let $m\geq 1$, $Y$ be a connected $\Thn$-space, and $X\hookrightarrow Y$ be a monomorphism in $\Thnsset$. For $\defThn\in \Thn$ and $\defS\geq 0$, the fiber of the discrete fibration 
 \[ (\QmXY_{-,\defThn,\defS})_!\colon \tndnec{\PP_{-,\defThn,\defS}}{0}{m}\to \tndnec{\repD}{0}{m} \]
 at an object $T\hookrightarrow \repD_{0,m}$
 in $\tndnec{\repD}{0}{m}$ is given by the set
 \[\mathrm{fib}_{T\hookrightarrow\repD_{0,m}}((\QmXY_{-,\defThn,\defS})_!)\cong\begin{cases}
\prod_{B(T)} Y_{\defThn,\defS} & \text{if} \;\; T\neq \repD \\
X_{\defThn,\defS} & \text{if} \;\; T=\repD.
\end{cases}\] 
\end{prop}

\begin{proof}
    Let $T=\repD[m_1]\vee\ldots\vee\repD[m_t]\hookrightarrow \repD_{0,m}$ be a totally non-degenerate necklace in $\repD$. If $T\neq \repD$, we show that there is an isomorphism of sets \[ \textstyle \mathrm{fib}_{T\hookrightarrow \repD_{0,m}}((Q_{-,\defThn,\defS})_!)\cong \prod_{B(T)} Y_{\defThn,\defS}. \] 
    Given a totally non-degenerate necklace $T\hookrightarrow (\PP_{-,\defThn,\defS})_{0,m}$ which is sent by $(Q_{-,\defThn,\defS})_!$ to the totally non-degenerate necklace $T\hookrightarrow \repD_{0,m}$, then, for each $1\leq i\leq t$, the restriction of $T\to (\PP_{-,\defThn,\defS})_{0,m}$ to the bead $\repD[m_i]$ corresponds by the description of $\PP_{-,\defThn,\defS}$ given in \cref{PPthetakpushout} to a non-degenerate $m_i$-simplex
    \[  \textstyle  \repD[m_i]\to \{y_i\}\times \partial\repD\hookrightarrow(\coprod_{Y_{\defThn,\defS}\setminus X_{\defThn,\defS}} \partial\repD) \amalg (\coprod_{X_{\defThn,\defS}} \repD), \]
    for some $y_i\in Y_{\defThn,\defS}$. Then the data $(T\hookrightarrow \repD_{0,m},(y_i)_{1\leq i\leq t})$ uniquely determine the necklace $T\to (\PP_{-,\defThn,\defS})_{0,m}$, hence giving the desired isomorphism.
    
    Now, if $T=\repD$, necessarily $T=\repD\to\repD_{0,m}$ is the identity, and we show that there is an isomorphism of sets 
    \[ \mathrm{fib}_{T\hookrightarrow \repD_{0,m}}((Q_{-,\defThn,\defS})_!)\cong X_{\defThn,\defS}. \]
    This follows from the fact that a non-degenerate $m$-simplex of $\PP_{-,\defThn,\defS}$ comes from a non-degenerate $m$-simplex $\repD\to \{x\}\times \repD\hookrightarrow \coprod_{X_{\defThn,\defS}} \repD$, for some $x\in X_{\defThn,\defS}$, and a similar argument to the one above.
\end{proof}

We further record the following. 

\begin{prop} \label{prop:Necofspine}
Let $m\geq 1$ and $X$ be a connected $\Thn$-space. For $\defThn\in \Thn$ and $\defS\geq 0$, there is an isomorphism of categories
\[ \tndnec{L(\Sp \times X)_{-,\defThn,\defS}}{0}{m}\cong X_{\defThn,\defS}^{\times m}, \]
where the set $X_{\defThn,\defS}^{\times m}$ is seen as a discrete category.
\end{prop}

\begin{proof}
Using that there is an isomorphism of categories 
\[ \tndnec{\Sp}{0}{m} = \{\id_{\Sp}\} \xrightarrow{\cong} \{\Sp \hookrightarrow \repD\} \]
and by \cref{computefibers} applied to the identity map $X\hookrightarrow X$, there are isomorphisms of categories
\begin{align*}
\tndnec{L(\Sp \times X)_{-,\defThn,\defS}}{0}{m}&\cong\tndnec{\Sp}{0}{m}\times_{\tndnec{\repD}{0}{m}} \tndnec{L(\repD\times X)_{-,\defThn,\defS}}{0}{m}\\
&\cong \fib_{\Sp \to \repD[m]}(\QmXY_{-,\defThn,\defS})_! \cong X^{\times m}_{\defThn, \defS}. \qedhere
 \end{align*}
\end{proof}

\subsection{Auxiliary results about weighted colimits over \texorpdfstring{$\tndnec{\repD}{0}{m}$}{Nec}} \label{Auxiliary results about weighted colimits}

Recall from \cite[Theorem 2.1.2]{LRfib} that there is an equivalence between the categories of functors $(\tndnec{\repD}{0}{m})^{\op}\to \set$ and of discrete fibrations over $\tndnec{\repD}{0}{m}$. We now identify the set-valued functor corresponding to the discrete fibration $(Q_{-,\defThn,\defS})_!$ under this equivalence.

\begin{notation}
    Let $m\geq 1$, $Y$ be a connected $\Thn$-space, and $X\hookrightarrow Y$ be a monomorphism in $\Thnsset$. We define a functor
\[ G(X\hookrightarrow Y)\colon (\tndnec{\repD}{0}{m})^{\op}\to \Thnsset \] 
given on objects by
\[(T\hookrightarrow \repD_{0,m})\mapsto\begin{cases}
\prod_{B(T)} Y & \text{if} \;\; T\neq \repD \\
X\; (\cong \prod_{B(\repD)} X)  & \text{if} \;\; T=\repD,
\end{cases}\] 
and on morphisms by 
\[(f\colon U\hookrightarrow T)\mapsto\begin{cases}
 B(f)^*\colon \prod_{B(T)} Y\to \prod_{B(U)} Y & \text{if} \;\; U,T\neq \repD\\
X\hookrightarrow Y\xrightarrow{B(f)^*} \prod_{B(U)}Y & \text{if} \;\; U\neq \repD,  \; T=\repD,
\end{cases}\] 
where $B(f)^*$ is given by pre-composition with $B(f)\colon B(U)\to B(T)$ from \cref{Bonarrows}. 

For $\defThn\in \Thn$ and $\defS\geq 0$, we write $G(X\hookrightarrow Y)_{\defThn,\defS}$ for the composite
\[ G(X\hookrightarrow Y)_{\defThn,\defS}\colon (\tndnec{\repD}{0}{m})^{\op}\xrightarrow{G(X\hookrightarrow Y)} \Thnsset\xrightarrow{(-)_{\defThn,\defS}} \set. \]
\end{notation}

\begin{prop}\label{prop:Qdiscfib}
    Let $m\geq 1$, $Y$ be a connected $\Thn$-space, and $X\hookrightarrow Y$ be a monomorphism in $\Thnsset$. For $\defThn\in \Thn$ and $\defS\geq 0$, the discrete fibration 
\[ (Q_{-,\defThn,\defS})_!\colon \tndnec{\PP_{-,\defThn,\defS}}{0}{m}\to \tndnec{\repD}{0}{m} \]
corresponds to the functor
\[ G(X\hookrightarrow Y)_{\defThn,\defS}\colon (\tndnec{\repD}{0}{m})^{\op}\to \set. \]
\end{prop}

\begin{proof}
    We show that there is a natural isomorphism between the functor obtained from $(Q_{-,\defThn,\defS})_!$ by taking fibers and the functor $G(X\hookrightarrow Y)_{\defThn,\defS}$. In \cref{computefibers}, we have shown that, for every $T\hookrightarrow \repD_{0,m}$ in $\tndnec{\repD}{0}{m}$, there is an isomorphism of sets
    \[ \mathrm{fib}_{T\hookrightarrow \repD_{0,m}}((\QmXY_{-,\defThn,\defS})_!)\cong\begin{cases}
\prod_{B(T)} Y_{\defThn,\defS} & \text{if} \;\; T\neq \repD \\
X_{\defThn,\defS} & \text{if} \;\; T=\repD 
\end{cases} = G(X\hookrightarrow Y)_{\defThn,\defS}(T\hookrightarrow \repD_{0,m}). \]
    It remains to show that these isomorphisms are natural. For this, note that if $f\colon U\hookrightarrow T$ is a map in $\tndnec{\repD}{0}{m}$, then by the proof of \cref{prop:discretefib}, the map $f$ acts on the fibers of $(Q_{-,\defThn,\defS})_!$ by pre-composition
    \[ f^*\colon \mathrm{fib}_{T\to\repD_{0,m}}((Q_{-,\defThn,\defS})_!)\to \mathrm{fib}_{U\to\repD_{0,m}}((Q_{-,\defThn,\defS})_!). \] 
    A direct computation using this description and the definition of $G(X\hookrightarrow Y)_{\defThn,\defS}$ on morphisms shows that the isomorphisms of \cref{computefibers} assemble into a natural isomorphism. 
\end{proof}

We now use this to express the hom $\Thn$-spaces of $\Ch\PP$ in terms of certain weighted colimits. We refer the reader to e.g.~\cite[\textsection~3.1]{Kelly} for an account on the theory of weighted colimits.

\begin{rmk} \label{defnofweightedcolimit}
    Here we will be interested in two cases of weighted colimits: the ordinary weighted colimits and the simplicially enriched weighted colimits. We recall the definition of these weighted colimits in our case of interest. 
    
    Let $\cA$ and $\cD$ be small categories. Given functors $W\colon \cA^{\op}\to (s)\set$ and $F\colon \cA\to (s)\set^{\cD^{\op}}$, the weighted colimit of $F$ by $W$ can be computed using \cite[(3.70)]{Kelly} as the coequalizer in $(s)\set^{\cD^{\op}}$
    \[ \textstyle\colim^W_\cA F\cong \mathrm{coeq}(\coprod_{a\to a'\in \cA} F(a)\times W(a')\rightrightarrows \coprod_{a\in \cA} F(a)\times W(a)). \]
\end{rmk}

 We first introduce the following notation.

\begin{notation}
For $m\geq 1$, we define a functor
\[H_m\colon \tndnec{\repD}{0}{m}\to \sset\] 
given on objects by 
\[  (T\hookrightarrow \repD_{0,m} )\mapsto \Hom_{\CL T}(\alpha,\omega) \]
and on morphisms by 
\[ (f\colon U\hookrightarrow T)\mapsto ((\CL f)_{\alpha,\omega}\colon \Hom_{\CL U}(\alpha,\omega)\to \Hom_{\CL T}(\alpha,\omega)). \]
\end{notation}

Given the description of the hom $\Thn$-spaces of $\Ch$ from \cref{cor:computationshomC1ordered}, we are interested in understanding the colimit featured in the following proposition.

\begin{prop} \label{prop:homasweighted1}
Let $m\geq 1$, $Y$ be a connected $\Thn$-space, and $X\hookrightarrow Y$ be a monomorphism in $\Thnsset$. Then there is an isomorphism in $\Thnsset$
\[ \colim_{T\in \tndnec{\PP_{-,\star,\star}}{0}{m}} \Hom_{\CL T}(\alpha,\omega)\cong \colim^{G(X\hookrightarrow Y)_{\star,\star}}_{\tndnec{\repD}{0}{m}} H_m,   \]
where $\colim_{T\in \catnec{\PP_{-,\star,\star}}{a}{b}} \Hom_{\CL T}(\alpha,\omega)$ is the $\Thn$-bi-space given at $\defThn\in\Thn$ and $\defS\geq 0$ by the colimit in $\sset$
\[\colim_{T\in \catnec{\PP_{-,\defThn,\defS}}{a}{b}} \Hom_{\CL T}(\alpha,\omega), \] and $\colim^{G(X\hookrightarrow Y)_{\star,\star}}_{\tndnec{\repD}{0}{m}} H_m$ is the $\Thn$-bi-space given at $\defThn\in \Thn$ and $\defS\geq 0$ by the colimit in~$\sset$ of $H_m$ weighted by $G(X\hookrightarrow Y)_{\defThn,\defS}$. 
\end{prop}

\begin{proof}
Let $\defThn\in \Thn$ and $\defS\geq 0$. Recall from \cref{prop:Qdiscfib} that the category of elements of the functor $G(X\hookrightarrow Y)_{\defThn,\defS}$ is given by the discrete fibration
\[ (Q_{-,\defThn,\defS})_!\colon \tndnec{\PP_{-,\defThn,\defS}}{0}{m}\to \tndnec{\repD}{0}{m}.\] 
So by \cite[(7.1.8)]{RiehlCHT} we have isomorphisms in $\sset$
\begin{align*}
 \colim_{T\in \tndnec{\PP_{-,\defThn,\defS}}{0}{m}} \Hom_{\CL T}(\alpha,\omega) &\cong \colim_{\tndnec{\PP_{-,\defThn,\defS}}{0}{m}} H_m \circ  (Q_{-,\defThn,\defS})_! \\ &\cong \colim^{G(X\hookrightarrow Y)_{\defThn,\defS}}_{\tndnec{\repD}{0}{m}} H_m. \qedhere
 \end{align*}
\end{proof}

\begin{lemma} \label{lem:invertingweights}
Let $\cA$ and $\cD$ be small categories, and $F\colon \cA\to \sset$ and $W\colon \cA^{\op}\to \set^{\cD^{\op}}$ be functors. Write $\iota\colon \set^{\cD^{\op}}\hookrightarrow \sset^{\cD^{\op}}$ for the canonical inclusion and note that $\sset^{\cD^{\op}}$ is canonically enriched over $\sset$. Then there is an isomorphism in $\sset^{\cD^{\op}}$
\[ \colim^{W_\star}_{\cA} F \cong \colim^{F}_{\cA^{\op}} \iota W \]
where $\colim^{W_\star}_{\cA} F\colon \cD^{\op}\to \sset$ is the functor sending an object $d\in \cD$ to the colimit of the functor~$F$ weighted by 
\[ W_d\coloneqq \cA^{\op}\xrightarrow{W} \set^{\cD^{\op}}\xrightarrow{\ev_d} \set \]
and $\colim^{F}_{\cA^{\op}} \iota W$ is the $\sset$-enriched colimit of $\iota W\colon \cA^{\op}\to \sset^{\cD^{\op}}$ weighted by $F$.
\end{lemma}

\begin{proof}

Using \cref{defnofweightedcolimit}, for every $d\in \cD$, there is an isomorphism in $\sset$
\begin{align*}
    \colim^{W_d}_{\cA} F & \textstyle\cong \mathrm{coeq}\left(\coprod_{a\to a'\in \cA} W_d(a)\times F(a') \rightrightarrows \coprod_{a\in \cA} W_d(a)\times F(a)\right)
\end{align*} 
natural in $d\in \cD$. Hence this yields an isomorphism in $\sset^{\cD^{\op}}$
\[ \textstyle \colim^{W_\star}_{\cA} F\cong \mathrm{coeq}\left(\coprod_{a\to a'\in \cA} W(a)\times F(a') \rightrightarrows \coprod_{a\in \cA} W(a)\times F(a)\right).  \]
On the other hand, again by \cref{defnofweightedcolimit}, we have an isomorphism in $\sset^{\cD^{\op}}$
\begin{align*}  \colim^{F}_{\cA^{\op}} \iota W & \textstyle\cong \mathrm{coeq}\left(\coprod_{a'\to a\in \cA^{\op}} F(a')\times W(a) \rightrightarrows \coprod_{a\in \cA} F(a)\times W(a)\right) 
 \\ 
 &\textstyle \cong \mathrm{coeq}\left(\coprod_{a\to a'\in \cA} W(a)\times F(a') \rightrightarrows \coprod_{a\in \cA} W(a)\times F(a)\right).
\end{align*}
Hence, we get the desired isomorphism. 
\end{proof}

\begin{rmk} \label{rmk:enrichment}
Let $\varphi\colon \sset^{\DThnop}\cong \Thnssset$ be one of the two canonical isomorphism, and consider the inclusion
\[ \iota\colon\Thnsset\cong \set^{\DThnop}\hookrightarrow \sset^{\DThnop} \stackrel{\varphi}{\cong} \Thnssset. \]
Then $\sset^{\DThnop}$ is canonically enriched over $\sset$ and we consider the $\sset$-enrichment of $\Thnssset$ via $\varphi\colon \sset^{\DThnop}\cong \Thnssset$.
\end{rmk}

\begin{prop} \label{homweights}
Let $m\geq 1$, $Y$ be a connected $\Thn$-space, and $X\hookrightarrow Y$ be a monomorphism in $\Thnsset$. We have the following isomorphism in $\Thnssset$
\[ \colim_{T\in \tndnec{\PP_{-,\defThn,\defS}}{0}{m}} \Hom_{\CL T}(\alpha,\omega)\cong  \colim^{H_m}_{(\tndnec{\repD}{0}{m})^{\op}} \iota G(X\hookrightarrow Y),   \]
where $\colim^{H_m}_{(\tndnec{\repD}{0}{m})^{\op}} \iota G(X\hookrightarrow Y)$ is the $\sset$-enriched colimit of $\iota G(X\hookrightarrow Y)$ weighted by~$H_m$. 
\end{prop}

\begin{proof}
This is obtained by taking in \cref{lem:invertingweights} $\cA=\tndnec{\repD}{0}{m}$, $\cD=\ThnS$, $F=H_m$, and $W=G(X\hookrightarrow Y)$ and combining with \cref{prop:homasweighted1}.
\end{proof}

We further compute the colimit of the functor $G(X\hookrightarrow X)$, which will be useful to describe the hom $\Thn$-spaces of the categorification $\Ch(L(\Sp\times X))$.

\begin{prop} \label{homstrictasweighted0}
Let $m\geq 1$ and $X$ be a connected $\Thn$-space. Then there is an isomorphism in $\Thnsset$
\[ \colim_{\tndnec{\repD}{0}{m}} G(X\hookrightarrow X)\cong \colim^{\repS[0]}_{(\tndnec{\repD}{0}{m})^{\op}} G(X\hookrightarrow X)\cong X^{\times m}. \]
\end{prop}

\begin{proof}
Consider the canonical inclusion $\Sp\hookrightarrow \repD_{0,m}$. Then $|B(\Sp)|=m$ and there is a canonical isomorphism $\prod_{B(\Sp)} X\cong X^{\times m}$ in $\Thnsset$. 

We define a natural cone $\gamma$ under $G(X\hookrightarrow X)$ with summit $X^{\times m}$ as follows. Given a necklace $T\hookrightarrow \repD_{0,m}$ of $(\tndnec{\repD}{0}{m}$, we construct a necklace $\overline{T}\hookrightarrow \repD_{0,m}$ with the same set of joints as $T$, but with vertex set all vertices of 
$\repD_{0,m}$. Then there are canonical inclusions $j\colon \Sp\hookrightarrow \overline{T}$ and $T\hookrightarrow \overline{T}$, and moreover $B(T)\cong B(\overline{T})$. We define the component $\gamma_T$ to be the composite
\[ \textstyle \gamma_T\coloneqq \left( G(X\hookrightarrow X)(T)=\prod_{B(T)} X\cong \prod_{B(\overline{T})} X\xrightarrow{B(j)^*} \prod_{B(\Sp)} X\cong X^{\times m}\right). \]
Note that $\gamma$ is natural in $T\hookrightarrow \repD_{0,m}$ in $\tndnec{\repD}{0}{m}$. Indeed, this follows from the fact that, if $f\colon U\hookrightarrow T$ is a map in $(\tndnec{\repD}{0}{m}$, then it induces a map $\overline{f}\colon \overline{U}\hookrightarrow\overline{T}$ under $\Sp$.

We show that $\gamma$ is a colimit cone. Let $\delta$ be a cone under $G(X\hookrightarrow X)$ with summit $Y\in \Thnsset$. Define a map $d\colon X^{\times m}\to Y$ to be the following composite
\[ \textstyle d\coloneqq \left( X^{\times m}\cong \prod_{B(\Sp)} X=G(X\hookrightarrow X)(\Sp)\xrightarrow{\delta_{\Sp}} Y\right). \]
Then, by naturality of $\delta$, we have that $d\circ \gamma=\delta$. Moreover, the map $d$ is the unique map $X^{\times m} \to Y$ with this property as $\gamma_{\Sp}$ is given by the canonical isomorphism $\prod_{B(\Sp)} X\cong X^{\times m}$.
\end{proof}

Finally, we record here the following useful facts that depend on results postponed to \cref{Projective cofibrancy results}.

 \begin{prop} \label{wcobyHmLQ}
For $m\geq 1$, the functor
\[ \colim^{H_m}_{(\tndnec{\repD}{0}{m})^{\op}} (-)\colon (\MSThndiag)^{(\tndnec{\repD}{0}{m})^{\op}}_\inj\to \MSThndiag \]
given by taking the $\sset$-enriched colimit weighted by $H_m$ is left Quillen. 
\end{prop}

\begin{proof}
As we will see in \cref{Hmprojcof}, the functor $H_m\colon \tndnec{\repD}{0}{m}\to \MSspace$ is projectively cofibrant, and so the result follows from \cite[Theorem~3.3]{Gambino} by considering the model structure $\MSThndiag$ as enriched over $\MSspace$ in the correct variable as in \cref{rmk:enrichment}. 
\end{proof}

 \begin{prop} \label{GXXLQE}
 Let $m\geq 1$ and $X$ be a connected $\Thn$-space. The functor 
\[ \colim^{(-)}_{(\tndnec{\repD}{0}{m})^{\op}} \iota G(X\hookrightarrow X)\colon (\MSspace)^{(\tndnec{\repD}{0}{m})^{\op}}_\inj\to \MSThndiag \]
given by taking the $\sset$-enriched colimit of the functor $\iota G(X\hookrightarrow X)$ weighted by a functor $\tndnec{\repD}{0}{m}\to \sset$
is left Quillen. 
 \end{prop}
 
 \begin{proof}
As we will see in \cref{iotaGXprojcof}, the functor 
\[ \iota G(X\hookrightarrow X)\colon {(\tndnec{\repD}{0}{m})^{\op}}\to \MSThndiag \]
is projectively cofibrant, and so the result follows from \cite[Theorem~3.3]{Gambino} by considering the model structure $\MSThndiag$ as enriched over $\MSspace$ in the correct variable as in \cref{rmk:enrichment}.
\end{proof}

\subsection{Study of \texorpdfstring{$\Ch\PP\to \Sh_m Y$}{CP(X,Y)->ShY}} \label{Study of CPP}

We are now ready to give an explicit description of $\Ch\PP$ and study the image under $\Ch$ of the canonical map $I\colon \PP\to L(\repD\times Y)$. 

\begin{notation}
For $m\geq 0$, we write $\Sh_m$ for the functor
\[ \Sh_m\colon \Thnsset \xrightarrow{\repD\times (-)} \sThnsset \xrightarrow{L} \pcatThn\xrightarrow{\Ch} \Thncat \]
and $\partial\Sh_m$ for its \emph{boundary} 
\[ \partial\Sh_m\colon \Thnsset \xrightarrow{\partial\repD\times (-)} \sThnsset \xrightarrow{L} \pcatThn\xrightarrow{\Ch} \Thncat. \]
\end{notation}

By applying $\Ch$ to the diagram of \cref{notationPP}, as $\Ch$ commutes with colimits, we have the following diagram in $\Thncat$.
\begin{tz}
\node[](1) {$\partial \Sh_m X$};
\node[below of=1](2) {$\partial \Sh_m Y$}; 
\node[right of=1,xshift=1.4cm](3) {$\Sh_m X$}; 
\node[below of=3](4) {$\Ch \PP$};
\node[below right of=4,xshift=1.5cm](5) {$\Sh_m Y$}; 
\pushout{4};

\draw[->] (1) to (3);
\draw[->] (1) to (2);
\draw[->] (3) to (4);
\draw[->] (2) to (4);
\draw[->,bend left] (3) to (5);
\draw[->,bend right=20] (2) to (5);
\draw[->,dashed] (4) to node[below,la,xshift=-5pt]{$\Ch\inc$} (5);
    \end{tz}

\begin{prop} \label{prop:computationpushprod}
Let $m\geq 1$, $Y$ be a connected $\Thn$-space, and $X\hookrightarrow Y$ be a monomorphism in $\Thnsset$. Then the $\Thnsset$-enriched category $\Ch\PP$ is the directed $\Thnsset$-enriched category such that: 
\begin{itemize}[leftmargin=0.6cm]
    \item its set of objects $\Ob(\Ch\PP)$ is $\{0,1,\ldots,m\}$, 
    \item for $0<j-i<m$, the hom $\Thn$-space $\Hom_{\Ch\PP}(i,j)$ is given by
    \[ \Hom_{\Ch\PP}(i,j)\cong \Hom_{\Sh_m Y}(i,j), \]
    \item the hom $\Thn$-space $\Hom_{\Ch\PP}(0,m)$ is given by
    \[ \Hom_{\Ch\PP}(0,m)\cong\diag( \colim^{H_m}_{(\tndnec{\repD}{0}{m})^{\op}} \iota G(X\hookrightarrow Y)) \]
\end{itemize}
\end{prop}

\begin{proof}
By \cref{PP0}, we have that $\Ob(\Ch\PP) = \PP_0 =\{0,1,\ldots,m \}$. Now recall from \cref{prop:PP1ordered} that $\PP_{-,\defThn,\defS}$ is $1$-ordered for all $\defThn\in\Thn$ and $\defS\geq 0$. Hence, we can apply \cref{cor:computationshomC1ordered} and so we get that, for all $0\leq i<j\leq m$, 
\[ \Hom_{\Ch\PP}(i,j)\cong \diag(\colim_{T\in \tndnec{\PP_{-,\star,\star}}{i}{j}} \Hom_{\CL T}(\alpha,\omega)). \]
As $L(\repD\times Y)\cong \PP[m][Y]$, we also get that, for all $0\leq i<j\leq m$, 
\[ \Hom_{\Sh_m Y}(i,j)\cong \diag(\colim_{T\in \tndnec{L(\repD\times Y)_{-,\star,\star}}{i}{j}} \Hom_{\CL T}(\alpha,\omega)). \]
Now, if $0<j-i<m$, by \cref{lemmanecklacePPvsLY}, we have a natural isomorphism of categories
\[ \tndnec{\PP_{-,\star,\star}}{i}{j}\cong \tndnec{L(\repD\times Y)_{-,\star,\star}}{i}{j} \]
so that $\Hom_{\Ch\PP}(i,j)\cong \Hom_{\Sh_m Y}(i,j)$. Finally, by \cref{homweights}, we get that
\begin{align*} \Hom_{\Ch\PP}(0,m) &\cong \diag(\colim_{T\in \tndnec{\PP_{-,\star,\star}}{0}{m}} \Hom_{\CL T}(\alpha,\omega)) \\
    &\cong\diag( \colim^{H_m}_{(\tndnec{\repD}{0}{m})^{\op}} \iota G(X\hookrightarrow Y))
    \end{align*}
which concludes the proof.
\end{proof}

 \begin{prop} \label{prop:lasthom}
 Let $m\geq 1$, $Y$ be a connected $\Thn$-space, and $X\hookrightarrow Y$ be a (trivial) cofibration in $\Thnsset$. Then the map
 \[ \Hom_{\Ch\PP}(0,m)\to \Hom_{\Sh_m Y}(0,m) \]
 is a (trivial) cofibration in $\MSThnsset$.
 \end{prop}
 
 \begin{proof}
 By \cref{prop:computationpushprod} applied once to $X\hookrightarrow Y$ and once to the identity $Y\hookrightarrow Y$, we have the following isomorphisms.
 \begin{tz}
 \node[](1) {$\Hom_{\Ch\PP}(0,m)$}; 
 \node[below of=1](2) {$\Hom_{\Sh_m Y}(0,m)$};
 \node[right of=1,xshift=4.7cm](1') {$\diag( \colim^{H_m}_{(\tndnec{\repD}{0}{m})^{\op}} \iota G(X\hookrightarrow Y))$}; 
 \node[below of=1'](2') {$\diag( \colim^{H_m}_{(\tndnec{\repD}{0}{m})^{\op}} \iota G(Y\hookrightarrow Y))$}; 
 \draw[->] (1) to node[above,la]{$\cong$} (1');
 \draw[->] (2) to node[below,la]{$\cong$} (2');
 \draw[->] (1) to (2); 
 \draw[->] (1') to (2');
 \end{tz}
 As $X\hookrightarrow Y$ is a (trivial) cofibration in $\MSThnsset$ and $\iota\colon \MSThnsset\to \MSThndiag$ is left Quillen by \cref{diagiotaQE}, then $\iota (X\hookrightarrow Y)$ is also a (trivial) cofibration in $\MSThndiag$. As (trivial) cofibrations are defined level-wise in $(\MSThndiag)^{(\tndnec{\repD}{0}{m})^{\op}}_\inj$, it is straightforward to check by unpacking the definitions that 
 \[ \iota G(X\hookrightarrow X)\to \iota G(X\hookrightarrow Y)\] 
 is a (trivial) cofibration in $(\MSThndiag)^{(\tndnec{\repD}{0}{m})^{\op}}_\inj$. By \cref{wcobyHmLQ}, the functor \[ \colim^{H_m}_{(\tndnec{\repD}{0}{m})^{\op}} (-)\colon (\MSThndiag)^{(\tndnec{\repD}{0}{m})^{\op}}_\inj\to \MSThndiag \]
 is left Quillen, and so 
 \[ \colim^{H_m}_{(\tndnec{\repD}{0}{m})^{\op}} \iota G(X\hookrightarrow Y)\to \colim^{H_m}_{(\tndnec{\repD}{0}{m})^{\op}} \iota G(Y\hookrightarrow Y) \]
 is a (trivial) cofibration in $\MSThndiag$. Finally, by \cref{diagiotaQE}, we have that the functor $\diag\colon \MSThndiag\to \MSThnsset$ is left Quillen, and so we conclude that the map 
 \[ \diag( \colim^{H_m}_{(\tndnec{\repD}{0}{m})^{\op}} \iota G(X\hookrightarrow Y))\to \diag( \colim^{H_m}_{(\tndnec{\repD}{0}{m})^{\op}} \iota G(Y\hookrightarrow Y)) \]
 is a (trivial) cofibration in $\MSThnsset$, as desired.
 \end{proof}

\subsection{Study of \texorpdfstring{$\Sh_m X\to \Sigma_m X$}{Sh X->Sigma X}} \label{Study of Sigmam}

We now show that the categorification of $L(\Sp\times X)$ is $\Sigma_m X$. Then we construct and study a $\Thnsset$-functor $\Sh_m X\to \Sigma_m X$, which will be shown in \cref{C preserves weak equivalences} to be a retract of the image under $\Ch$ of the map $L(\Sp\times X)\hookrightarrow L(\repD\times X)$.

\begin{lemma} \label{lem:Sh1vsSigma1}
Let $X$ be a connected $\Thn$-space. There is a natural isomorphism in $\Thncat$ 
\[ \Sh_1 X\cong\Sigma X. \]
\end{lemma}

\begin{proof}
We need to show that, if $X\in \Thnsset$ is connected, then $\Sh_1 X=\Ch(L(\repD[1]\times X))$ is isomorphic to $\Sigma X$. We first compute $\CL_*(L(\repD[1]\times X))$. For this, we apply the colimit-preserving functor $\CL_*\colon \sThnsset\to \sCat^{\DThnop}$ to the pushout in $\sThnsset$ from \cref{RmkPushout} describing $L(\repD[1]\times X)$. At $\defThn\in \Thn$ and $\defS\geq 0$, as $\CL$ commutes with colimits, we have 
\[ \textstyle (\CL_* X)_{\defThn,\defS}\cong \CL(X_{\defThn,\defS}) \cong \CL(\coprod_{X_{\defThn,\defS}} \repD[0])\cong \coprod_{X_{\defThn,\defS}} [0], \]
and we have
\[ \textstyle (\CL_* (\repD[1]\times X))_{\defThn,\defS}\cong \CL(\repD[1]\times X_{\defThn,\defS}) \cong \CL(\coprod_{X_{\defThn,\defS}} \repD[1])\cong \coprod_{X_{\defThn,\defS}} \CL \repD[1] \cong \coprod_{X_{\defThn,\defS}} \Sigma \repS[0]. \]
Hence $\CL_*(L(\repD[1]\times X))_{\defThn,\defS}$ is the below pushout in $\sCat$.
\begin{tz}
\node[right of=3,xshift=2cm](1) {$\coprod_2\coprod_{X_{\defThn,\defS}} [0]$};
\node[below of=1](2) {$\coprod_2 [0]$}; 
\node[right of=1,xshift=2.3cm](3) {$\coprod_{X_{\defThn,\defS}} \Sigma \repS[0]$}; 
\node[below of=3](4) {$\CL_*(L(\repD[1]\times X))_{\defThn,\defS}$}; 
\pushout{4};

\draw[->] (1) to (3);
\draw[->] (1) to (2);
\draw[->] (3) to (4);
\draw[->] (2) to (4);
\end{tz}
As $\Sigma\colon \sset\to {}^{\{0,1\}/}\sCat$ commutes with colimits, it takes the coproduct $X_{\defThn,\defS}=\coprod_{X_{\defThn,\defS}}\repS[0]$ to the above pushout, and so
\[ \textstyle \CL_*(L(\repD[1]\times X))_{\defThn,\defS}\cong \Sigma( \coprod_{X_{\defThn,\defS}}\repS[0])\cong \Sigma (X_{\defThn,\defS}). \]
This shows that $\CL_*(L(\repD[1]\times X))$ is the $\Thnssset$-enriched category $\Sigma (\iota X)$ and, by applying $\diag_*$, we get that $\Sh_1 X=\diag_*\CL_*(L(\repD[1]\times X))$ is the $\Thnsset$-enriched category $\Sigma X$, as desired. 
\end{proof}

\begin{cor} \label{rem:computeSigma}
Let $m\geq 1$ and $X$ be a connected $\Thn$-space. Then there is a natural isomorphism in $\Thncat$
\[ \Ch(L(\Sp\times X))\cong \Sigma_m X. \]
\end{cor}

\begin{proof}
Given $X\in \Thnsset$ connected, since $\Ch$ commutes with colimits and $\Sh_1=\Ch( L (\repD[1]\times (-)))$, we have a natural isomorphism in $\Thncat$
    \[ \Ch(L(\Sp\times X))=\Ch(L((\repD[1] \amalg_{\repD[0]} \ldots \amalg_{\repD[0]} \repD[1])\times X)) \cong \Sh_1 X\amalg_{[0]}\ldots \amalg_{[0]} \Sh_1 X. \]
    As on connected objects $\Sh_1$ coincides with $\Sigma$ by \cref{lem:Sh1vsSigma1} and $\Sigma_m=\Sigma\amalg_{[0]} \ldots \amalg_{[0]} \Sigma$, we have
    \[ \Sh_1 X\amalg_{[0]}\ldots \amalg_{[0]} \Sh_1 X\cong\Sigma X\amalg_{[0]}\ldots \amalg_{[0]} \Sigma X\cong \Sigma_m X \]
    and so we get the desired result.
\end{proof}

\begin{rmk}\label{homstrictasweighted}
Let $m\geq 1$ and $X$ be a connected $\Thn$-space. By \cref{homofSigmam,homstrictasweighted0} there are natural isomorphisms in $\Thnsset$
\begin{align*}
    \Hom_{\Sigma_m X}(0,m) &\cong X^{\times m} \cong \colim^{\repS[0]}_{(\tndnec{\repD}{0}{m})^{\op}} G(X\hookrightarrow X) \\ &\cong \diag(\colim^{\repS[0]}_{(\tndnec{\repD}{0}{m})^{\op}} \iota G(X\hookrightarrow X)).
\end{align*}
\end{rmk}

We now build the desired $\Thnsset$-enriched functor $\Sh_m X\to \Sigma_m X$ and show that it is a Dwyer-Kan equivalence.

\begin{prop} \label{prop:defnofShtoSigma}
Let $m\geq 0$ and $X$ be a connected $\Thn$-space. Then there is a natural $\Thnsset$-enriched functor $\Sh_mX\to\Sigma_mX$
such that 
\begin{itemize}[leftmargin=0.6cm]
\item it is the identity on objects,
\item for $0<j-i<m$ the following diagram commutes
\begin{tz}
 \node[](1) {$\Hom_{\Sh_m X}(i,j)$}; 
 \node[right of=1,xshift=2.7cm](2) {$\Hom_{\Sh_{j-i} X}(i,j)$};
 \node[below of=1](1') {$\Hom_{\Sigma_{m} X}(i,j)$}; 
 \node[below of=2](2') {$\Hom_{\Sigma_{j-i} X}(i,j)$}; 
 \draw[->] (1) to (1');
 \draw[->] (2) to (2');
 \draw[->] (1) to node[above,la]{$\cong$}  (2); 
 \draw[->] (1') to node[below,la]{$\cong$} (2');
 \end{tz}
 for full subcategories $\Sh_{j-i} X\subseteq \Sh_m X$ and $\Sigma_{j-i} X\subseteq \Sigma_m X$ spanned by the objects $i, i+1, \ldots, j$,
 \item the following diagram commutes
 \begin{tz}
 \node[](1) {$\Hom_{\Sh_m X}(0,m)$}; \node[below of=1](2) {$\Hom_{\Sigma_m X}(0,m)$};
 \node[right of=1,xshift=4.6cm](1') {$\diag( \colim^{H_m}_{(\tndnec{\repD}{0}{m})^{\op}} \iota G(X\hookrightarrow X))$}; 
 \node[below of=1'](2') {$\diag( \colim^{\repS[0]}_{(\tndnec{\repD}{0}{m})^{\op}} \iota G(X\hookrightarrow X))$}; 
 \draw[->] (1) to node[above,la]{$\cong$} (1');
 \draw[->] (2) to node[below,la]{$\cong$} (2');
 \draw[->] (1) to (2); 
 \draw[->] (1') to (2');
 \end{tz}
 where the right-hand map is induced by the unique map $H_m\to \repS[0]$ and the horizontal maps are the isomorphisms from \cref{prop:computationpushprod} applied to the identity at $X$ and from \cref{homstrictasweighted}.
 \end{itemize}
\end{prop}

\begin{proof}
The desired $\Thnsset$-enriched functor can be constructed by induction on $m\geq 0$. If $m=0$, it is the identity at $[0]$ and, if $m=1$, it coincides with the isomorphism from \cref{lem:Sh1vsSigma1}. If $m>1$, the construction is fully determined by the conditions.
\end{proof}

\begin{prop} \label{prop:homstrictvscoh}
 Let $m\geq 1$ and $X$ be a connected $\Thn$-space. Then the $\Thnsset$-enriched functor $\Sh_m X\to \Sigma_m X$ from \cref{prop:defnofShtoSigma} induces a weak equivalence in $\MSThnsset$
 \[ \Hom_{\Sh_m X}(0,m)\to \Hom_{\Sigma_m X}(0,m).\]
 \end{prop}
 
 \begin{proof}
 By \cref{prop:defnofShtoSigma}, we have the following isomorphisms.
 \begin{tz}
 \node[](1) {$\Hom_{\Sh_m X}(0,m)$}; 
 \node[below of=1](2) {$\Hom_{\Sigma_m X}(0,m)$};
 \node[right of=1,xshift=4.6cm](1') {$\diag( \colim^{H_m}_{(\tndnec{\repD}{0}{m})^{\op}} \iota G(X\hookrightarrow X))$}; 
 \node[below of=1'](2') {$\diag( \colim^{\repS[0]}_{(\tndnec{\repD}{0}{m})^{\op}} \iota G(X\hookrightarrow X))$}; 
 \draw[->] (1) to node[above,la]{$\cong$} (1');
 \draw[->] (2) to node[below,la]{$\cong$} (2');
 \draw[->] (1) to (2); 
 \draw[->] (1') to (2');
 \end{tz}
 Since the values of $H_m$ are contractible by \cite[Corollary 3.10]{DuggerSpivakRigidification}, then $H_m\to \repS[0]$ is a weak equivalence in $(\MSspace)^{(\tndnec{\repD}{0}{m})^{\op}}_\inj$. Hence, by \cref{GXXLQE}, the functor \[ \colim^{(-)}_{(\tndnec{\repD}{0}{m})^{\op}} \iota G(X\hookrightarrow X)\colon (\MSspace)^{(\tndnec{\repD}{0}{m})^{\op}}_\inj\to \MSThndiag \]
 preserves weak equivalences, and so 
 \[ \colim^{H_m}_{(\tndnec{\repD}{0}{m})^{\op}} \iota G(X\hookrightarrow X)\to \colim^{\repS[0]}_{(\tndnec{\repD}{0}{m})^{\op}} \iota G(X\hookrightarrow X) \]
 is a weak equivalence in $\MSThndiag$. Finally, by \cref{diagiotaQE}, we have that the functor $\diag\colon \MSThndiag\to \MSThnsset$ preserves weak equivalences, and so we conclude that the map 
 \[ \diag( \colim^{H_m}_{(\tndnec{\repD}{0}{m})^{\op}} \iota G(X\hookrightarrow X))\to \diag( \colim^{\repS[0]}_{(\tndnec{\repD}{0}{m})^{\op}} \iota G(X\hookrightarrow X)) \]
 is a weak equivalence in $\MSThnsset$, as desired.
 \end{proof}

\begin{cor} \label{SigmaandSh}
Let $m\geq 1$ and $X$ be a connected $\Thn$-space. Then the $\Thnsset$-enriched functor from \cref{prop:defnofShtoSigma} defines a weak equivalence in $\MSThncat$
\[ \Sh_m X\xrightarrow{\simeq}\Sigma_m X.\]
\end{cor}

\begin{proof}
We show this by induction on $m$. If $m=1$, then $\Sh_1 X\cong \Sigma X$ by \cref{lem:Sh1vsSigma1}. If $m>1$, first observe that $\Sh_m X$ and $\Sigma_m X$ are directed $\Thnsset$-enriched categories with set of objects $\{0,1,\ldots,m\}$ and the map $\Ob(\Sh_m X)\to \Ob(\Sigma_m X)$ is the identity. So it is enough to show that 
\[ \Hom_{\Sh_m X}(i,j)\to \Hom_{\Sigma_m X}(i,j) \]
is a weak equivalence in $\MSThnsset$, for all $0\leq i<j\leq m$. If $i=0$ and $j=m$, this is the content of \cref{prop:homstrictvscoh}. If $0<j-i<m$, using the isomorphisms from \cref{prop:defnofShtoSigma} for corresponding subcategories $\Sh_{m-1} X\subseteq \Sh_m X$ and $\Sigma_{m-1} X\subseteq \Sigma_m X$, we can conclude by induction. 
\end{proof}

\section{The homotopy coherent categorification is a Quillen equivalence} \label{Section 4}

The goal of this section is to prove the main theorem. Precisely, we show that $\Ch$ preserves cofibrations, respectively weak equivalences, in \cref{C preserves cofibrations}, respectively \cref{C preserves weak equivalences}, so that the adjunction $\Ch\dashv \Nh$ is a Quillen pair. Finally, in \cref{C is a Quillen equivalence}, we show that $\Ch\dashv \Nh$ is further a Quillen equivalence.

\subsection{\texorpdfstring{$\Ch$}{C} preserves cofibrations} \label{C preserves cofibrations}

In order to show that the functor $\Ch$ is left Quillen, we first prove that it preserves cofibrations. 
 
\begin{thm} \label{thm:Chprescof}
The functor $\Ch\colon \pcatinj\to \MSThncat$ preserves cofibrations. 
\end{thm}

\begin{proof}
By \cref{rem:gencof}, a set of generating cofibrations in $\pcatinj$ is given by the map $\emptyset\to \repD[0]$ together with all maps of the form
\[L((\partial\repD\hookrightarrow \repD)\widehat{\times} (\partial\repThn\hookrightarrow \repThn)\widehat{\times}(\partial \repS\hookrightarrow \repS))\]
for $m\geq 1$, $\defThn\in \Thn$, and $k\geq 0$.

First observe that the image of the map $\emptyset\to \repD[0]$ under $\Ch$ is the $\Thnsset$-enriched functor $\emptyset\to [0]$, which is a cofibration in $\MSThncat$. 

Now, let $m\geq 1$, $\defThn\in \Thn$, and $\defS\geq 0$. If we write
\[ (X\hookrightarrow Y)\coloneqq (\partial\repThn\hookrightarrow \repThn)\widehat{\times}(\partial \repS\hookrightarrow\repS),\]
the image under $\Ch$ of the map $L((\partial\repD\hookrightarrow \repD)\widehat{\times} (X\hookrightarrow Y))$ is the $\Thnsset$-enriched functor
\[ \Ch\inc\colon\Ch\PP\to \Sh_m Y. \]
By \cref{rmk:monoconnectedtarget}, the map $X\hookrightarrow Y$ is a monomorphism in $\Thnsset$ with $Y$ connected. Hence, by \cref{prop:computationpushprod}, we have that $\Ob(\Ch\PP)=\{0,1,\ldots,m\}=\Ob(\Sh_m Y)$ and, for all $0<j-i<m$, we have that
\[ \Hom_{\Ch\PP}(i,j)=\Hom_{\Sh_m Y}(i,j).\]
Moreover, by \cref{prop:lasthom}, the map 
\[ \Hom_{\Ch\PP}(0,m)\to \Hom_{\Sh_m Y}(0,m) \]
 is a cofibration in $\MSThnsset$. Applying \cref{pushoutlemma}, we conclude that the $\Thnsset$-enriched functor $\Ch I$ is a cofibration in $\MSThncat$, as desired.
\end{proof}

\subsection{\texorpdfstring{$\Ch$}{C} preserves weak equivalences} \label{C preserves weak equivalences}

We now show that the functor $\Ch$ preserves weak equivalences. For this, we first prove that it sends Dwyer-Kan equivalences between fibrant objects to weak equivalences. 

 \begin{prop} \label{thm:Chpreswefibrant}
 The functor $\Ch\colon \pcatinj\to \MSThncat$ sends Dwyer-Kan equivalences between fibrant objects to weak equivalences.
 \end{prop}

\begin{proof}
Let $f\colon W\to Z$ be a Dwyer-Kan equivalence between fibrant objects in $\pcatinj$. By definition, the functor $\Ho f\colon \Ho W\to \Ho Z$ is an equivalence of categories and, for all $a,b\in W_0$, the map $\Map_W(a,b)\to \Map_Z(fa,fb)$ is a weak equivalence in $\MSThnsset$. By \cref{lem:eqhtpycat}, we obtain that the functor
\[\Ho\Ch f\colon \Ho \Ch W\to \Ho\Ch Z \]
is an equivalence of categories. By $2$-out-of-$3$, using \cref{HomChigher} and the fact that weak equivalences in $\injThnspace$ are in particular weak equivalences in $\MSThnsset$, we get that the map
\[\Hom_{\Ch W}(a,b)\to \Hom_{\Ch Z}(fa,fb)\]
is a weak equivalence in $\MSThnsset$. Hence the $\Thnsset$-enriched functor
$\Ch f\colon\Ch W\to\Ch Z$ is a weak equivalence in $\MSThncat$, as desired.
\end{proof}

We now aim to prove that the functor $\Ch$ sends fibrant replacements in $\pcatinj$ as constructed 
in \cref{rem:genanodyne} to weak equivalences in $\MSThncat$. For this, we first show that $\Ch$ sends the map $L((\Sp\hookrightarrow \repD)\widehat{\times}(\partial\repThn\hookrightarrow\repThn)\widehat{\times}(\partial\repS\hookrightarrow \repS))$ to a trivial cofibration. 

\begin{lemma} \label{lem:spinehom}
Let $m\geq 1$ and $X$ be a connected $\Thn$-space. Then the functor $\Ch$ sends the map $L((\Sp \hookrightarrow \repD)\times X)$ to a $\Thnsset$-enriched functor $\Sigma_mX\to \Sh_m X$ such that the induced map on hom $\Thn$-spaces
\[\Hom_{\Sigma_mX}(0,m)\to \HomSh(0,m) \]
is given by the diagonal of the leg of the weighted colimit
\[  H_m(\Sp)\times \iota G(X\hookrightarrow X)(\Sp)\to \colim^{H_m}_{(\tndnec{\repD}{0}{m})^{\op}} \iota G(X\hookrightarrow X). \]
\end{lemma}

\begin{proof}
By \cref{rem:computeSigma}, the image under $\Ch$ of the map $L((\Sp \hookrightarrow \repD)\times X)$ in $\pcatThn$ is a $\Thnsset$-enriched functor of the form $\Sigma_m X\to \Sh_m X$. Then, by \cref{cor:computationshomC1ordered}, we have that the map in $\Thnsset$  
\[ \Hom_{\Sigma_m}(0,m)\to \HomSh(0,m)\]
    is the diagonal of the map in $\Thnssset$
    \[\colim_{T\in \tndnec{L(\Sp\times X)_{-,\star,\star}}{0}{m}} \Hom_{\CL T}(\alpha,\omega)\to \colim_{T\in \tndnec{L(\repD\times X)_{-,\star,\star}}{0}{m}} \Hom_{\CL T}(\alpha,\omega)\]
    induced at $\defThn\in \Thn$ and $\defS\geq 0$ by the inclusion of categories 
 \[ \tndnec{L(\Sp\times X)_{-,\defThn,\defS}}{0}{m}\hookrightarrow \tndnec{L(\repD\times X)_{-,\defThn,\defS}}{0}{m}. \] 
 Under the isomorphisms from \cref{prop:Necofspine,homweights}, this map in $\Thnssset$ corresponds to the leg of the weighted colimit
    \[ \repS[0]\times \iota X^{\times m} \cong H_m(\Sp)\times \iota G(X\hookrightarrow X)(\Sp)\to \colim^{H_m}_{(\tndnec{\repD}{0}{m})^{\op}} \iota G(X\hookrightarrow X), \]
    as desired. 
\end{proof}

\begin{lemma} \label{preNerveSegal}
Let $m\geq 1$ and $X$ be a connected $\Thn$-space. Then the functor $\Ch$ sends the trivial cofibration in $\pcatinj$
\[ L((\Sp \hookrightarrow \repD)\times X)\]
to a trivial cofibration in $\MSThncat$.
\end{lemma}

\begin{proof} 
By \cref{thm:Chprescof}, the cofibration $L((\Sp \hookrightarrow \repD)\times X)$ in $\pcatinj$ is sent by $\Ch$ to a cofibration in $\MSThncat$. It remains to show that it is also a weak equivalence in~$\MSThncat$. 

By \cref{lem:spinehom}, the image under $\Ch$ of $L((\Sp \hookrightarrow \repD)\times X)$ is a $\Thnsset$-enriched functor of the form $\Sigma_m X\to \Sh_m X$. We show by induction on $m\geq 1$ that its composite
\[ \Sigma_m X\to \Sh_m X\xrightarrow{\simeq} \Sigma_m X\] 
with the weak equivalence in $\MSThncat$ from \cref{SigmaandSh} is the identity. Then, by $2$-out-of-$3$, we can deduce that $\Sigma_m X\to \Sh_m X$ is a weak equivalence in $\MSThncat$, as desired. 

When $m=1$, this follows from \cref{lem:Sh1vsSigma1}. If $m>1$, recall that $\Sigma_m X$ and $\Sh_m X$ are directed $\Thnsset$-enriched categories with set of objects $\{0,1,\ldots,m\}$ and both $\Thnsset$-enriched functors act on objects as the identity. So it remains to show that, for all $0\leq i\leq j\leq m$, the following composite is the identity in $\Thnsset$.
\[ \Hom_{\Sigma_m X}(i,j)\to \Hom_{\Sh_m X}(i,j)\to \Hom_{\Sigma_m X}(i,j) \]
If $i=0$ and $j=m$, by \cref{prop:defnofShtoSigma,lem:spinehom}, the above composite can be identified with the diagonal of the following commutative triangle in $\Thnssset$.
\begin{tz}
    \node[](1) {$H_m(\Sp)\times \iota G(X\hookrightarrow X)(\Sp)$}; 
    \node[right of=1,xshift=5cm](2) {$\colim^{H_m}_{(\tndnec{\repD}{0}{m})^{\op}} \iota G(X\hookrightarrow X)$}; 
    \node[below of=2](3) {$\colim^{\repS[0]}_{(\tndnec{\repD}{0}{m})^{\op}} \iota G(X\hookrightarrow X)$}; 

    \draw[->] (1) to (2); 
    \draw[->] (2) to (3); 
    \draw[->] ($(1.south)+(2cm,0)$) to node[below,la]{$\cong$} ($(3.north)-(2.5cm,0)$);
\end{tz}
Hence it is the identity. Now, if $0<j-i<m$, we conclude by induction using the isomorphisms from \cref{prop:defnofShtoSigma}
 for corresponding subcategories $\Sh_{m-1} X\subseteq \Sh_m X$ and $\Sigma_{m-1} X\subseteq \Sigma_m X$.
\end{proof}

\begin{lemma} \label{prop:CpresSegal}
Let $m\geq 1$ and $X\hookrightarrow Y$ be a monomorphism in $\Thnsset$ of the form 
\[ (\partial \repThn\hookrightarrow \repThn)\widehat{\times} (\partial\repS\hookrightarrow \repS) \]
for $\defThn\in \Thn$ and $\defS\geq 0$. Then the functor $\Ch$ sends the trivial cofibration in $\pcatinj$
\[ L((\Sp \hookrightarrow\repD)\widehat{\times} (X\hookrightarrow Y)) \]
to a trivial cofibration in $\MSThncat$.
\end{lemma}

\begin{proof}
By \cref{thm:Chprescof}, the cofibration $L((\Sp  \hookrightarrow \repD)\widehat{\times} (X\hookrightarrow Y))$ in $\pcatinj$ is sent by $\Ch$ to a cofibration in $\MSThncat$. It remains to show that it is also a weak equivalence in $\MSThncat$. 

We first deal with the cases where $X\hookrightarrow Y$ is not one of the following maps in $\Thnsset$
\[ \emptyset\hookrightarrow \repS[0], \quad \repS[0]\amalg\repS[0]\hookrightarrow \repS[1], \quad \text{or} \quad \repS[0]\amalg\repS[0]\hookrightarrow \repThn[1;0] \]
so that $X$ and~$Y$ are both connected $\Thn$-spaces. In this case, using \cref{rem:computeSigma}, the functor $\Ch$ sends the pushout-product map $L((\Sp \hookrightarrow \repD)\widehat{\times} (X\hookrightarrow Y))$ to the 
canonical $\Thnsset$-enriched functor
\[\Sigma_m Y\amalg_{\Sigma_m X}  \Sh_m X\to\Sh_m Y.\]
This $\Thnsset$-enriched functor is the unique dashed arrow that fits into the following commutative diagram in $\MSThncat$,
\begin{tz}
\node[](1) {$\Sigma_m X$}; 
\node[below of=1](2) {$\Sigma_m Y$}; 
\node[right of=1,xshift=1.6cm](3) {$\Sh_m X$}; 
\node[below of=3](4) {$\Sigma_m Y\amalg_{\Sigma_m X}  \Sh_m X$}; 
\pushout{4};

\draw[right hook->] (1) to node[above,la]{$\simeq$} (3);
\draw[->] (1) to (2);
\draw[->] (3) to (4);
\draw[right hook->] (2) to node[above,la]{$\simeq$} (4);

\node[below right of=4,xshift=1.2cm](5) {$\Sh_m Y$}; 

\draw[->,bend left=25] (3) to (5); 
\draw[right hook->,bend right=15] (2) to node[below,la]{$\simeq$} (5);

\draw[->,dashed] (4) to node[below,la,xshift=-2pt]{$\exists!$} (5);
\end{tz}
where the top and bottom horizontal $\Thnsset$-enriched functors are the trivial cofibrations from \cref{preNerveSegal}, and the middle $\Thnsset$-enriched functor is a trivial cofibration as a pushout of a trivial cofibration. By $2$-out-of-$3$, it follows that the dashed $\Thnsset$-enriched functor is a weak equivalence in $\MSThncat$, as desired.

If instead $X=\repS[0]\amalg\repS[0]$ and $Y=\repThn[1;0]$ or $Y=\repS[1]$, one could adjust the argument above. The key fact is to observe that, in this case, the top horizontal $\Thnsset$-enriched functor in the relevant diagram is replaced by the coproduct of trivial cofibrations in $\MSThncat$
\[ \Sigma_m \repS[0]\amalg\Sigma_m\repS[0]\xrightarrow{\simeq} \Sh_m \repS[0]\amalg \Sh_m \repS[0],\] 
which is a trivial cofibration, too.

Finally, if $X=\emptyset$ and $Y=\Delta[0]$, one could also adjust the argument above noticing that the top horizontal $\Thnsset$-enriched functor in the relevant diagram is the identity at $\emptyset$.
\end{proof}

We now show that $\Ch$ sends the map $L((\partial\repD\hookrightarrow \repD)\widehat{\times} (X\hookrightarrow Y))$ to a trivial cofibration, where $X\hookrightarrow Y$ is a trivial cofibration in $\MSThnsset$. For this, we first need to identify a generating set of trivial cofibrations $X\hookrightarrow Y$ in $\MSThnsset$ with connected $Y$. 

\begin{rmk} \label{indecomposables}
    Every object $Y$ in $\Thnsset$ can be written as a coproduct $Y\cong\coprod_{[y]\in \pi_0 Y} Y_{[y]}$ in~$\Thnsset$, where $Y_{[y]}$ is the fiber of $Y\to \pi_0 Y$ at $[y]\in \pi_0 Y$ and is connected.
\end{rmk}

\begin{lemma} \label{lemmaGenCofs}
There exists a set $\cJ$ of generating 
trivial cofibrations in $\MSThnsset$ such that every map $X\hookrightarrow Y$ in $\cJ$ has $Y$ connected.
\end{lemma}

\begin{proof}
Let $\cJ'$ be a generating set of trivial cofibrations in $\MSThnsset$, and $f\colon X\hookrightarrow Y$ be a map in~$\cJ'$. Using \cref{indecomposables}, the map $f$ can be rewritten as a coproduct
\[ \textstyle f\colon X\cong \coprod_{[y]\in \pi_0 Y} f^{-1}(Y_{[y]})\hookrightarrow \coprod_{[y]\in \pi_0 Y} Y_{ [y]}\cong Y.  \]
For every $[y]\in \pi_0 Y$, observe that the map $f^{-1}(Y_{[y]})\hookrightarrow Y_{[y]}$ is a retract of $f\colon X\hookrightarrow Y$, hence a trivial cofibration in $\MSThnsset$. By setting
\[\cJ\coloneqq\{f^{-1}(Y_{[y]})\hookrightarrow Y_{[y]} \mid (f\colon X\hookrightarrow Y)\in\cJ', [y]\in \pi_0 Y \}\]
we see that $\cJ$ generates the same class as $\cJ'$, namely the class of trivial cofibrations of $\MSThnsset$. Moreover, note that $\cJ$ is a set as it is indexed by the set $\coprod_{X\hookrightarrow Y\in \cJ'} \pi_0 Y$.
\end{proof}

\begin{lemma} \label{prop:Cpresinjtrivcof}
Let $m\geq 1$ and $X\hookrightarrow Y$ be a trivial cofibration in $\MSThnsset$. Then the functor $\Ch$ sends the trivial cofibration in $\pcatinj$
\[ L((\partial\repD\hookrightarrow\repD)\widehat{\times} (X\hookrightarrow Y)) \]
to a trivial cofibration in $\MSThncat$.
\end{lemma}

\begin{proof}
Without loss of generality we can assume that $X\hookrightarrow Y$ belongs to $\cJ$, where $\cJ$ is a set of generating 
trivial cofibrations in $\MSThnsset$ as in \cref{lemmaGenCofs}. Then the functor $\Ch$ sends the map $L((\partial\repD\hookrightarrow\repD)\widehat{\times} (X\hookrightarrow Y))$ to the $\Thnsset$-enriched functor
\[\Ch\PP\to \Sh_m Y.\]
By construction of the set $\cJ$, the map $X\hookrightarrow Y$ is a monomorphism in $\Thnsset$ with $Y$ connected. Hence, by \cref{prop:computationpushprod}, we have that $\Ob(\Ch\PP)=\{0,1,\ldots,m\}=\Ob(\Sh_m Y)$ and, for all $0<j-i<m$, we have that
\[ \Hom_{\Ch\PP}(i,j)=\Hom_{\Sh_m Y}(i,j).\]
Moreover, by \cref{prop:lasthom}, the map 
\[ \Hom_{\Ch\PP}(0,m)\to \Hom_{\Sh_m Y}(0,m) \]
 is a trivial cofibration in $\MSThnsset$. Applying \cref{pushoutlemma}, we conclude that the $\Thnsset$-enriched functor $\Ch\PP\to \Sh_m Y$ is a trivial cofibration in $\MSThncat$, as desired.
\end{proof}

By assembling the above results, we get the following. 

\begin{prop} \label{thm:fibrepl}
Let $W$ be an object in $\pcatThn$. Then the functor $\Ch$ sends the fibrant replacement $W\to W^{\fib}$ in $\pcatinj$ to a weak equivalence in $\MSThncat$. 
\end{prop}

\begin{proof}
If $\cJ$ is a set of generating trivial cofibrations for $\MSThnsset$, by \cref{rem:genanodyne}, a fibrant replacement $W\to W^{\fib}$ is obtained as a transfinite composition of pushouts of maps of the form
\[L((\Sp  \hookrightarrow \repD)\widehat{\times} (\partial\repThn\hookrightarrow \repThn)\widehat{\times}(\partial \repS\hookrightarrow \repS))\]
for $m\geq 1$, $\defThn\in \Thn$, and $\defS\geq 0$, and of the form
\[ L((\partial\repD\hookrightarrow \repD)\widehat{\times} (X\hookrightarrow Y)) \]
for $m\geq 1$ and $X\hookrightarrow Y\in \cJ$. By \cref{prop:CpresSegal,prop:Cpresinjtrivcof}, we have that $\Ch$ sends every such map to a trivial cofibration in $\MSThncat$. As $\Ch$ commutes with colimits, the $\Thnsset$-enriched functor $\Ch W\to \Ch(W^{\fib})$ is a transfinite composition of pushouts of trivial cofibrations in $\MSThncat$, and so is also a trivial cofibration in $\MSThncat$. 
\end{proof}

We can now deduce the desired result. 

\begin{thm} \label{Cpreswe}
The functor $\Ch\colon \pcatinj\to \MSThncat$ preserves weak equivalences. 
\end{thm}

\begin{proof}
Let $W\to Z$ be a weak equivalence in $\pcatinj$. By \cref{weinpcat}, this means that the induced map $W^{\fib}\to Z^{\fib}$ between fibrant replacements is a Dwyer-Kan equivalence. Then, we have a commutative square in $\MSThncat$,
\begin{tz}
\node[](1) {$\Ch W$}; 
\node[below of=1](2) {$\Ch(W^{\fib})$}; 
\node[right of=1,xshift=1cm](3) {$\Ch Z$}; 
\node[below of=3](4) {$\Ch(Z^{\fib})$}; 
\draw[->] (1) to node[left,la]{$\simeq$} (2); 
\draw[->] (3) to node[right,la]{$\simeq$} (4); 
\draw[->] (2) to node[above,la]{$\simeq$} (4); 
\draw[->] (1) to (3);
\end{tz}
where the vertical $\Thnsset$-enriched functors are weak equivalences in $\MSThncat$ by \cref{thm:fibrepl}, and the bottom horizontal one is a weak equivalence in $\MSThncat$ by \cref{thm:Chpreswefibrant}. Hence $\Ch W\to \Ch Z$ is also a weak equivalence in $\MSThncat$ by $2$-out-of-$3$, and this shows the desired result. 
\end{proof}

\subsection{\texorpdfstring{$\Ch$}{C} is a Quillen equivalence} \label{C is a Quillen equivalence}

By assembling \cref{thm:Chprescof,Cpreswe}, the functor $\Ch$ preserves cofibrations and weak equivalences and so it is a left Quillen functor. Hence we have the following. 

\begin{thm} \label{QPChNh}
The adjunction 
\begin{tz}
\node[](1) {$\MSThncat$}; 
\node[right of=1,xshift=3.1cm](2) {$\pcatinj$}; 

\draw[->] ($(1.east)-(0,5pt)$) to node[below,la]{$\Nh$} ($(2.west)-(0,5pt)$);
\draw[->] ($(2.west)+(0,5pt)$) to node[above,la]{$\Ch$} ($(1.east)+(0,5pt)$);

\node[la] at ($(1.east)!0.5!(2.west)$) {$\bot$};
\end{tz}
is a Quillen pair.
\end{thm}

The goal of this section is to show that the Quillen pair $\Ch\dashv \Nh$ is in fact a Quillen equivalence. For this, we first compare it to the Quillen equivalence $c\dashv N$ recalled in \cref{subsec:strictnerve}.

\begin{prop} \label{NsimeqNh}
Let $\cC$ be a fibrant $\MSThnsset$-enriched category. The natural canonical map
\[N\cC\to \Nh\cC\]
is a weak equivalence in $\injsThnspace$, and so a weak equivalence in $\injsThnsset$ and $\pcatproj$.
\end{prop}

\begin{proof}
    By \cref{projQuillen,QPChNh}, the following functors are right Quillen
    \[ N\colon \MSThncat\to \pcatproj \quad \text{and} \quad \Nh\colon \MSThncat\to \pcatinj, \]
    and so, as $\cC$ is fibrant in $\MSThncat$, then $N\cC$ is fibrant in $\pcatproj$ and $\Nh\cC$ is fibrant in $\pcatinj$. In particular, they both satisfy the Segal condition and are such that, for every $m\geq 0$, the $\Thn$-spaces $(N\cC)_m$ and $(\Nh\cC)_m$ are fibrant in $\MSThnsset$.

    Next, observe that there is a canonical map $N\cC\to \Nh\cC$ induced by the $\Thnsset$-enriched functors
\[ \Sh_m (\repThn\times\repS)\to \Sigma_m (\repThn\times\repS) \]
from \cref{prop:defnofShtoSigma}, for $m\geq 0$, $\defThn\in \Thn$, and $\defS\geq 0$. At $m=0,1$, this map induces equalities
\[(N\cC)_0=\Ob\cC=(\Nh\cC)_0\quad\text{ and }\quad(N\cC)_1=\Mor\cC=(\Nh\cC)_1.\]
Given $m>1$, there is a commutative diagram in $\MSThnsset$ 
\begin{tz}
\node[](1) {$(N\cC)_m$}; 
\node[right of=1,xshift=3.3cm](2) {$(N\cC)_1\times^{(h)}_{(N\cC)_0}\ldots \times^{(h)}_{(N\cC)_0} (N\cC)_1$}; 
\node[below of=1](3) {$(\Nh\cC)_m$}; 
\node[below of=2](4) {$(\Nh\cC)_1\times^{(h)}_{(\Nh\cC)_0}\ldots \times^{(h)}_{(\Nh\cC)_0} (\Nh\cC)_1$}; 

\draw[->] (1) to (3); 
\draw[->] (1) to node[above,la]{$\cong$} (2);
\draw[->] (2) to node[right,la]{$\cong$} (4);
\draw[->] (3) to node[below,la]{$\simeq$} (4);
\end{tz}
where the horizontal maps are weak equivalences as $N\cC$ and $\Nh\cC$ satisfy the Segal condition. Then by $2$-out-of-$3$, the left-hand map is also a weak equivalence in $\MSThnsset$. Since the $\Thn$-spaces $(N\cC)_m$ and $(\Nh\cC)_m$ are fibrant in $\MSThnsset$, the map $(N\cC)_m\to (\Nh\cC)_m$ is in fact a weak equivalence in $\injThnspace$. This shows that $N\cC\to \Nh\cC$ is a weak equivalence in $\injsThnspace$.
\end{proof}

We can deduce from this result the desired Quillen equivalence.

\begin{thm} \label{ThmMain}
The adjunction 
\begin{tz}
\node[](1) {$\MSThncat$}; 
\node[right of=1,xshift=3.1cm](2) {$\pcatinj$}; 

\draw[->] ($(1.east)-(0,5pt)$) to node[below,la]{$\Nh$} ($(2.west)-(0,5pt)$);
\draw[->] ($(2.west)+(0,5pt)$) to node[above,la]{$\Ch$} ($(1.east)+(0,5pt)$);

\node[la] at ($(1.east)!0.5!(2.west)$) {$\bot$};
\end{tz}
is a Quillen equivalence.
\end{thm}

\begin{proof}
We have a triangle of right Quillen functors from \cref{prop:idQE,projQuillen,QPChNh}
\begin{tz}
\node[](1) {$\MSThncat$}; 
\node[above right of=1,xshift=3cm,yshift=-.3cm](2) {$\pcatinj$}; 
\node[below right of=1,xshift=3cm,yshift=.3cm](3) {$\pcatproj$}; 
\draw[->] ($(1.east)+(0,.2cm)$) to node[above,la]{$\Nh$} ($(2.west)-(0,5pt)$); 
\draw[->] (2) to node[right,la]{$\id$} node[left,la]{$\simeq$} (3); 
\draw[->] ($(1.east)-(0,.2cm)$) to node[above,la]{$\simeq$} node[below,la]{$N$} ($(3.west)+(0,5pt)$); 
\end{tz}
which commutes up to isomorphism at the level of homotopy categories by \cref{NsimeqNh}. Moreover, the functor $N$ and $\id$ are Quillen equivalences by \cref{projQuillen,prop:idQE}. Hence, by $2$-out-of-$3$, we conclude that $\Nh$ is also a Quillen equivalence. 
\end{proof}

\section{Projective cofibrancy results} \label{Projective cofibrancy results}

In this section we provide the proofs for some technical facts that have been used in the previous sections. As a preliminary tool, in
\cref{subsec:Cube} we give an alternative combinatorial description of the category $\tndnec{\repD}{0}{m}$ as the category $\Cube_m$. Then in \cref{Projective cofibrancy of GX}, respectively \cref{Projective cofibrancy of H}, we show that the functor 
\[ \textstyle H_m\colon \tndnec{\repD}{0}{m}\to \MSspace, \]
respectively the functor
\[ \textstyle \iota G(X\hookrightarrow X)\colon (\tndnec{\repD}{0}{m})^{\op}\to \MSThndiag, \]
is projectively cofibrant.

\subsection{Combinatorics of necklaces} \label{subsec:Cube}

Recall that $\tndnec{\repD}{0}{m}$ is the category of totally non-degenerate necklaces in $\repD_{0,m}$. By \cref{rem:NecFmposet}, its objects are monomorphisms $T\hookrightarrow \repD_{0,m}$ with $T$ a necklace and its morphisms are monomorphisms over $\repD_{0,m}$, and so it is a poset. We now describe this category in a more combinatorial way.

\begin{defn}
Let $m\geq 1$. We define the category $\Cube_m$ to be the poset such that
\begin{itemize}[leftmargin=0.6cm]
    \item its objects are pairs $(I,S)$ of subsets $I\subseteq S\subseteq \{1,\ldots,m-1\}$, 
    \item there is a morphism $(I',S')\to (I,S)$ if and only if $I'\subseteq I$ and $S=S'\cup I$.
\end{itemize}
By convention, the category $\Cube_1$ is the terminal category.
\end{defn}

\begin{rmk} \label{genmorCube}
    The category $\Cube_m$ is generated by two different kinds of morphisms, namely 
    \[ (I\setminus \{j\},S)\to (I,S) \quad \text{and} \quad (I\setminus \{j\},S\setminus \{j\})\to (I,S) \]
    for every object $(I,S)\in \Cube_m$ and every element $j\in I$.
\end{rmk}

\begin{prop}
For $m\geq 1$, there are assignments
\begin{align*}
(T\hookrightarrow \repD_{0,m})& \mapsto (I_T,S_T) \\
(T_{(I,S)}\hookrightarrow \repD_{0,m}) &\mapsfrom (I,S)
\end{align*}
that define an isomorphism of categories
\[\tndnec{\repD}{0}{m} \cong \Cube_m.\]
\end{prop}

\begin{proof}
We first construct the functor $\tndnec{\repD}{0}{m}\to \Cube_m$. Given $f\colon T\hookrightarrow \repD_{0,m}$ in $\tndnec{\repD}{0}{m}$, we set $(I_T,S_T)$ to be the object of $\Cube_m$ given by
\[ S_T\coloneqq \{ f(v)\mid v\in T_0\}\setminus \{0,m\}\subseteq \{1,\ldots,m-1\}=\repD_0\setminus \{0,m\} \]
and $I_T\coloneqq S_T\setminus J_T$, where 
\[ J_T\coloneqq \{ f(v)\mid v \text{ is a joint in } T\}\setminus \{0,m\}\subseteq S_T. \] 
Then, given a map $g\colon T'\hookrightarrow T$ in $\tndnec{\repD}{0}{m}$, as $g$ is a monomorphism by \cref{rem:NecFmposet}, we have that $S_{T'}\subseteq S_{T}$ and $J_{T}\subseteq J_{T'}$; thus $I_{T'}\subseteq I_T$. It remains to show that $S_{T}=S_{T'}\cup I_{T}$. For this, it is enough to see that each element of $S_{T}$ that is not in $S_{T'}$ is in $I_{T}$, i.e., is not the image of a joint of~$T$. But since $J_{T}\subseteq J_{T'}\subseteq S_{T'}$, then any image of a joint in $T$ is contained in $S_{T'}$. Hence we get a map $(I_{T'}, S_{T'})\to (I_T,S_T)$ in $\Cube_m$.

We now construct the functor $\Cube_m\to \tndnec{\repD}{0}{m}$. Given a pair $(I,S)$ in $\Cube_m$, we set $T_{(I,S)}\hookrightarrow \repD_{0,m}$ to be the necklace such that $T_{(I,S)}$ has set of vertices $S\cup \{0,m\}$ and set of joints $(S\setminus I)\cup \{0,m\}$. Then, given a map $(I',S')\to (I,S)$, we need to show that there is an induced monomorphism $T_{(I',S')}\hookrightarrow T_{(I,S)}$. Indeed, as $I'\subseteq I$ and $S=S'\cup I$, we have that $S'\subseteq S$ and
\[ S\setminus I= (S'\cup I)\setminus I\subseteq (S'\cup I')\setminus I'= S'\setminus I'.\]
Hence the set of vertices of $T_{(I',S')}$ is contained in that of $T_{(I,S)}$, and the set of joints of $T_{(I,S)}$ is contained in that of $T_{(I',S')}$. In particular, this says that every bead of~$T_{(I',S')}$ is sent in a bead of~$T_{(I,S)}$ and so there is a monomorphism $T_{(I',S')}\hookrightarrow T_{(I,S)}$. 

Clearly, the two constructions are inverse to each other and so we get the desired isomorphism of categories.
\end{proof}

We now aim to give a description of the bead functor $B\colon \tndnec{\repD}{0}{m}\to \set$ from \cref{Bonarrows} as a functor $\Cube_m\to \set$. 

\begin{defn}
    We construct a functor \[ \cB\colon \Cube_m\to \set. \]
    Given an object $(I,S)$ in $\Cube_m$, as $S\setminus I\subseteq \{1,\ldots,m-1\}$, write $S\setminus I=\{ s_1<s_2<\ldots<s_{t-1}\}$, and set $s_0\coloneqq 0$ and $s_t\coloneqq m$. We define $\cB(I,S)$ to be the set
    \[\cB(I,S)\coloneqq \left\{ \{ s\in S\mid s_{i-1}\leq s\leq s_{i}\} \mid 1\leq i\leq t \right\}\]
    and we refer to its elements as \emph{interval in $S$}. 
    
    Given a morphism $(I',S')\to (I,S)$ in $\Cube_m$, there is an induced assignment $\cB(I',S')\to \cB(I,S)$ sending an interval in $S'$ to the interval in $S$ that contains it. This is well-defined as $S'\subseteq S$ and $S\setminus I\subseteq S'\setminus I'$.
\end{defn}

\begin{lemma} \label{bISvsbeads}
For $m\geq 1$, the following diagram of categories commutes up to isomorphism. 
    \begin{tz}
\node[](1) {$\tndnec{\repD}{0}{m}$}; 
\node[right of=1,xshift=2cm](2) {$\Cube_m$}; 
\node(3) at ($(1.east)!0.5!(2.west)-(0,1.2cm)$) {$\set$}; 

\draw[->] (1) to node[above,la]{$\cong$} (2); 
\draw[->] ($(1.south east)-(5pt,0)$) to node[left,la,xshift=-5pt,pos=0.6]{$B$} (3); 
\draw[->] ($(2.south west)+(5pt,0)$) to node[right,la,xshift=5pt,pos=0.6]{$\cB$} (3);
\end{tz}
\end{lemma}

\begin{proof}
    Given a necklace $T\hookrightarrow \repD_{0,m}$, we have a canonical natural isomorphism of sets
    \[ B(T)\cong\cB(I_T,S_T),\]
    which can be constructed using the fact that the set $(S_T\setminus I_T)\cup \{0,m\}$ corresponds to the set of joints of $T$ and so an element of $\cB(I_T,S_T)$ corresponds to the data of all vertices of $\repD$ contained in a bead of $T$.
\end{proof}

Using this, we can now describe the functor $G(X\hookrightarrow X)\colon (\tndnec{\repD}{0}{m})^{\op}\to \Thnsset$ introduced in \cref{Auxiliary results about weighted colimits} as a functor $\Cube^{\op}_m\to \Thnsset$.

\begin{defn}
Let $m\geq 1$ and $X$ be a connected $\Thn$-space. We define a functor \[ \cG(X)\colon \Cube_m^{\op}\to \Thnsset \]
given on objects by 
\[ \textstyle (I,S)\mapsto \prod_{\cB(I,S)} X \] 
and on a morphism $(I',S')\to (I,S)$ by the map 
\[ \textstyle \prod_{\cB(I,S)} X\to \prod_{\cB(I',S')} X \] induced by pre-composition along the induced map $\cB(I',S')\to \cB(I,S)$. 
\end{defn}

\begin{prop}
Let $m\geq 1$ and $X$ be a connected $\Thn$-space. The following triangle of categories commutes up to isomorphism.
\begin{tz}
\node[](1) {$(\tndnec{\repD}{0}{m})^{\op}$}; 
\node[right of=1,xshift=2cm](2) {$\Cube_m^{\op}$}; 
\node(3) at ($(1.east)!0.5!(2.west)-(0,1.2cm)$) {$\Thnsset$}; 

\draw[->] (1) to node[above,la]{$\cong$} (2); 
\draw[->] ($(1.south east)-(5pt,0)$) to node[left,la,xshift=-5pt,pos=0.6]{$G(X\hookrightarrow X)$} (3); 
\draw[->] ($(2.south west)+(5pt,0)$) to node[right,la,xshift=5pt,pos=0.6]{$\cG(X)$} (3);
\end{tz}
\end{prop}

\begin{proof}
    This follows directly from \cref{bISvsbeads}.
\end{proof}

We now aim for a more combinatorial description of the functor $H_m\colon \tndnec{\repD}{0}{m}\to \sset$ introduced in \cref{Auxiliary results about weighted colimits} as a functor $\Cube_m\to \sset$. 

\begin{rmk}
    Given a morphism $(I',S')\to (I,S)$ in $\Cube_m$, there is a partition of $I$ as
    \[ I= I'\amalg (I\cap (S'\setminus I'))\amalg ((I\cup S')\setminus S'). \]
\end{rmk}

\begin{defn} \label{def:Hm}
Let $m\geq 1$. We define a functor 
\[ \cH_m\colon \Cube_m\to \sset \]
given on objects by the map
\[ \textstyle (I,S)\mapsto \prod_I \repS[1] \] 
and on a morphism $(I',S')\to (I,S)$ by 
\begin{tz}
\node[](0){$\text{}$};
    \node[right of=0,xshift=4.5cm](1) {$\cH_m(I',S')=\prod_{I'}\repS[1]$}; 
    \node[below of=1,xshift=.05cm](2) {$\cH_m(I,S)=\prod_I\repS[1]$};
    \punctuation{2}{.};
    
    \draw[->] ($(1.south)-(.8cm,0)$) to ($(2.north)-(.85cm,0)$);
    \draw[->] ($(1.south)+(.85cm,0)$) to node[right,la]{$\prod_{I'} \repS[1]\times \prod_{I\cap (S'\setminus I')}\langle 1\rangle \times \prod_{(I\cup S')\setminus S'}\langle 0 \rangle$} ($(2.north)+(.8cm,0)$);
    \end{tz}
\end{defn}

By a \cite[Corollary~3.10]{DuggerSpivakRigidification}, we have the following computations for the hom spaces of the categorification of necklaces.

\begin{lemma} \label{lem:homnecklaces}
    Let $T=\repD[m_1]\vee \ldots \vee \repD[m_t]$ be a necklace with $t\geq 1$ and $m_i\geq 1$ for $1\leq i\leq t$. Then there is a natural isomorphism in $\sset$
    \[ \textstyle \Hom_{\CL T}(\alpha,\omega)\cong \prod_{i=1}^t \prod_{[1,m_i-1]}\repS[1]\cong \prod_{\coprod_{i=1}^t [1,m_i-1]}\repS[1]. \]
\end{lemma}

\begin{prop}
For $m\geq 1$, the following triangle of categories commutes up to isomorphism
\begin{tz}
\node[](1) {$\tndnec{\repD}{0}{m}$}; 
\node[right of=1,xshift=2cm](2) {$\Cube_m$}; 
\node(3) at ($(1.east)!0.5!(2.west)-(0,1.2cm)$) {$\sset$}; 

\draw[->] (1) to node[above,la]{$\cong$} (2); 
\draw[->] ($(1.south east)-(5pt,0)$) to node[left,la,xshift=-5pt,pos=0.6]{$H_m$} (3); 
\draw[->] ($(2.south west)+(5pt,0)$) to node[right,la,xshift=5pt,pos=0.6]{$\cH_m$} (3);
\end{tz}
\end{prop}

\begin{proof}
Recall that $H_m$ sends a necklace $T=\repD[m_1]\vee \ldots \vee \repD[m_t]\hookrightarrow \repD_{0,m}$ to
    \[ \textstyle \Hom_{\CL T}(\alpha,\omega)\cong \prod_{\coprod_{i=1}^t [1,m_i-1]}\repS[1], \]
    where the isomorphism holds by \cref{lem:homnecklaces}. Note that the set $\coprod_{i=1}^t [1,m_i-1]$ can be made into a poset with the lexicographic order. Moreover, a direct computation shows that the posets $\coprod_{i=1}^t [1,m_i-1]$ and $I_T$ have the same cardinality, namely $\sum_{i=1}^t (m_i-1)$, and so there is a unique isomorphism $\coprod_{i=1}^t [1,m_i-1]\cong I_T$ preserving the order. This induces an isomorphism $\sset$
    \[\textstyle H_m(T)\cong \prod_{\coprod_{i=1}^t [1,m_i-1]}\repS[1]\cong \prod_{I_T} \repS[1]= \cH_m(I_T,S_T). \]
    It remains to show that this isomorphism is compatible with morphisms. 
    
    By \cref{genmorCube}, it is enough to check that it is compatible with the generating morphisms $(I_T\setminus \{j\},S_T)\to (I_T,S_T)$ and $(I_T\setminus \{j\}, S_T\setminus \{j\})\to (I_T,S_T)$ of $\Cube_m$, for all $j\in I_T$. Note that an element $j\in I_T$ corresponds to a vertex $\ell\in \repD[m_i]$ with $0<\ell<m_i$ for some $1\leq i\leq t$. 
    
    In the case $(I_T\setminus \{j\},S_T)\to (I_T,S_T)$, by definition of $\cH_m$, the induced map is given by 
    \begin{tz}
    \node[](1) {$\cH_m(I_T\setminus \{j\},S_T)=\prod_{I_T\setminus \{j\}}\repS[1]$}; 
    \node[below of=1,xshift=.15cm](2) {$\cH_m(I_T,S_T)=\prod_{I_T}\repS[1]$};
    \punctuation{2}{.};
    
    \draw[->] ($(1.south)-(1cm,0)$) to ($(2.north)-(1.15cm,0)$);
    \draw[->] ($(1.south)+(1.25cm,0)$) to node[right,la]{$(\prod_{I_T\setminus\{j\}} \repS[1])\times \langle 1\rangle$} ($(2.north)+(1.1cm,0)$);
    \end{tz} 
    Then, the necklace $U\hookrightarrow \repD_{0,m}$ corresponding to $(I_T\setminus \{j\},S_T)$ is the subnecklace of $T$ given by 
    \[ U\cong \repD[m_1]\vee \ldots \vee \repD[m_{i-1}] \vee \repD[\ell] \vee \repD[m_i-\ell]\vee \repD[m_{i+1}]\vee\ldots   \vee\repD[m_t] \]
    and the inclusion $U\hookrightarrow T$ is induced by $\repD[\ell]\vee \repD[m_i-\ell]\hookrightarrow \repD[m_i]$. The latter induces a map 
    \[ \Hom_{\CL[\ell]\amalg_{[0]} \CL[m_i-\ell]}(\alpha,\omega)\cong \Hom_{\CL[m_i]}(0,\ell)\times \Hom_{\CL[m_i]}(\ell,m_i)\to \Hom_{\CL[m_i]}(0,m_i)  \]
 which corresponds to the composition map of $\CL[m_i]$ as in \cref{defn:CL}. Hence the image under~$H_m$ of the inclusion $U\hookrightarrow T$ is given by 
 \begin{tz}
    \node[](1) {$H_m(U)=\prod_{(\coprod_{i=1}^t [1,m_i-1])\setminus \{\ell\}} \repS[1]$}; 
    \node[below of=1,xshift=-.4cm](2) {$H_m(T)=\prod_{\coprod_{i=1}^t [1,m_i-1]} \repS[1]$};
    \punctuation{2}{.};
    
    \draw[->] ($(1.south)-(2.1cm,0)$) to ($(2.north)-(1.7cm,0)$);
    \draw[->] ($(1.south)+(.5cm,0)$) to node[right,la]{$(\prod_{(\coprod_{i=1}^t [1,m_i-1])\setminus \{\ell\}} \repS[1]) \times \langle 1\rangle$} ($(2.north)+(.9cm,0)$);
    \end{tz}
    This shows that the isomorphisms are compatible with this first type of generating morphisms. 

    In the case $(I_T\setminus \{j\}, S_T\setminus \{j\})\to (I_T,S_T)$, by definition of $\cH_m$, the induced map is given by 
    \begin{tz}
    \node[](1) {$\cH_m(I_T\setminus \{j\},S_T\setminus\{j\})=\prod_{I_T\setminus \{j\}}\repS[1]$}; 
    \node[below of=1,xshift=.6cm](2) {$\cH_m(I_T,S_T)=\prod_{I_T}\repS[1]$};
    \punctuation{2}{.};
    
    \draw[->] ($(1.south)-(.7cm,0)$) to ($(2.north)-(1.3cm,0)$);
    \draw[->] ($(1.south)+(1.6cm,0)$) to node[right,la]{$(\prod_{I_T\setminus\{j\}} \repS[1])\times \langle 0\rangle$} ($(2.north)+(1cm,0)$);
    \end{tz} 
    Then, the necklace $U\hookrightarrow \repD_{0,m}$ corresponding to $(I_T\setminus \{j\},S_T\setminus \{j\})$ is the subnecklace of $T$ given by 
    \[ U\cong \repD[m_1]\vee \ldots \vee \repD[m_{i-1}] \vee \repD[m_i-1]\vee \repD[m_{i+1}]\ldots  \ldots \vee\repD[m_t]\]
    and the inclusion $U\hookrightarrow T$ is induced by the coface map $d^\ell\colon \repD[m_i-1]\to \repD[m_i]$. The latter induces a map 
    \[ \Hom_{\CL[m_i-1]}(0,m_i-1) \to \Hom_{\CL[m_i]}(0,m_i) \]
    as described in \cref{CLoncofaces}. Hence the image under $H_m$ of the inclusion $U\hookrightarrow T$ is given by 
 \begin{tz}
    \node[](1) {$H_m(U)=\prod_{(\coprod_{i=1}^t [1,m_i-1])\setminus \{\ell\}} \repS[1]$}; 
    \node[below of=1,xshift=-.4cm](2) {$H_m(T)=\prod_{\coprod_{i=1}^t [1,m_i-1]} \repS[1]$};
    \punctuation{2}{.};
    
    \draw[->] ($(1.south)-(2.1cm,0)$) to ($(2.north)-(1.7cm,0)$);
    \draw[->] ($(1.south)+(.5cm,0)$) to node[right,la]{$(\prod_{(\coprod_{i=1}^t [1,m_i-1])\setminus \{\ell\}} \repS[1]) \times \langle 0\rangle$} ($(2.north)+(.9cm,0)$);
    \end{tz}
    This shows that the isomorphisms are compatible with this second type of generating morphisms, and concludes the proof. 
\end{proof}

\subsection{Projective cofibrancy of \texorpdfstring{$\cG(X)$}{G(X)}} \label{Projective cofibrancy of GX}

In this section, we aim to show that the functor $\cG(X)$ is cofibrant in $(\MSThnsset)_{\proj}^{\Cube_m^{\op}}$. For this, all the results in this section are towards proving that $\cG(X)$ satisfies the left lifting property against all trivial fibrations in $(\MSThnsset)_{\proj}^{\Cube_m^{\op}}$. 

Let $\cP_m\coloneqq \cP( \{1,\ldots,m-1\})$ be the poset of subsets of $\{1,\ldots,m-1\}$ ordered by inclusion. Then there is an embedding
\[ \sigma\colon \cP_m\hookrightarrow \Cube_m, \quad I\mapsto (I,\{1,\ldots,m-1\}) \]
which admits a retraction 
\[ r\colon \Cube_m\to \cP_m, \quad (I,S)\mapsto I\cup (\{1,\ldots,m-1\}\setminus S). \]
Note that $r$ is well-defined since, given a morphism $(I',S')\to (I,S)$ in $\Cube_m$, then there is a morphism $r(I',S')\to r(I,S)$ in $\cP_m$ as, using that $S=S'\cup I$, we have
\[ I'\cup (\{1,\ldots,m-1\}\setminus S') \subseteq I\cup (\{1,\ldots,m-1\}\setminus S') =I\cup (\{1,\ldots,m-1\}\setminus S).\] 

It is straightforward to check that $r \sigma=\id_{\cP_m}$. Moreover, we have a natural transformation $\alpha\colon \id_{\Cube_m}\to \sigma r$ given at $(I,S)$ by the morphism in $\Cube_m$ 
\[ (I,S)\to (I\cup (\{1,\ldots,m-1\}\setminus S), \{1,\ldots,m-1\})=\sigma r(I,S) \]
which exists as $I\subseteq I\cup (\{1,\ldots,m-1\}\setminus S)$ and $\{1,\ldots,m-1\}= S\cup I\cup (\{1,\ldots,m-1\}\setminus S)$.

\begin{lemma} \label{beadsofrIS}
Let $m\geq 1$ and $(I,S)$ be an object in $\Cube_m$. Then the component $(I,S)\to \sigma r(I,S)$ of $\alpha$ induces a natural isomorphism of sets
\[ \cB(I,S)\cong \cB(\sigma r(I,S)). \]
\end{lemma}

\begin{proof}
First note that $\{1,\ldots,m-1\} \setminus (I\cup (\{1,\ldots,m-1\}\setminus S)) =  (\{1,\ldots,m-1\} \setminus I)\cap S=S\setminus I$. Write $S\setminus I=\{ s_1< \ldots < s_{t-1} \}$. Then we have 
\[ \cB(I,S) = \big\{ \{s\in S\mid s_{i-1}\leq s\leq s_i\} \mid 1\leq i\leq t \big\}, \]
\[ \cB(\sigma r(I,S)) = \big\{ \{s\in \{1,\ldots,m-1\} \mid s_{i-1}\leq s\leq s_i\} \mid 1\leq i\leq t \big\}. \]
So the map $(I,S)\to \sigma r(I,S)$ induces a canonical isomorphism between these sets given by 
\[ \{s\in S\mid s_{i-1}\leq s\leq s_i\}\mapsto \{s\in \{1,\ldots,m-1\} \mid s_{i-1}\leq s\leq s_i\}. \qedhere \]
\end{proof}

\begin{lemma} \label{Galphaisiso}
    Let $m\geq 1$ and $X$ be a connected $\Thn$-space. Then the natural transformation $\cG(X) \circ \alpha^{\op}\colon \cG(X) \circ \sigma^{\op} r^{\op}\to \cG(X)$ is an isomorphism in $(\Thnsset)^{\Cube_m^{\op}}$.
\end{lemma}

\begin{proof}
    The component at $(I,S)$ in $\Cube_m$ of $\cG(X)\circ \alpha^{\op}$ is given by the map 
    \[ \textstyle \prod_{\cB(\sigma r(I,S))} X\to \prod_{\cB(I,S)} X \]
    induced by pre-composing with the isomorphism of sets $\cB(I,S)\cong \cB(\sigma r(I,S))$ from \cref{beadsofrIS}. Hence, this is an isomorphism.
\end{proof}

\begin{cor} \label{cor:nattransfoutofGX}
    Let $m\geq 1$, $X$ be a connected $\Thn$-space, and $F\colon \Cube^{\op}_m\to \Thnsset$ be a functor. There is a natural isomorphism of sets 
    \[ (\Thnsset)^{\Cube_m^{\op}}(\cG(X),F)\cong (\Thnsset)^{\cP_m^{\op}}(\cG(X)\circ \sigma^{\op},F\circ \sigma^{\op}). \] 
\end{cor}

\begin{proof}
    We define maps in both directions by sending $\beta\colon \cG(X)\to F$ to $\beta\circ \sigma^{\op}\colon \cG(X)\circ \sigma^{\op} \to F\circ \sigma^{\op}$, and by sending $\gamma\colon \cG(X)\circ \sigma^{\op} \to F\circ \sigma^{\op}$ to the composite 
    \[ \cG(X)\xrightarrow{(\cG(X)\circ \alpha^{\op})^{-1}}\cG(X)\circ \sigma^{\op} r^{\op} \xrightarrow{\gamma\circ r^{\op}} F\circ \sigma^{\op} r^{\op} \xrightarrow{F\circ \alpha^{\op}} F, \]
    where $\cG(X)\circ \alpha^{\op}$ is invertible by \cref{Galphaisiso}. The fact that these constructions are inverse to each other is a consequence of the relation $r \sigma=\id_{\cP_m}$ and the naturality of $\alpha$.
\end{proof}

The following is a straightforward verification.

\begin{lemma} \label{lem:Pm12cofinal}
    Let $m\geq 1$. Write $\cP_m^{1,2}$ and $\cP_m^{\geq 1}$ for the sub-posets of $\cP_m$ given by
    \[ \cP_m^{1,2}=\{ I\subseteq \{1,\ldots,m-1\}\mid |I|=1,2 \} \quad \text{and} \quad \cP_m^{\geq 1}=\{ I\subseteq \{1,\ldots,m-1\}\mid |I|\geq 1\}. \]
    Then the inclusion $\cP_m^{1,2}\hookrightarrow \cP_m^{\geq 1}$ is cofinal, and so $(\cP_m^{1,2})^{\op}\hookrightarrow (\cP_m^{\geq 1})^{\op}$ is final.
\end{lemma}

\begin{lemma} \label{lem:colimitascoeq}
    Let $m\geq 1$, $X$ be a connected $\Thn$-space, and $I\subseteq \{1,\ldots,m-1\}$. Then there is an isomorphism in $\Thnsset$
    \[ \textstyle \colim_{I\subsetneq J\in \cP_m} \cG(X)(\sigma J)\cong \mathrm{coeq}\left(\coprod_{\substack{I\subsetneq J\in \cP_m \\
    |J|=|I|+2}} \cG(X)(\sigma J)\rightrightarrows \coprod_{\substack{I\subsetneq J\in \cP_m \\
    |J|=|I|+1}} \cG(X)(\sigma J) \right). \]
\end{lemma}

\begin{proof}
    Note that we have isomorphisms of posets 
    \[ \{I\subsetneq J\in \cP_m\mid |J|=|I|+1,|I|+2\}\cong \cP_{m-|I|}^{1,2}\quad \text{and}\quad \{I\subsetneq J\in \cP_m\} \cong \cP_{m-|I|}^{\geq 1}. \]
    Hence, by \cref{lem:Pm12cofinal}, the inclusion 
    \[ \{I\subsetneq J\in \cP_m\mid |J|=|I|+1,|I|+2\}^{\op}\hookrightarrow \{I\subsetneq J\in \cP_m \}^{\op} \]
    is final. Using the formula for colimits in terms of coequalizers as in the dual of \cite[Theorem~V.2.2]{MacLane}), we obtain the desired result.
\end{proof}

\begin{rmk} \label{rmk:descrmaps}
Let $m\geq 1$ and $I\subseteq \{1,\ldots,m-1\}$. Write $\{1,\ldots,m-1\}\setminus I=\{s_1<\ldots<s_{t-1}\}$ and set $s_0=0$, $s_t=m$. Recall that 
\[ \cB(\sigma I)=\{ [s_{i-1},s_i]\subseteq \{1,\ldots,m-1\} \mid 1\leq i\leq t\}. \] 
Then, for $1\leq j\leq t-1$, the map
\[ \cB(\sigma I)\to \cB(\sigma (I\amalg \{s_j\})) \]
is given by
\[ [s_{i-1},s_i] \mapsto \begin{cases}
[s_{i-1},s_i] & \text{if} \; 1\leq i\leq t, i\neq j,j+1 \\
[s_{j-1},s_{j+1}] & \text{if} \; i=j,j+1.
\end{cases} \]
\end{rmk}

\begin{lemma} \label{lem:pullbacksofGX}
    Let $m\geq 1$, $X$ be a connected $\Thn$-space, and $I\subseteq \{1,\ldots,m-1\}$. For all $j_0,j_1\in \{1,\ldots,m-1\}\setminus I$, there is a pullback square in $\Thnsset$
    \begin{tz}
\node[](1) {$\cG(X)(\sigma(I\amalg \{j_0,j_1\}))$}; 
\node[below of=1](2) {$\cG(X)(\sigma(I\amalg \{j_0\}))$}; 
\node[right of=1,xshift=2.9cm](3) {$\cG(X)(\sigma(I\amalg \{j_1\}))$}; 
\node[below of=3](4) {$\cG(X)(\sigma I)$}; 
\pullback{1};
\punctuation{4}{.};

\draw[->] (1) to (3);
\draw[->] (1) to (2);
\draw[->] (3) to (4);
\draw[->] (2) to (4);
\end{tz}
\end{lemma}

\begin{proof}
Recall that $\cG(X)(\sigma I)=\prod_{\cB(\sigma I)} X$. Hence, to show that the desired square is a pullback, as $\prod_{(-)} X\colon \set^{\op}\to \Thnsset$ sends colimits in $\set$ to limits in $\Thnsset$, it is enough to show that the following square is a pushout of sets.
\begin{tz}
\node[](1) {$\cB(\sigma I)$}; 
\node[below of=1](2) {$\cB(\sigma(I\amalg \{j_0\}))$}; 
\node[right of=1,xshift=2.9cm](3) {$\cB(\sigma(I\amalg \{j_1\}))$}; 
\node[below of=3](4) {$\cB(\sigma(I\amalg \{j_0,j_1\}))$}; 
\pushout{4};

\draw[->] (1) to (3);
\draw[->] (1) to (2);
\draw[->] (3) to (4);
\draw[->] (2) to (4);
\end{tz}
Now, as $j_0, j_1\in \{1,\ldots,m-1\}\setminus I$, by \cref{rmk:descrmaps}, for $\epsilon=0,1$, the map
\[ \cB(\sigma I)\to \cB(\sigma(I\amalg \{j_\epsilon\})) \]
identifies the intervals with end point $j_\epsilon$ and starting point $j_\epsilon$, and the map 
\[ \cB(\sigma(I\amalg \{j_\epsilon\}))\to \cB(\sigma(I\amalg \{j_0,j_1\})) \]
identifies the intervals with end point $j_{|\epsilon-1|}$ and starting point $j_{|\epsilon-1|}$. It is then clear from these descriptions that the above square is a pushout.
\end{proof}

\begin{lemma} \label{lem:colimcofib}
    Let $m\geq 1$, $X$ be a connected $\Thn$-space, and $I\subseteq \{1,\ldots,m-1\}$. Then the induced map
    \[ \colim_{I\subsetneq J\in \cP_m} \cG(X)(\sigma J)\to \cG(X)(\sigma I)\]
    is a monomorphism in $\Thnsset$.
\end{lemma}

\begin{proof}
    By \cref{lem:colimitascoeq}, we have an isomorphism in $\Thnsset$
    \[ \textstyle \colim_{I\subsetneq J\in \cP_m} \cG(X)(\sigma J)\cong \mathrm{coeq}\left(\coprod_{\substack{I\subsetneq J\in \cP_m \\
    |J|=|I|+2}} \cG(X)(\sigma J)\rightrightarrows \coprod_{\substack{I\subsetneq J\in \cP_m \\
    |J|=|I|+1}} \cG(X)(\sigma J) \right). \]
    Then the fact that the map
    \[ \textstyle\mathrm{coeq}\left(\coprod_{\substack{I\subsetneq J\in \cP_m \\
    |J|=|I|+2}} \cG(X)(\sigma J)\rightrightarrows \coprod_{\substack{I\subsetneq J\in \cP_m \\
    |J|=|I|+1}} \cG(X)(\sigma J) \right) \to \cG(X)(\sigma I) \]
    is a monomorphism in $\Thnsset$ follows from the following observations. First we have that, for $j\in \{1,\ldots,m-1\}\setminus I$, the map 
    \[ \cG(X)(\sigma (I\amalg \{j\}))\to \cG(X)(\sigma I) \]
    is a monomorphism in $\Thnsset$ as it is the image under $\prod_{(-)} X\colon \set^{\op}\to \Thnsset$ of the epimorphism $\cB(\sigma I)\to \cB(\sigma(I\amalg \{j\}))$ described in \cref{rmk:descrmaps}. Then, for $j_0,j_1\in \{1,\ldots,m-1\}\setminus I$, by \cref{lem:pullbacksofGX}, the intersection of the images of the monomorphisms 
    \[ \cG(X)(\sigma (I\amalg \{j_0\}))\hookrightarrow \cG(X)(\sigma I) \quad \text{and} \quad \cG(X)(\sigma (I\amalg \{j_1\}))\hookrightarrow \cG(X)(\sigma I) \]
    is precisely the image of the monomorphism $\cG(X)(\sigma (I\amalg \{j_0,j_1\}))\hookrightarrow \cG(X)(\sigma I)$. Hence they are identified in the coequalizer.
\end{proof}

\begin{thm} \label{GXprojcof}
    Let $m\geq 1$ and $X$ be a connected $\Thn$-space. Then we have that the functor $\cG(X)\colon \Cube^{\op}_m\to \MSThnsset$ is projectively cofibrant.
\end{thm}

\begin{proof}
    Let $\rho\colon F\to G$ be a trivial fibration in $(\MSThnsset)^{\Cube^{\op}_m}_\proj$, i.e., for all $(I,S)$ in $\Cube_m$, the map $\rho_{(I,S)}\colon F(I,S)\to G(I,S)$ is a trivial fibration in $\MSThnsset$. We show that there is a lift $\gamma$ in the below left diagram in $(\Thnsset)^{\Cube^{\op}_m}$, which is equivalent by \cref{cor:nattransfoutofGX} to showing that there is a lift in the below right diagram in $(\Thnsset)^{\cP_m^{\op}}$.
    \begin{tz}
    \node[](1) {$\cG(X)$}; 
    \node[right of=1,xshift=.4cm](2) {$G$}; 
    \node[above of=2](3) {$F$}; 

    \draw[->] (1) to node[below,la]{$\beta$} (2); 
    \draw[->] (3) to node[right,la]{$\rho$} (2);
    \draw[->,dashed] (1) to node[above,la,xshift=-4pt]{$\gamma$} (3);

    \node[right of=2,xshift=1.5cm](1) {$\cG(X)\circ\sigma^{\op}$}; 
    \node[right of=1,xshift=1.2cm](2) {$G\circ \sigma^{\op}$}; 
    \node[above of=2](3) {$F\circ \sigma^{\op}$}; 

    \draw[->] (1) to node[below,la]{$\beta\circ \sigma^{\op}$} (2); 
    \draw[->] (3) to node[right,la]{$\rho\circ \sigma^{\op}$} (2);
    \draw[->,dashed] (1) to node[above,la,xshift=-4pt]{$\gamma\circ \sigma^{\op}$} ($(3.south west)$);
    \end{tz}
   To this end, for $I\subseteq \{1,\ldots,m-1\}$, we construct the components $\gamma_{\sigma I}$ by reverse induction on $|I|\leq m-1$ in such a way that the below right diagram commutes and, for every $I\subsetneq J\in \cP_m$, the below left diagram commutes.
\begin{tz}
\node[](1) {$\cG(X)(\sigma J)$}; 
\node[below of=1](2) {$\cG(X)(\sigma I)$}; 
\node[right of=1,xshift=1.6cm](3) {$F(\sigma J)$}; 
\node[below of=3](4) {$F(\sigma I)$}; 

\draw[->] (1) to node[above,la]{$\gamma_{\sigma J}$} (3);
\draw[->] (1) to (2);
\draw[->] (3) to (4);
\draw[->] (2) to node[below,la]{$\gamma_{\sigma I}$} (4);

\node[right of=4,xshift=1.5cm](1) {$\cG(X)(\sigma I)$}; 
    \node[right of=1,xshift=1cm](2) {$G(\sigma I)$}; 
    \node[above of=2](3) {$F(\sigma I)$}; 

    \draw[->] (1) to node[below,la]{$\beta_{\sigma I}$} (2); 
    \draw[->] (3) to node[right,la]{$\rho_{\sigma I}$} (2);
    \draw[->] (1) to node[above,la,xshift=-4pt]{$\gamma_{\sigma I}$} (3);
\end{tz}
 If $|I|=m-1$, then $I=\{1,\ldots,m-1\}$. As there exists no $J\supsetneq I$ and $\cG(X)(\sigma I)= X$ is cofibrant in $\MSThnsset$, we get a lift $\gamma_{\sigma \{1,\ldots,m-1\}}$ satisfying the desired conditions. Now suppose that $|I|<m-1$ and assume that the components of $\gamma\circ \sigma^{\op}$ have already been constructed for all $J\subseteq \{1,\ldots,m-1\}$ with $|J|>|I|$ and satisfy the induction hypothesis. 
    Then there is an induced map $\colim_{I\subsetneq J\in \cP_m} \cG(X)(\sigma J)\to F(\sigma I)$ in the following diagram given by the universal property of colimit.  
    \begin{tz}
    \node[](1') {$\cG(X)(\sigma J)$};
    \node[below of=1'](0) {$\colim_{I\subsetneq J\in \cP_m} \cG(X)(\sigma J)$};
    \node[below of=0](1) {$\cG(X)(\sigma I)$}; 
    \node[right of=1,xshift=2cm](2) {$G(\sigma I)$}; 
    \node[above of=2](3) {$F(\sigma I)$}; 
    \node[above of=3](2') {$F(\sigma J)$};

    \draw[->] (1') to (0); 
    \draw[->] (1') to node[above,la]{$\gamma_{\sigma J}$} (2'); 
    \draw[->] (2') to (3);
    \draw[right hook->] (0) to (1);
    \draw[->,dashed] (0) to node[above,la]{$\exists !$} (3);
    \draw[->] (1) to node[below,la]{$\beta_{\sigma I}$} (2); 
    \draw[->] (3) to node[right,la]{$\rho_{\sigma I}$} (2);
    \draw[->,dashed] (1) to node[above,la]{$\gamma_{\sigma I}$} (3);
    \end{tz}
    As $\colim_{I\subsetneq J\in \cP_m} \cG(X)(\sigma J)\hookrightarrow \cG(X)(\sigma I)$ is a cofibration in $\MSThnsset$ by \cref{lem:colimcofib}, there is a lift in the above diagram. This builds the desired lift $\gamma\circ \sigma^{\op}$.
\end{proof}

As a consequence of \cite[Remark A.2.8.6]{htt}, the Quillen equivalence $\iota\dashv (\pi_*)_*$ from \cref{diagiotaQE} induces by post-composition the following Quillen equivalence.

 \begin{prop} \label{iotastarLQP}
 The adjunction
 \begin{tz}
\node[](1) {$(\MSThnsset)^{\Cube_m^{\op}}_\proj$}; 
\node[right of=1,xshift=3.6cm](2) {$(\MSThndiag)^{\Cube_m^{\op}}_\proj$}; 

\draw[->] ($(2.west)-(0,5pt)$) to node[below,la]{$((\pi_*)_*)_*$} ($(1.east)-(0,5pt)$);
\draw[->] ($(1.east)+(0,5pt)$) to node[above,la]{$\iota_*$} ($(2.west)+(0,5pt)$);

\node[la] at ($(1.east)!0.5!(2.west)$) {$\bot$};
\end{tz}
is a Quillen equivalence.
 \end{prop}

 The fact that $\iota_*$ is left Quillen together with \cref{GXprojcof} gives the following.

\begin{cor} \label{iotaGXprojcof}
    Let $m\geq 1$ and $X$ be a connected $\Thn$-space. Then we have that the functor $\iota\cG(X)\colon \Cube_m^{\op}\to \MSThndiag$ is projectively cofibrant.
\end{cor}

\subsection{Projective cofibrancy of \texorpdfstring{$\cH_m$}{Hm}} \label{Projective cofibrancy of H}

To prove that the functor $\cH_m$ is cofibrant in $(\MSspace)_\proj^{\Cube_m}$, we apply the following criterion; see the statement at \cite{GarnerMathoverflow}, there attributed to \cite{DuggerUniversal}.

\begin{thm} \label{thm:criterionprojcof}
Let $F\colon \cD\to \sset$ be a functor. For every $\defS\geq 0$, write $F_k$ for the composite
\[ F_k\colon \cD\xrightarrow{ F}\sset \xrightarrow{(-)_\defS} \set, \]
and suppose that the following conditions are satisfied: 
\begin{rome}[leftmargin=0.6cm]
\item the functor $F_\defS$ can be written as a coproduct of representables in $\set^{\cD}$
\[ \textstyle F_k\cong \coprod_{i\in I} \cD(d^\defS_i,-), \]
where $\{d^\defS_i\}_{i\in I}$ is a family of objects in $\cD$,  
\item the functor $F_\defS$ splits as a coproduct in $\set^{\cD}$
\[ F_\defS\cong N_\defS\amalg D_\defS \] where $N_\defS\colon \cD\to \set$ 
 (resp.~$D_\defS\colon \cD\to \set$) are functors such that, for every $d\in \cD$, the set $N_\defS(d)$ (resp.~$D_\defS(d)$) consists exactly of the non-degenerate (resp.~degenerate) $k$-simplices of~$F(d)$.
\end{rome}
Then $F$ is cofibrant in $(\MSspace)^{\cD}_\proj$.
\end{thm}

To apply \cref{thm:criterionprojcof} to $F=\cH_m$ and $\cD=\Cube_m$, we first want to verify Condition (i). 

\begin{notation}
Let $m\geq 1$ and $\defS\geq 0$. Denote by $0\colon [\defS]\to [1]$ (resp.~$1\colon [\defS]\to [1]$) the constant maps in $\Delta$ at $0$ (resp.~$1$). We write
\[ \Dnc([\defS],[1])\coloneqq \Delta([\defS],[1])\setminus\{0,1\} \]
for the subset of $\Delta([\defS],[1])$ consisting of the non-constant maps. With this notation, we denote by $(\cF_m)_\defS$ the presheaf in $\set^{\Cube_m}$ given by
\[ \textstyle (\cF_m)_\defS\coloneqq \coprod_{(I',S')\in \Cube_m} \coprod_{\prod_{I'} \Dnc([\defS],[1])} \Cube_m((I',S'),-)\]
\end{notation}

We now show that $(\cH_m)_\defS$ can be written as the coproduct of representables in $\set^{\Cube_m}$ given by~$(\cF_m)_\defS$.

\begin{prop} \label{prop:cond(i)}
    Let $m\geq 1$ and $\defS\geq 0$. Then there is an isomorphism in $\set^{\Cube_m}$
    \[ \textstyle (\cH_m)_\defS\cong (\cF_m)_\defS. \]
\end{prop}

To show this, we need to construct for each $(I,S)\in \Cube_m$ an isomorphism 
\[ \textstyle (\cH_m)_\defS(I,S)\cong (\cF_m)_\defS(I,S)=\coprod_{(I',S')\in \Cube_m} \coprod_{\prod_{I'} \Dnc([\defS],[1])} \Cube_m((I',S'),(I,S)) \]
that is natural in $(I,S)$.  

\begin{rmk} 
Recall that, for $(I,S)\in \Cube_m$, we have
\[ \textstyle (\cH_m)_\defS(I,S)=(\prod_{I} \repS[1])_\defS=\prod_{I} \Delta([\defS],[1]).  \]
Moreover, as $\Cube_m$ is a poset, the set $\Cube_m((I',S'),(I,S))$ is either a point or empty, for all $(I',S')\in \Cube_m$. In particular, we can identify an element of $(\cF_m)_\defS(I,S)$ with a tuple $((I',S'),(\tau_i)_{i\in I'})$ with $(I',S')$ an object of $\Cube_m$ such that $\Cube_m((I',S'),(I,S))=\{*\}$ and $(\tau_i)_{i\in I'}$ an element of $\prod_{I'} \Dnc([\defS],[1])$. Hence
\[  \textstyle (\cF_m)_\defS(I,S)\cong \left\{((I',S'), (\tau_i)_{i\in I'})\in \Cube_m\times  \prod_{I'} \Dnc([\defS],[1]) \mid I'\subseteq I,\; S=S'\cup I \right\}. \]
\end{rmk}

We construct a natural map $ \alpha_{(I,S)}\colon (\cH_m)_\defS(I,S)\to (\cF_m)_\defS(I,S)$ and an inverse $\beta_{(I,S)}$ of $\alpha_{(I,S)}$.

\begin{constr} \label{constr:alpha}
Let $m\geq 1$, $\defS\geq 0$, and $(I,S)\in\Cube_m$. Given a tuple $(\sigma_i)_{i\in I}\in \prod_{I} \Delta([\defS],[1])$, we define $(I^\sigma, S^\sigma)$ to be the object of $\Cube_m$ given by 
\[ I^\sigma\coloneqq \{i\in I\mid \sigma_i\in \Dnc([\defS],[1]) \} \quad \text{and} \quad S^\sigma\coloneqq S\setminus \{i\in I\mid \sigma_i =0 \}. \]
Observe that $I^\sigma\subseteq S^\sigma$, $I^\sigma\subseteq I$, and $S=S^\sigma\cup I$. We further set $(\tau^\sigma_i)_{i\in I^\sigma}$ to be the tuple in $\prod_{I^\sigma} \Dnc([\defS],[1])$ given by $\tau^\sigma_i\coloneqq \sigma_i$ for all $i\in I^\sigma$; note that this is well-defined by definition of $I^\sigma$. We then define $\alpha_{(I,S)}$ to be the map
\[ \alpha_{(I,S)}\colon (\cH_m)_\defS(I,S)\to (\cF_m)_\defS(I,S), \quad (\sigma_i)_{i\in I} \mapsto ((I^\sigma,S^\sigma), (\tau^\sigma_i)_{i\in I^\sigma}). \]
These assignments assemble into a natural transformation $\alpha\colon (\cH_m)_\defS\to (\cF_m)_\defS$.
\end{constr}

\begin{constr} \label{constr:beta}
Let $m\geq 1$ and $\defS\geq 0$, and consider an object $(I,S)\in \Cube_m$. Given a tuple $((I',S'),(\tau_i)_{i\in I'})\in (\cF_m)_k(I,S)$, we define $(\sigma^\tau_i)_{i\in I}\in \prod_I \Delta([\defS],[1])$ to be given at $i\in I$ by 
\[ \sigma^\tau_i\coloneqq \begin{cases} 
    \tau_i & \text{if} \; i \in I' \\
    1 & \text{if} \; i \in I\cap (S'\setminus I') \\
    0 & \text{if} \; i \in (I\cup S')\setminus S'.
    \end{cases} \]
We then define $\beta_{(I,S)}$ to be the map
\[ \beta_{(I,S)}\colon (\cF_m)_\defS(I,S)\to (\cH_m)_\defS(I,S), \quad ((I',S'),(\tau_i)_{i\in I'})\mapsto (\sigma^\tau_i)_{i\in I}. \]
\end{constr}

\begin{proof}[Proof of \cref{prop:cond(i)}]
    A direct computation shows that, for all $(I,S)\in \Cube_m$, the maps $\alpha_{(I,S)}$ and $\beta_{(I,S)}$ are inverse to each other, and so the natural transformation $\alpha\colon (\cH_m)_\defS\xrightarrow{\cong} (\cF_m)_\defS$ provides the desired natural isomorphism.
\end{proof}

We now prove Condition (ii) of \cref{thm:criterionprojcof}. For this, we first study the non-degenerate simplices of $\prod_I \repS[1]$. 

\begin{lemma} \label{nondegsimplicesofprodD[1]}
    Let $\defS\geq 0$ and $I$ be a finite set. A $\defS$-simplex in the product $\prod_I \repS[1]$, i.e., a tuple $(\sigma_i)_{i\in I}\in \prod_I \Delta([\defS],[1])$, is non-degenerate if and only if $\Dnc([\defS],[1])\subseteq \{\sigma_i\mid i\in I\}$. 
\end{lemma}

\begin{proof}
    A non-constant map $\sigma\colon [\defS]\to [1]$ is uniquely determined by an integer $0\leq \ell<k$ such that $\sigma(i)=0$ for $0\leq i\leq \ell$ and $\sigma(i)=1$ for $\ell+1\leq i\leq k$. In other words, it is uniquely determined by an integer $0\leq \ell<k$ such that $\sigma(\ell) \neq \sigma(\ell+1)$. We denote the map associated to $0\leq \ell<k$ by $\rho_\ell\colon [\defS]\to [1]$, and so we have $\Dnc([\defS],[1])=\{\rho_\ell\mid 0\leq \ell <k\}$.

    Now, by definition, a $\defS$-simplex in $\prod_I \repS[1]$, i.e., a tuple $(\sigma_i)_{i\in I}\in \prod_I \Delta([\defS],[1])$, is degenerate if and only if there is $0\leq \ell < \defS$ and  $(\sigma'_i)_{i\in I} \in \prod_I \Delta([\defS-1],[1])$ with $s_\ell\sigma'_i = \sigma_i$, i.e., $\sigma_i(\ell) = \sigma_i(\ell+1)$ for all $i \in I$. Hence, by negation, we get that $(\sigma_i)_{i\in I}\in \prod_I \Delta([\defS],[1])$ is non-degenerate if and only if, for all $0\leq \ell<\defS$, there exists an $i \in I$ such that $\sigma_i(\ell) \neq \sigma_i(\ell+1)$. By the above arguments, this is equivalent to saying that, for all $0\leq \ell<\defS$, there exists an $i \in I$ such that $\sigma_i=\rho_\ell$. Hence, this proves that $\Dnc([\defS],[1])=\{\rho_\ell\mid 0\leq \ell <k\}\subseteq \{\sigma_i\mid i\in I\}$.
\end{proof}

\begin{notation}
    Let $m\geq 1$ and $\defS\geq 0$. For $(I',S')\in \Cube_m$, we define subsets of $\prod_{I'} \Dnc([\defS],[1])$
    \begin{align*}
         N_k(I') &\coloneqq \textstyle \left\{ (\tau_i)_{i\in I'}\in \prod_{I'} \Dnc([\defS],[1])\mid \Dnc([\defS],[1])\subseteq \{\tau_i\mid i\in I'\} \right\}, \\
         D_k(I') &\textstyle\coloneqq (\prod_{I'} \Dnc([\defS],[1]))\setminus N_k(I'). 
        \end{align*}  
        With these notations, we denote by $(\cN_m)_\defS$ and $(\cD_m)_\defS$ the sub-presheaves of $(\cF_m)_\defS$ in $\set^{\Cube_m}$ given by
        \begin{align*}
          (\cN_m)_\defS &\coloneqq \textstyle\coprod_{(I',S')\in \Cube_m} \coprod_{N_k(I')} \Cube_m((I',S'),-), \\
          (\cD_m)_\defS & \coloneqq \textstyle \coprod_{(I',S')\in \Cube_m} \coprod_{D_k(I')} \Cube_m((I',S'),-). 
     \end{align*}
     We also write $(\cH_m)_\defS(I,S)^{\mathrm{nd}}$ (resp.~$(\cH_m)_\defS(I,S)^{\mathrm{deg}}$ for the subsets of non-degenerate (resp.~degenerate) $k$-simplices of $\cH_m(I,S)=\prod_I \repS[1]$.
\end{notation}

We show that $(\cH_m)_\defS$ splits as non-degenerate and degenerate simplices as follows.

\begin{prop} \label{prop:cond(ii)}
    Let $m\geq 1$ and $\defS\geq 0$. Then there is an isomorphism in $\set^{\Cube_m}$
    \[ \textstyle (\cH_m)_\defS\cong (\cN_m)_\defS\amalg (\cD_m)_\defS, \]
    and, at every object $(I,S)$ in $\Cube_m$, 
    it restricts to isomorphisms 
    \[ (\cH_m)_\defS(I,S)^{\mathrm{nd}}\cong (\cN_m)_\defS(I,S) \quad \text{and} \quad (\cH_m)_\defS(I,S)^{\mathrm{deg}}\cong (\cD_m)_\defS(I,S). \]
\end{prop}

\begin{rmk} \label{rem:HisotoNUD}
Using \cref{prop:cond(i)} and the fact that, by definition, for every $(I',S')\in \Cube_m$, we have $\prod_{I'} \Dnc([\defS],[1]) = N_\defS(I')\amalg D_\defS(I')$, there are isomorphisms in $\set^{\Cube_m}$
    \[ (\cH_m)_\defS\cong (\cF_m)_\defS\cong (\cN_m)_\defS\amalg (\cD_m)_\defS.  \]
    Recall that the first natural isomorphism has component at an object $(I,S)\in \Cube_m$ the map $\alpha_{(I,S)}\colon (\cH_m)_\defS(I,S)\to (\cF_m)_\defS(I,S)$ from \cref{constr:alpha} with inverse $\beta_{(I,S)}$ from \cref{constr:beta}, and note that the second isomorphism is just a re-ordering of the coproduct.
\end{rmk}

\begin{lemma} \label{lem:restrofalpha}
    Let $m\geq 1$ and $\defS\geq 0$. Given an object $(I,S)$ in $\Cube_m$, the inverse assignments 
\[ \alpha_{(I,S)}\colon (\cH_m)_\defS(I,S)\to (\cF_m)_\defS(I,S) \quad \text{and} \quad \beta_{(I,S)}\colon (\cF_m)_\defS(I,S)\to (\cH_m)_\defS(I,S)\]
restrict to assignments  
\[ \alpha_{(I,S)}\colon (\cH_m)_\defS(I,S)^{\mathrm{nd}}\to (\cN_m)_\defS(I,S) \quad \text{and} \quad \beta_{(I,S)}\colon (\cN_m)_\defS(I,S)\to (\cH_m)_\defS(I,S)^{\mathrm{nd}}. \]
\end{lemma}

\begin{proof}
    This is straightforward from the definition of $(\cN_m)_\defS$ and the characterization of non-degenerate $\defS$-simplices of $\prod_I \repS[1]$ from \cref{nondegsimplicesofprodD[1]}.
\end{proof}

\begin{proof}[Proof of \cref{prop:cond(ii)}]
    By \cref{rem:HisotoNUD}, we have an isomorphism $(\cH_m)_\defS\cong (\cN_m)_\defS\amalg (\cD_m)_\defS$, which, at every object $(I,S)\in \Cube_m$, restricts by \cref{lem:restrofalpha} to an isomorphism 
    \[(\cH_m)_\defS(I,S)^{\mathrm{nd}}\cong (\cN_m)_\defS(I,S). \]
    Hence, it also restricts at every object $(I,S)\in \Cube_m$ to an isomorphism 
    \[(\cH_m)_\defS(I,S)^{\mathrm{deg}}\cong (\cD_m)_\defS(I,S). \]
    This shows the desired result.
\end{proof}

Finally, by \cref{prop:cond(i),prop:cond(ii)}, the functor $\cH_m$ satisfies the condition of \cref{thm:criterionprojcof}, and so we get the following.

\begin{thm} \label{Hmprojcof}
    Let $m\geq 1$. The functor $\cH_m\colon \Cube_m\to \MSspace$ is projectively cofibrant.
\end{thm}

\bibliographystyle{amsalpha}
\bibliography{ref2}

\end{document}